\documentclass[a4paper,11pt,leqno]{amsart}
\usepackage{a4wide}
\usepackage{amssymb,amsthm,amsmath}
\usepackage{lmodern}
\usepackage{url}

\newtheorem{theorem}{Theorem}[section]
\newtheorem{corollary}[theorem]{Corollary}
\newtheorem{lemma}[theorem]{Lemma}
\newtheorem{proposition}[theorem]{Proposition}

\newtheorem{question}[theorem]{Question}

\newtheorem{proposition-definition}[theorem]{Proposition-Definition}
\theoremstyle{definition}
\newtheorem{definition}[theorem]{Definition}
\theoremstyle{remark}
\newtheorem{remark}[theorem]{Remark}

\newtheorem{example}[theorem]{Example}

\DeclareMathOperator{\Ind}{Ind}
\DeclareMathOperator{\Ord}{Ord}
\DeclareMathOperator{\Mod}{Mod}
\DeclareMathOperator{\ind}{c-Ind}
\DeclareMathOperator{\Hom}{Hom}
\DeclareMathOperator{\End}{End}
\DeclareMathOperator{\val}{val}
\DeclareMathOperator{\Ker}{Ker}
\DeclareMathOperator{\charf}{char}
\DeclareMathOperator{\diag}{diag}
\DeclareMathOperator{\Res}{Res}
\DeclareMathOperator{\res}{res}
\DeclareMathOperator{\Ima}{Im}
\newcommand*{\St}{\mathrm{St}}
\newcommand*{\charone}{\mathbf{1}}
\newcommand*{\Id}{\mathrm{Id}}
\newcommand*{\is}{\mathrm{is}}

\author{N. Abe}
\address[N. Abe]{Department of Mathematics, Hokkaido University, Kita 10, Nishi 8, Kita-Ku, Sapporo, Hokkaido, 060-0810, Japan}
\thanks{The first-named author was supported by JSPS KAKENHI Grant Number 26707001.}
\email{abenori@math.sci.hokudai.ac.jp}

\author{G. Henniart} 
\address[G. Henniart]{Universit\'e de Paris-Sud, Laboratoire de Math\'ematiques d'Orsay, Orsay cedex F-91405 France;
CNRS, Orsay cedex F-91405 France}
\email{Guy.Henniart@math.u-psud.fr}

\author{M.-F. Vign\'eras}
\address[M.-F. Vign\'eras]{Institut de Math\'ematiques de Jussieu, 175 rue du Chevaleret, Paris 75013 France}
\email{vigneras@math.jussieu.fr}

\title{Modulo $p$ representations of reductive $p$-adic groups: functorial properties}
\date{}

\begin{document}
\begin{abstract}
Let $F$ be a  local field with residue characteristic $p$, let $C$ be an algebraically closed field of characteristic $p$, and let $\mathbf G$ be a connected reductive $F$-group.  In a previous paper, Florian Herzig and  the authors  classified irreducible admissible $C$-representations of $G=\mathbf G(F)$ in terms of supercuspidal representations of Levi subgroups of $G$. Here, for a parabolic subgroup $P$ of $G$  with Levi subgroup $M$ and  an irreducible admissible $C$-representation $\tau$ of   $M$, we determine the lattice of subrepresentations of   $\Ind_P^G \tau$  and we show that $\Ind_P^G \chi \tau$ is  irreducible for a general unramified character $\chi$ of $M$. In the reverse direction, we compute the image by the two adjoints of $\Ind_P^G$ of an irreducible admissible representation $\pi$ of  $G$.
On the way, we prove that the right adjoint of $\Ind_P^G $ respects admissibility, hence coincides with Emerton's ordinary part functor $\Ord_{\overline P}^G$ on admissible representations. 
\end{abstract}
\maketitle

\setcounter{tocdepth}{2}  
\renewcommand{\labelenumi}{(\roman{enumi})}
\tableofcontents
\section{Introduction} 
\subsection{Classification results of \cite{MR3600042}}\label{subsec:from_AHHV}
The present paper is a sequel to \cite{MR3600042}. The overall setting is the same: $p$ is a prime number, $F$ a local field with finite residue field of characteristic $p$, $\mathbf G$  a connected reductive $F$-group and  $G=\mathbf G(F)$ is seen as a topological locally pro-$p$ group. We fix an algebraically closed field $C$ of characteristic $p$ and we study the smooth representations of $G$ over $C$-vector spaces - we write   $\Mod_C^\infty(G)$ for the category they form.

Let $P$ be a parabolic subgroup of $G$ with a Levi decomposition $P=MN$ and $\sigma$ a supercuspidal  $C$-representation of $M$, in the sense that  it is  irreducible, admissible, and does not appear as a subquotient of a representation of $M$ obtained by parabolic induction from  an irreducible, admissible $C$-representation of a proper Levi sugroup of $M$. Then there is a maximal parabolic subgroup $P(\sigma)$ of $G$ containing $P$ to which $\sigma$ inflated to $P$ extends; we write $e(\sigma)$ for that extension. For each parabolic subgroup $Q$ of $G$ with $P\subset Q \subset P(\sigma)$, we form
$$I_G(P,\sigma,Q)=\Ind_{P(\sigma)}^G(e(\sigma)\otimes \St_Q^{P(\sigma)})$$
where $\St_Q^{P(\sigma)}= \Ind_Q^{P(\sigma)}1/ \sum\Ind_{Q'}^{P(\sigma)}1$, the sum being over parabolic subgroups $Q'$ of $G$ with $Q\subsetneq Q'\subset P(\sigma)$. 

The classification result of   \cite{MR3600042}   is that  $I_G(P,\sigma,Q)$ is irreducible  admissible, and that conversely any irreducible admissible $C$-representation of $G$ has the form $I_G(P,\sigma,Q)$, where $P$ is determined up to conjugation, and, once $P$ is fixed, $Q$ is determined and so is the isomorphism class of $\sigma$.

\subsection{Main results}\label{S:1.2}
The classification raises natural questions: if $G$ is a Levi subgroup of a parabolic subgroup $R$ in a larger  connected reductive group $H$, what is the structure of $\Ind_R^H \pi$ when $\pi $ is a irreducible admissible $C$-representation of $G$?

We show that $\Ind_R^H \pi$ has finite length and multiplicity $1$; we determine its irreducible constituents and the lattice of its subrepresentations: see section \ref{sec:3} for precise results and proofs. As an application, we answer a question of Jean-Francois Dat, in showing that  $\Ind_R^H \chi \pi$ is irreducible when $\chi$ is a general unramified character of $G$.

If 
$P_1$ is a parabolic subgroup of $G$ with Levi decomposition $P_1=M_1 N_1$, then $\Ind_{P_1}^G:\Mod^\infty_C(M_1) \to \Mod^\infty_C(G)$ has a left adjoint $L_{P_1}^G$, which is the usual Jacquet functor  $(-)_{N_1}$ taking $N_1$-coinvariants, and also a right adjoint functor $R^G_{P_1}$ \cite{Vigneras-adjoint}. It is   natural  to apply  $L_{P_1}^G $ and $R_{P_1}^G $ to $\pi $. They turn out to be irreducible or $0$, in sharp contrast to the case of complex representations of $G$. To state precise results, we fix a minimal parabolic subgroup $B$ of $G$ and a Levi decomposition $B=ZU$ of $B$, and we consider only parabolic subgroups containing $B$ and their Levi components containing $Z$. We simply say ``let $P=MN$ be a standard parabolic subgroup of $G$''  to mean that $P$ contains $B$ and $M$ is the Levi component of $P$ containing $Z$, $N$ being the unipotent radical of $P$.  

\begin{theorem}\label{thm:1.1}
Let $P=MN$  and $P_1=M_1 N_1$ be  standard parabolic subgroups of $G$, let $\sigma$ be a supercuspidal $C$-representation of $M$ and let $Q$  be a parabolic subgroup of $G$ with $P\subset Q\subset P(\sigma)$.
\begin{enumerate}
\item $L_{P_1}^G I_G(P,\sigma, Q)$ is isomorphic to $I_{M_1}(P\cap M_1, \sigma, Q\cap M_1)$ if  $ P_1\supset P$ and the group   generated by $P_1\cup Q$ contains  $P(\sigma)  $, and is $0$ otherwise.
\item $R_{P_1}^G I_G(P,\sigma, Q)$ is isomorphic to $I_{M_1}(P\cap M_1, \sigma, Q\cap M_1)$ if $ P_1\supset Q$, and is $0$ otherwise.
\end{enumerate}\end{theorem}

See \S \ref{S:6} and \S \ref{S:7} for the proofs, with consequences already drawn in \S \ref{S:6.1}: in particular, we prove that an irreducible admissible $C$-representation $\pi$ of $G$ is supercuspidal exactly when $L_P^G\pi$ and $R_P^G\pi$ are $0$ for any proper parabolic subgroup $P$ of $G$.

As the construction of $I_G(P,\sigma, Q)$ involves   parabolic induction, we are naturally led to investigate, as an intermediate step, the composite functors $L_{P_1}^G \Ind_P^G$ and $R_{P_1}^G \Ind_P^G$, for standard parabolic subgroups  $P=MN$  and $P_1=M_1 N_1$ of $G$. In \S \ref{S:5}, we prove:

\begin{theorem}\label{thm:1.2}   The functor $L_{P_1}^G \Ind_P^G:\Mod_C(M)\to \Mod_C(M_1)$ is isomorphic to the functor $\Ind_{P\cap M_1}^{M_1} L_{P_1\cap M}^M$, and the functor $R_{P_1}^G \Ind_P^G:\Mod_C(M)\to \Mod_C(M_1)$ is isomorphic to the functor $\Ind_{P\cap M_1}^{M_1} R_{P_1\cap M}^M$.

 \end{theorem}
 We actually describe explicitly the functorial isomorphism for  $L_{P_1}^G \Ind_P^G$ whereas the case of $R_{P_1}^G \Ind_P^G$ is obtained by adjunction properties.
The fact that $R_{P_1}^G$ has no direct explicit description has consequence for the proof of Theorem~\ref{thm:1.1} (ii).
We first prove:
\begin{theorem}\label{thm:1.4}
If $\pi$ is an admissible $C$-representation of $G$,   then $R_P^G \pi$ is an admissible $C$-representation of $M$.
\end{theorem}
To prove Theorem~\ref{thm:1.1} (ii) we in fact use $\Ord_{\overline{P}_1}^G$ in place of $R_{P_1}^G$.
It follows that on admissible $C$-representations of $G$, $R_P^G$  coincides with Emerton's ordinary part functor $\Ord_{\overline P}^G$ (as extended to the case of $C$-representations in \cite{Vigneras-adjoint}). Note that,  if the characteristic of $F$ is $0$ and  $\pi$ is an admissible  $C$-representation of $G$, then $L_P^G\pi$ is admissible. But in contrast, when $F$ has characteristic $p$, we produce  in \S \ref{S:4} an example, for $G=GL(2,F)$, of an admissible $C$-representation $\pi$ of $G$ such that $L_B^G\pi$ is not admissible.

\subsection{Outline of the proof}
After the initial section \S \ref{S:2} devoted to notation and preliminaries, our paper mainly follows the layout above. 
However admissibility questions are explored in \S\ref{S:4}, where Theorem~\ref{thm:1.4} is established: as mentioned above, the result is used in the proof Theorem~\ref{thm:1.1} (ii).

Without striving for the utmost generality, we have taken care not to use unnecessary assumptions. In particular, from section \S \ref{S:4} on, we consider a general commutative ring $R$ as coefficient  ring, imposing conditions on $R$ only when useful. The reason  is that for arithmetic applications it is important to consider the case where $R$ is artinian and $p$-nilpotent  or invertible in $R$. 
Only when we use the classification do we assume $R=C$.
Our results are valid for $R$ noetherian and $p$ nilpotent in $R$ in sections \S \ref{S:4} to  \S \ref{S:7}.
For example, when $R$ is noetherian and $p$ is nilpotent in $R$, Theorem \ref{thm:1.2} is valid  (Theorem \ref{thm:5.4} and Corollary \ref{cor:5.5}) and a version to Theorem \ref{thm:1.1} is obtained in Theorem \ref{thm:6.1} and Corollary \ref{cor:6.2}. Likewise Theorem \ref{thm:1.4} is valid when $R$ is noetherian and  $p$ is nilpotent in $R$ (Theorem \ref{thm:4.11}).

In a companion paper, the authors will investigate the effect of taking invariants under a pro-$p$  Iwahori subgroup in the modules $I_G(P,\sigma,Q)$ of \ref{subsec:from_AHHV}.

 \section{Notation, useful facts and preliminaries}\label{S:2}
 \subsection{The group $G$ and its standard parabolic subgroups $P=MN$}\label{S:2.1}
 In all that follows, $p$ is a prime number, $F$ is a local field with finite residue field $k$ of characteristic $p$; as usual, we write $O_F$ for the ring of integers of $F$, $P_F$ for its maximal ideal and $v_F$ the absolute value of $F$ normalised by $v_F(F^*)=\mathbb Z$. We denote an algebraic group over $F$ by a bold letter, like $\mathbf H$, and use the same ordinary letter for the group of $F$-points, $H=\mathbf H(F)$.
  We fix a connected reductive $F$-group $\mathbf G$. We fix a maximal $F$-split subtorus $\mathbf T$ and write $\mathbf Z$ for its $\mathbf G$-centralizer; we also fix a minimal parabolic subgroup $\mathbf B$ of $\mathbf G$  with Levi component $\mathbf Z$, so that $\mathbf B=\mathbf  Z  \mathbf U$ where $ \mathbf U$ is the unipotent radical of  $\mathbf B$. Let $X^*( \mathbf T)$ be the group of $F$-rational characters of  $\mathbf T$ and $\Phi$ the subset of roots of  $\mathbf T$ in the Lie algebra of  $\mathbf G$. Then $ \mathbf B$ determines a subset $\Phi^+$ of positive roots - the roots of $\mathbf T$ in the Lie algebra of $\mathbf U$-  and a subset of simple roots $\Delta$. The $\mathbf G$-normalizer  $\mathbf N_{\mathbf G}$ of $\mathbf T$ acts on $X^*( \mathbf T)$ and through that action,  $\mathbf N_{\mathbf G}/ \mathbf Z$ identifies with the Weyl group of the root system $\Phi$.
Set $\mathcal N:=\mathbf N_{\mathbf G}(F)$ and note that  $\mathbf N_{\mathbf G}/ \mathbf Z \simeq \mathcal N/Z$; we write $\mathbb W $ for  $\mathcal N/Z$.

A standard parabolic subgroup  of $\mathbf G$ is a parabolic $F$-subgroup containing $\mathbf B$. Such a  parabolic subgroup $\mathbf P$  has a unique Levi subgroup $\mathbf M$ containing $\mathbf Z$, so that $\mathbf P=\mathbf M\mathbf N$ where $\mathbf N$ is the unipotent radical of $\mathbf P$ -  we also call  $\mathbf M$ standard.  By a common abuse of language to describe the preceding situation, we simply say ``let $P=MN$ be a standard parabolic subgroup of $G$''; we sometimes write $N_P$ for $N$ and $M_P$  for $M$. The parabolic subgroup of $G$ opposite to $P$ will be written $\overline P$ and its unipotent radical $\overline N$, so that $\overline P=M\overline N$, but beware that $\overline P$ is not standard !  We write  $\mathbb W_M$ for the Weyl group $M\cap\mathcal N/Z$.

If $\mathbf P=\mathbf M\mathbf N$ is a standard parabolic subgroup of $\mathbf{G}$, then $\mathbf M\cap \mathbf B$ is a minimal parabolic subgroup of $\mathbf M$. If $\Phi_M$ denotes the set of roots of  $\mathbf T$ in the Lie algebra of $\mathbf M$, with respect to $\mathbf M\cap \mathbf B$  we have   $\Phi_M^+=\Phi_M\cap \Phi^+ $ and $\Delta_M=\Phi_M\cap \Delta$.  We also write  $\Delta_P$ for  $\Delta_M$ as $P$ and $M$ determine each other, $P=MU$. Thus we obtain a bijection $P\mapsto \Delta_P$ from standard parabolic subgroups of $G$ to subsets of $\Delta$, with $B$ corresponds to $\Phi$ and $G$ to $\Delta$. If $I$ is a subset of $\Delta$, we sometimes denote by $P_I=M_IN_I$ the corresponding standard parabolic subgroup of $G$. If $I=\{\alpha\}$ is a singleton, we write $P_\alpha=M_\alpha  N_\alpha$. We note a few useful properties.  If $P_1$ is another standard parabolic subgroup of $G$, then $P\subset P_1$ if and only if $\Delta_P\subset \Delta_{P_1}$; we have $\Delta_{P\cap P_1}= \Delta_P\cap \Delta_{P_1}$ and the parabolic subgroup corresponding to $\Delta_P\cup \Delta_{P_1}$ is  the subgroup $\langle P,P_1\rangle$  of $G$ generated by $ P $ and $P_1$. The standard parabolic subgroup of $M$ associated to $\Delta_M \cap \Delta_{M_1}$ is $M\cap P_1=(M\cap M_1)(M\cap N_1)$ \cite[Proposition 2.8.9]{MR794307}.
It is convenient to write $G'$ for the subgroup of $G$ generated by the unipotent radicals of the parabolic subgroups; it is also the normal subgroup of $G$ generated by $U$, and we have $G=ZG'$.
%

For each $\alpha\in  X^*(T)$, the homomorphism $x\mapsto v_F(\alpha(x)):T\to\mathbb Z$ extends uniquely to a homomorphism $Z\to\mathbb Q$ that we denote in the same way.  This defines a homomorphism $Z\xrightarrow{v} X_*(T)\otimes \mathbb  Q$ such that $\alpha(v(z))= v_F(\alpha(z))$ for $z\in Z, \alpha \in X^*(T)$. 

\bigskip An interesting situation occurs when $\Delta=I\sqcup J$ is the union of two orthogonal subsets $I$ and $J$. In that case, $G'=M'_IM'_J$, $M'_I$ and $M'_J$ commute with each other, and their intersection is finite and central in $G$  \cite[II.7 Remark 4]{MR3600042}.

  \subsection{Representations of $G$}\label{S:2.2}
 As apparent in the abstract and the introduction, our main interest lies in smooth $C$-representations of $G$, where $C$ is an {\bf algebraically closed field of characteristic $p$}, which we fix throughout. However many of our arguments do not necessitate so strong a hypothesis on coefficients, so we let $R$ be a {\bf fixed commutative ring}. 
 
 Occasionally we shall consider an $R[A]$-module $V$ where $A$ is a monoid. An element $v$ of $V$
 is called {\bf $A$-finite} if its translates under $A$ generate a finitely generated submodule of $V$. If $R$ is noetherian
 the $A$-finite elements in $V$ generate a submodule of $V$, that we write $V^{A-f}$. When $A$ is generated by an element $t$, we write $V^{t-f}$ instead of $V^{A-f}$.
 
We speak indifferently of $R[H]$-modules and of $R$-representations of $H$ for a locally profinite group $H$.
An $R[H]$-module $V$ is called {\bf smooth} if every vector in $V$ has an open stabilizer in $H$. The smooth $R$-representations of $H$ and $R[H]$-linear maps form an abelian category $\Mod_R^\infty(H)$.

  An  $R$-representation $V$ of  a locally profinite group $H$ is  {\bf admissible} if it is smooth and for any open compact subgroup $J$ of $H$, the $R$-submodule $V^J$ of $J$-fixed vectors is finitely generated. When  $R$ is noetherian, it is clear that it suffices to check this when $J$ is small enough. When $R$ is noetherian we write $\Mod_R^a(H)$ for the subcategory of $\Mod_R^\infty(H)$ made out of the admissible $R$-representations of $H$. We explore admissibility further in section \ref{S:4}.
 
If $P=MN$ is a standard parabolic subgroup of $G$,  the parabolic induction functor $\Ind_P^G: \Mod_R^\infty(M)\to \Mod_R^\infty(G)$ sends $W\in  \Mod_R^\infty(M)$ to  the smooth $R[G]$-module 
$\Ind_P^GW$ made out of functions $f:G\to W$ satisfying $f(mngk)=mf(g)$ for $m\in M, n\in N, g\in G$ and $k$ in some open subgroup  $K_f$ of $G$ - the action of $G$ is via right translation. 
The functor  $\Ind_P^G$ has a left adjoint $L_P^G:\Mod_R^\infty(G)\to \Mod_R^\infty(M)$ which sends $V$ in $\Mod_R^\infty(G)$ to the module of $N$-coinvariants $V_N$  of $V$, which is naturally a smooth $R[M]$-module. The functor  $\Ind_P^G$ has a right adjoint $R_P^G:\Mod_R^\infty(G)\to \Mod_R^\infty(M)$ \cite[Proposition 4.2]{Vigneras-adjoint}.

When $R$ is a {\bf field}, a smooth $R$-representation of $G$ is called {\bf irreducible} if it is a simple $R[G]$-module. An $R$-representation of $G$ is called {\bf supercuspidal}  it is  irreducible, admissible, and does not appear as a subquotient of a representation of $M$ obtained by parabolic induction from  an irreducible, admissible representation of a proper Levi subgroup of $M$.

\subsection{On compact induction} 
If $X$ is a locally profinite space with a countable basis of open sets,  and $V$ is an $R$-module, we write $ C_c^\infty(X, V)$ for the space of compactly supported   locally constant functions $X\to V$. 
One verifies that the natural map $ C_c^\infty(X, R)\otimes_RV\to C_c^\infty(X, V)$ is an isomorphism. 
\begin{lemma}The $R$-module $C_c^\infty(X, R)$ is free. When $X$ is compact, the submodule of constant functions is a direct factor of $C_c^\infty(X, R)$.
\end{lemma}
\begin{proof} The proof of  \cite[Appendix A.1]{MR3402357} when $X$ is compact is easily adapted to $ C_c^\infty(X, V)$ when $X$ is not compact.   
\end{proof}

\begin{example}\label{ex:fd}  $C_c^\infty(X, R)^H$ is a direct factor of $C_c^\infty(X, R)$ when $X$ is compact with a continuous  action of a profinite group $H$ with finitely many orbits (apply the lemma to the orbits which are open). 
\end{example}

Let $H$ be a locally profinite group and  $J$ a closed subgroup of $H$.

\begin{lemma} The quotient map $H\to J\backslash H$ has a continuous section.
\end{lemma}
\begin{proof} When $H$ is profinite, this is \cite[Proposition 2.2.2]{MR2599132}. In general, let $K$ be a compact open subgroup of $H$. Cover $H$ with disjoint double cosets $JgK$. It is enough to find, for any given $g$, a continuous section of the induced map $JgK\xrightarrow{\pi_g} J\backslash JgK $. The map $k\mapsto gk$ induces a continous bijective map $(K\cap g^{-1}Jg)\backslash K \xrightarrow{p} J\backslash JgK$. Because $J$ is closed in $H$, both spaces are Hausdorff and $(K\cap g^{-1}Jg)\backslash K$ is compact since $K$ is, so $p$ is a homeomorphism. If $\sigma$ is a continuous section of the quotient map $K \to (K\cap g^{-1}Jg)\backslash K$ then $x\mapsto g \sigma (p^{-1}(x))$ gives the desired section of $\pi_g$.
\end{proof}

 Let $\sigma$ be a continuous section of $H\to J\backslash H$, and let   $V$ be a smooth $R$-representation of $J$. Recall that $\ind_J^HV$ is the space of functions $f:H\to V$, left invariant by $J$, of compact support in $J\backslash H$, and smooth for $H$ acting by right translation. Immediately:
 
\begin{lemma} \label{lemma:usef}The  map $f\mapsto  f\circ \sigma :\ind_J^HV\to C_c^{\infty}(J\backslash H, V)$ is an $R$-module isomorphism.
\end{lemma} 
As  a consequence we get a useful induction/restriction property: let $W$ be a smooth $R$-representation   of $H$.
 \begin{lemma}\label{lemma:useful} The  map 
  $f\otimes w \mapsto (h\mapsto f(h)\otimes hw) :(\ind_J^H V)\otimes  W  \to \ind_J^H(V\otimes W)  $ 
 is an $R[H]$-isomorphism.
 \end{lemma}
  
\begin{proof} The map is linear and $H$-equivariant. Lemma \ref{lemma:usef} implies that it is bijective.
\end{proof}

\begin{remark} Arens' theorem says that if  $X$ is a homogeneous space for $H$ and $H/K$ is countable for a compact open subgroup $K$ of $H$,  then for $x\in X$ the orbit map $h\mapsto hx$ induces a homeomorphism $H/H_x\simeq X$. In particular, for two closed subgroups $I,J$ of $H$ such that $H=IJ$, we get a homeomorphism $I/(I\cap J)\simeq H/J$. Hence $(\ind_J^HV)|_I\simeq \ind_{I\cap J}^IV$ for any smooth $R$-representation $V$ of $J$.
\end{remark}
 
\subsection{$I_G(P,\sigma, Q)$ and minimality}\label{S:2.22} We recall from \cite{MR3600042} the construction of $I_G(P,\sigma,Q)$, our main object of study.  

\begin{proposition}\label{prop:2.1}Let $P=MN \subset Q$ be two standard parabolic subgroups of $G$ and $\sigma$ an $R$-representation of $M$.  Then the following  are equivalent:
\begin{enumerate}
\item $\sigma$ extends to a representation of $Q$ where $N$ acts trivially.
\item For each $\alpha \in \Delta_Q\setminus \Delta_P$, $Z\cap M'_\alpha$ acts trivially on $\sigma$. 
\end{enumerate}\end{proposition}
That comes from \cite[II.7 Proposition]{MR3600042} when $R=C$, but  the result is valid  for any  commutative ring $R$ \cite[II.7 first remark 2]{MR3600042}.  Besides, the extension of $\sigma$ to $Q$, when the conditions are fulfilled,  is unique; we write it $e_Q(\sigma)$; it is trivial on $N_Q$ and we view it equally as a representation of $M_Q$. The $R$-representation  $e_Q(\sigma)$ of $Q$ or $M_Q$ is smooth, or admissible, or irreducible (when $R$ is a field) if and only if $\sigma$ is.
Let $P_\sigma=M_\sigma N_\sigma$ be the standard parabolic subgroup of $G$ with $\Delta_{P_\sigma}=\Delta_\sigma $ where
\begin{equation}\label{eq:Ds}
\Delta_\sigma=\{\alpha \in \Delta\setminus \Delta_P\ | \ Z\cap M'_\alpha \ \text{acts trivially on } \ \sigma\}.  
\end{equation} 
There is a largest parabolic subgroup $P(\sigma)$ containing $P$ to which $\sigma$ extends: $\Delta_{P(\sigma)} =\Delta_P \cup \Delta_\sigma$. Clearly when $P\subset Q \subset P(\sigma)$, the restriction to $Q$ of 
$e_{P(\sigma)}(\sigma)$ is $e_Q(\sigma)$. If there is no risk of ambiguity, we write
$$e(\sigma)=e_{P(\sigma)}(\sigma).$$
 \begin{definition}\label{def:Gtriple} An {\bf $R[G]$-triple} is a triple $(P,\sigma,Q)$ made out of  a standard parabolic subgroup $P=MN$ of $G$, a smooth $R$-representation of $M$, and a parabolic subgroup $Q$ of $G$ with $P\subset Q \subset P(\sigma)$.  To an $R[G]$-triple $(P,\sigma,Q)$ is associated a smooth  $R$-representation of $G$:  
 $$I_G(P,\sigma,Q) =\Ind_{P(\sigma)}^G(e(\sigma)\otimes \St_Q^{P(\sigma)})$$
where $\St_Q^{P(\sigma)}$ is the quotient of $ \Ind_Q^{P(\sigma)}\charone$, $\charone$ denoting the trivial $R$-representation of   $Q$,  by the sum  of its subrepresentations $\Ind_{Q'}^{P(\sigma)}\charone$, the sum being over the set  of  parabolic subgroups $Q'$ of $G$ with $Q\subsetneq Q'\subset P(\sigma)$. 
\end{definition}

Note that $I_G(P,\sigma,Q) $ is naturally isomorphic to the quotient of $\Ind_{Q}^G(e_Q(\sigma))$ by the sum of its subrepresentations $\Ind_{Q'}^G(e_{Q'}(\sigma))$ for $Q\subsetneq Q'\subset P(\sigma)$ by Lemma \ref{lemma:useful}.  

 \bigskip It might happen that $\sigma$ itself has the form $e_{P}(\sigma_1)$ for some standard parabolic subgroup $P_1=M_1N_1$ contained in $P$  and some $R$-representation $\sigma_1$ of $M_1$. 
 In that case, $P(\sigma_1)=P(\sigma)$ and $e(\sigma)=e(\sigma_1)$. We say that $\sigma$   is  {\bf $e$-minimal} if  $\sigma=e_{P}(\sigma_1)$ implies $P_1=P, \sigma_1=\sigma$.  
 
 \begin{lemma}\label{lemma:min}Let $P=MN $ be a standard parabolic subgroup of $G$ and  let $\sigma$ be an  $R$-representation of $M$. There exists a unique standard parabolic subgroup  $P_{\min,\sigma}=M_{\min,\sigma}N_{\min,\sigma}$  of $G$ and a unique  $e$-minimal    representation of $\sigma_{\min}$ of $M_{\min,\sigma}$ with $\sigma= e_P(\sigma_{\min})$. Moreover $P(\sigma)=P(\sigma_{\min})$ and $e(\sigma)=e(\sigma_{\min})$.
  \end{lemma}
\begin{proof}  We have  \begin{equation}\Delta_{P_{\min, \sigma}}=\{\alpha\in \Delta_P \ | \ \ Z\cap M'_\alpha \ \text{does not act trivially on} \ \sigma\}, 
\end{equation} 
 $\sigma_{\min}$ is  the restriction of $\sigma $ to $M_{\min,\sigma}$, and    \begin{equation}\Delta_{  \sigma_{\min}}=\{\alpha\in \Delta  \ | \ \ Z\cap M'_\alpha \ \text{ acts trivially on } \ \sigma\}.  
\end{equation} 
\end{proof}

\begin{lemma} \label{lemma:2.2}Let $P=MN $ be a standard parabolic subgroup of $G$ and  $\sigma$ an $e$-minimal $R$-representation of $M$. Then $\Delta_P$ and $\Delta_\sigma$ are orthogonal. 
\end{lemma}
 
 That comes from \cite[II.7 Corollary 2]{MR3600042}. That corollary of loc. cit. also shows that when $R$ is a field and $\sigma$ is supercuspidal, then $\sigma$ is $e$-minimal.
 Lemma \ref{lemma:2.2} shows that $\Delta_{P_\sigma}$ and $\Delta_{\sigma_{\min}}$ are orthogonal. 
 
 Note that when   $\Delta_P$ and $\Delta_\sigma$ are orthogonal of union $\Delta=\Delta_P\sqcup \Delta_\sigma$, then $G=P(\sigma)= M M'_{ \sigma}$ and $e(\sigma) $ is the $R$-representation of $G$ simply obtained by extending $\sigma$ trivially on $M'_{ \sigma}$.

\begin{lemma} \label{lemma:2.3} Let $(P,\sigma,Q)$ be an $R[G]$-triple.  Then $(P_{\min,\sigma},\sigma_{\min},Q)$ is an  $R[G]$-triple
  and $I_G(P,\sigma,Q) =I_G(P_{\min,\sigma},\sigma_{\min},Q)$.
\end{lemma}
\begin{proof} We already saw that $P(\sigma)=P(\sigma_{\min})$ and $e(\sigma)=e(\sigma_{\min})$. 
\end{proof}

\subsection{Hecke algebras}\label{subsec:K,Hecke}
We fix a  special parahoric subgroup $\mathcal{K}$ of $G$ fixing a special vertex $x_0$ in the apartment $\mathcal{A}$ associated to $T$ in the  Bruhat-Tits building of the adjoint group of $G$.
If $V$ is an irreducible smooth $C$-representation of $\mathcal K$, we have the  compactly induced representation $\ind_{\mathcal{K}}^G V$ of $G$, its endomorphism algebra $\mathcal{H}_G(\mathcal K,V)$ and the centre $\mathcal Z_G(\mathcal K,V)$ of $\mathcal{H}_G(\mathcal K,V)$. For a standard parabolic subgroup $P=MN$ of $G$, the group $M\cap \mathcal K$ is a special parahoric subgroup of $M$ and $V_{N\cap \mathcal K}$ is an irreducible smooth $C$-representation of $M\cap\mathcal K$. For $W\in \Mod_C^\infty(M)$, there is an injective algebra homomorphism $$\mathcal{S}_P^G:\mathcal{H}_G(\mathcal K,V)\to \mathcal{H}_M(M\cap \mathcal K,V_{N\cap \mathcal K})$$ for which the natural isomorphism 
 $\Hom_G(\ind_K^G V, \Ind_P^G W)\simeq 
\Hom_M(\ind_{M\cap\mathcal K}^M V_{N\cap \mathcal K},  W)$  is $\mathcal{S}_P^G$-equivariant  \cite{MR3331726},  \cite{MR3001801}. Moreover. $\mathcal{S}_P^G (\mathcal Z_G(\mathcal K,V)) \subset  \mathcal Z_M(M\cap \mathcal K,V_{N\cap \mathcal K})$.

Let  $Z(M)$ denote  the maximal split central subtorus of $M$; it is equal to the group of $F$-points of the connected component in $\mathbf T$ of  $\bigcap_{\alpha \in \Delta_{M}} \Ker \alpha$.
Let $z\in Z(M)$.
We say that $z$ strictly contracts an open compact subgroup $N_0$ of $N$ if the sequence $(z^k N_0 z^{-k})_{k\in \mathbb N}$ is strictly decreasing of intersection $\{1\}$.
We say that $z$ strictly contracts $N$ if there exists an open compact subgroup $N_0\subset N$ such that $z$ strictly contracts $N_0$.
Choose $z\in Z(M)$ which strictly contracts $N$.
Let $\tau\in \mathcal{Z}_M(M\cap \mathcal{K},V_{N\cap \mathcal{K}})$ be a non-zero element which supports on $(M\cap \mathcal{K})z(M\cap \mathcal{K})$. (Such an element is unique up to constant multiplication.)
Then $\tau\in \Ima\mathcal{S}_P^G$ and the algebra $\mathcal{H}_M(\mathcal{K}\cap M,V_{N\cap \mathcal{K}})$ (resp.\ $\mathcal{Z}_M(M\cap \mathcal{K},V_{N\cap \mathcal{K}})$) is the localization of $\mathcal{H}_G(\mathcal{K},V)$ (resp.\ $\mathcal{Z}_G(\mathcal{K},V)$) at $\tau$.

   \section{Lattice of subrepresentations of $\Ind_P^G \sigma$, $\sigma$ irreducible admissible}\label{sec:3}
 
 \subsection{Result}\label{S:3.1}
 This section is a direct complement to \cite{MR3600042}. Our coefficient ring is $R=C$. We are given a standard parabolic subgroup $P_1=M_1N_1$ of $G$ and an irreducible admissible $C$-representation $\sigma_1$ of $M_1$. Our goal is to describe the lattice of subrepresentations of $\Ind_{P_1}^G \sigma_1$.  We shall see that $\Ind_{P_1}^G \sigma_1$ has finite length and is multiplicity free, meaning that its irreducible constituents occur with multiplicity $1$. We recall the main result of \cite{MR3600042} :

\begin{theorem}[Classification Theorem]\label{thm:3.1}
(A)   Let $P=MN $ be a standard parabolic subgroup of $G$ and  $\sigma$ a supercuspidal $C$-representation of $M$. Then $\Ind_P^G \sigma\in \Mod_C^\infty(G)$ has   finite length and is multiplicity free of irreducible constituents  the representations $I_G(P,\sigma,Q)$ for $P\subset Q \subset P(\sigma)$, and all $I_G(P,\sigma,Q)$ are admissible.

(B) Let $\pi$ be an irreducible admissible $C$-representation  of $G$. Then, there is a $C[G]$- triple $(P,\sigma,Q)$ with $\sigma$ supercuspidal,  such that $\pi$ is isomorphic to $I_G(P,\sigma,Q)$ and $\pi$ determines $P,Q$ and the isomorphism class of $\sigma$.
\end{theorem}

By the classification theorem, there is a standard parabolic subgroup $P=MN $ of $G$ and   a supercuspidal $C$-representation $\sigma$ of $M$ such that $\sigma_1$ occurs in $\Ind_{P\cap M_1}^{M_1}\sigma$. 
More precisely, if  $P(\sigma)$ is the largest standard parabolic subgroup of $G$  to which $\sigma$ extends, then by Proposition \ref{prop:2.1}, $P(\sigma)\cap M_1$ is the largest standard parabolic subgroup of $M_1$ to which $\sigma$ extends and $$\sigma_1 \simeq I_{M_1}(P\cap M_1,\sigma,Q)\simeq \Ind_{P(\sigma)\cap M_1}^{M_1}(e_{P(\sigma)\cap M_1}(\sigma)\otimes \St_Q^{P(\sigma)\cap M_1})$$  for some parabolic subgroup $Q$ of $M_1$ with $(P\cap M_1)\subset Q \subset (P(\sigma)\cap M_1)$. By transitivity of the parabolic induction,
$$\Ind_{P_1}^G \sigma_1\simeq \Ind_{P(\sigma) }^{G}(e (\sigma)\otimes \Ind^{P(\sigma)}_{P(\sigma)\cap M_1}\St_Q^{P(\sigma)\cap M_1}),$$
and we need to analyse this representation.
Our analysis  is based on \cite[\S 10]{MR2845621}. We recall   the structure of the lattice of subrepresentations of a finite length multiplicity free representation $X$. Let $J$ be the set of its irreducible constituents. For $j\in J$, there is a unique subrepresentation $X_j$ of $X$ with cosocle $j$ - it is the smallest  subrepresentation of $X$ with $j$ as a quotient. Put the order relation $\leq $ on $J$, where $i\leq j$ if $i$ is a constituent of $X_j$. Then the lattice  of subrepresentations of $X$ is isomorphic to the lattice  of lower sets in $(J,\leq )$ - recall that such a lower set is a subset $J'$ of $J$ such that if $j_1\in J, j_2\in J'$ and $j_1\leq j_2$ then $j_1\in J'$. A subrepresentation of $X$ is sent to the lower set made out of its irreducible constituents, and a lower set $J'$ of $J$ is sent to the sum of the subrepresentations $X_j$ for $j\in J'$.   We have $X_j=j$ iff $j$ is minimal in $(J, \leq)$ and $X_j=X$ iff $j$ is maximal in $(J, \leq)$. The socle of $X$ is the direct sum of the minimal $ j\in (J, \leq)$ and the cosocle of $X$ is the direct sum of the maximal $j\in (J, \leq)$. 

In the sequel $J$ will often be identified with $\mathcal P (I)$ for some subset $I$ of $\Delta$, both equipped with the order relation reverse to the inclusion. Thus we rather talk of upper sets in $\mathcal P (I)$ (for the inclusion). In that case the socle $I$ of $X$ and the cosocle $\emptyset$ of $X$  are both irreducible.

\begin{theorem}\label{thm:3.2}   With the above notations, $\Ind_{P_1}^G \sigma_1$ has finite length and is multiplicity free, of irreducible constituents   the $I_G(P,\sigma,Q')$ where $Q'$ is a parabolic subgroup of $G$ satisfying $P\subset Q'\subset P(\sigma)$ and $P_1\cap Q'=Q$. Sending $I_G(P,\sigma,Q')$ to $\Delta_{Q'}\cap (\Delta- \Delta_{P_1})$ gives an isomorphism of the lattice   of subrepresentations 
of $\Ind_{P_1}^G \sigma_1$  onto the lattice of upper sets in  $ \Delta_{P(\sigma)}\cap (\Delta- \Delta_{P_1})$.
\end{theorem}
The first assertion is a consequence of the classification theorem \ref{thm:3.1} since $\Ind_{P_1}^G \sigma_1$ is a subrepresentation of $\Ind_P^G \sigma$. For the rest of the proof, given in  \S \ref{sec:3.2}, we proceed along the classification, treating cases of increasing generality.
As an immediate consequence  of the theorem, we get an irreducibility criterion.

\begin{corollary}\label{cor:3.3} The representation $\Ind_{P_1}^G \sigma_1$ is irreducible if and only if $P_1$ contains $P(\sigma)$.
\end{corollary}

\begin{corollary}The socle and the cosocle of $\Ind_{P_1}^G \sigma_1$ are both irreducible.
\end{corollary}
 This is very different from the complex case \cite{arXiv:1605.08545}.

\subsection{Proof}\label{sec:3.2}
We proceed now to the proof of Theorem \ref{thm:3.2}.
The very first and basic case is when $P_1=B$ and $\sigma_1$ is the trivial representation $\charone $ of $Z$. The irreducible constituents of $\Ind_B^G\charone$ are the $\St_Q^G$ for the different standard parabolic subgroups $Q$ of $G$, each occuring with multiplicity $1$.

\begin{proposition} \label{prop:3.3} Let $Q$ be a standard parabolic subgroup  of $G$. 
\begin{enumerate}
\item The submodule of $\Ind_B^G\charone$ with cosocle $St_Q^G$ is $\Ind_Q^G\charone$. 
\item Sending $\St_Q^G$ to $\Delta_Q$ gives an isomorphism of the lattice of subrepresentations of $\Ind_B^G\charone$  onto the lattice of upper sets in $\mathcal P(\Delta)$.
\end{enumerate}
\end{proposition} 
\begin{proof} By the properties  recalled before Theorem \ref{thm:3.2}, (i) implies  (ii). For (i) the proof is given in \cite[\S 10]{MR2845621}  when $G$ is split, using results of Grosse-Kl\"onne \cite{MR3263032}. The general case is due to T. Ly \cite[beginning of \S 9]{MR3402357}. 
\end{proof}
We have variants of Proposition \ref{prop:3.3}.
 If $Q$ is a standard parabolic subgroup  of $G$, the subrepresentations of $\Ind_Q^G\charone$ are the subrepresentations of $\Ind_B^G\charone$ contained in  $\Ind_Q^G\charone$. So the lattice of subrepresentations of $\Ind_Q^G\charone$ is  isomorphic of the sublattice of   upper sets in $\mathcal P(\Delta)$ consisting of subsets containing $\Delta_Q$; intersecting with $\Delta\setminus \Delta_Q$ gives an isomorphism onto the lattice of upper sets in $\mathcal P(\Delta \setminus \Delta_Q)$.
More generally, 

\begin{proposition} \label{prop:3.5} Let  $P,Q$ be two standard parabolic subgroups of $G$ with $Q\subset P$. 
\begin{enumerate}
\item The irreducible constituents of  $\Ind_P^G
\St_Q^P$ are the $\St_{Q'}^G$ where $Q'\cap P =Q$, and each occurs with multiplicity $1$.
\item Sending 
$\St_{Q'}^G$ to $\Delta_{Q'}\cap (\Delta\setminus \Delta_P)$  gives an isomorphism of the lattice of subrepresentations of $\Ind_P^G\St_Q^P$  onto the lattice of upper sets in $\mathcal P(\Delta\setminus \Delta_P)$.
\end{enumerate}\end{proposition}
 \begin{proof} For (i),  note that  $\Ind_P^G
\St_Q^P$  is the quotient of $\Ind_Q^G\charone$ by the sum of its subrepresentations  $\Ind_{Q'}^G\charone$ for 
$Q'$ where $Q \subsetneq Q'\subset P$ and (i)  is the content of \cite[Corollary 9.2]{MR3402357}. 
The order $\St_{Q'}^G\leq \St_{Q''}^G$ on the irreducible constituents corresponds (as it does in $\Ind_B^G \charone$) to $\Delta_{Q''}\subset \Delta_{Q'}$. Again (ii) follows for (i). 
\end{proof}
\begin{remark}  Note that  $\mathcal P(\Delta-\Delta_P)$ does not depend on $Q$. The unique irreducible quotient of $\Ind_P^G
\St_Q^P$ is $\St_Q^G$,  and its unique subrepresentation is $\St_{Q'}^G$ where $\Delta_{Q'}=\Delta_{Q}\cup (\Delta-\Delta_P)$. 
\end{remark}
 The next case where $P_1=P, \sigma_1=\sigma$ is a consequence of :
\begin{proposition}\label{prop:3.6} Let $P=MN$ be a standard parabolic subgroup  of $G$ and $\sigma$ a supercuspidal $C$-representation of $M$. Then the map $X\mapsto \Ind_{P(\sigma)}^G (e(\sigma)\otimes X)$ gives an isomorphism of the lattice  of subrepresentations of $\Ind^{P(\sigma)}_P \charone$  onto the lattice of subrepresentations of $\Ind_P^G\sigma$.
\end{proposition}
It has the immediate consequence:
 \begin{corollary} Sending 
$I_G(P,\sigma,Q)$ to $\Delta_{Q} \setminus \Delta_P$  gives an isomorphism of the lattice of subrepresentations of $\Ind_P^G\sigma$  onto the lattice of upper sets in $\mathcal P(\Delta_{P(\sigma)}-\Delta_P)$.
\end{corollary}
The  proposition \ref{prop:3.6} is proved in two steps, inducing first to $P(\sigma)$ and then to $G$. In the first step we may as well assume that $P(\sigma)=G$:

\begin{lemma} Let $P=MN$ be a standard parabolic subgroup  of $G$ and $\sigma$ a supercuspidal $C$-representation of $M$ such that $P(\sigma)=G$. Then the map $X\mapsto  e(\sigma)\otimes X $ gives an isomorphism of the lattice of subrepresentations of $\Ind_P^G \charone$  onto the lattice of subrepresentations of $e(\sigma) \otimes \Ind_P^G\charone \simeq \Ind_P^G\sigma$.
\end{lemma}
\begin{proof}
By the classification theorem  \ref{thm:3.1}, the map $X\mapsto  e(\sigma)\otimes X $ gives a bijection between the irreducible constituents of  $\Ind_P^G \charone$ and those of $e(\sigma) \otimes \Ind_P^G\charone$. It is therefore enough to show that, for a parabolic subgroup $Q$ of $G$ containing $P$, the subrepresentation of  $e(\sigma) \otimes \Ind_P^G\charone$ with cosocle $e(\sigma) \otimes \St_Q^G$ is $e(\sigma) \otimes \Ind_Q^G\charone$. Certainly, $e(\sigma) \otimes \St_Q^G$ is a quotient of $e(\sigma) \otimes \Ind_Q^G\charone$. Assume that  $e(\sigma) \otimes \St_Q^G$ is a quotient of $e(\sigma) \otimes \Ind_{Q'}^G\charone$ for some parabolic subgroup $Q'$ of $G$ containing $P$; we want to conclude that $Q'=Q$.  Recall from \S \ref{S:2.2} that $\sigma$ being supercuspidal,  $\Delta_P$ and $\Delta_{\sigma}$ are orthogonal . Also, $e(\sigma)$ is obtained by extending $\sigma$ from $M$ to $G=MM'_\sigma$ trivially on $M'_\sigma$. Upon restriction  to $M'_\sigma$, therefore, $e(\sigma) \otimes \Ind_Q^G \charone $ is a direct sum of copies of $\Ind_Q^G \charone  $ whereas   $e(\sigma) \otimes \St_{Q'}^G$  is a direct sum of copies of $ \St_{Q'}^G $. Thus there is a non-zero $M'_\sigma$-equivariant map $\Ind_Q^G \charone\to \St_{Q'}^G $. Let $\mathbf M^{\is}_\sigma$ denote the isotropic part of the simply connected covering of the derived group $\mathbf M_\sigma$. Then 
$M'_\sigma$ is the image of $M^{\is}_\sigma$  in $M_\sigma$ \cite[II.4 Proposition]{MR3600042}; moreover, as a representation of $M^{\is}_\sigma$, $\Ind_Q^G \charone  $ is simply $\Ind_{Q^{\is}_\sigma}^{M^{\is}_\sigma }\charone  $ where $ Q^{\is}_\sigma$ is the parabolic subgroup of $M^{\is}_\sigma$ corresponding to $\Delta_Q \cap \Delta_\sigma$, whereas $ \St_{Q'}^G $ is $ \St_{Q'^{\is}_\sigma}^{M^{\is}_\sigma}$. It follows that $ \St_{Q'^{\is}_\sigma}^{M^{\is}_\sigma}$ is a quotient of $\Ind_{Q^{\is}_\sigma}^{M^{\is}_\sigma} \charone  $, thus  $\Delta_Q \cap \Delta_\sigma= \Delta_{Q'} \cap \Delta_\sigma$ which implies $\Delta_Q= \Delta_{Q'}$ and $Q=Q'$, since  $\Delta_Q$ and  $\Delta_{Q'}$ both contain  $\Delta_{P}$.
\end{proof}

The second step in the proof of Proposition \ref{prop:3.6} is an immediate consequence of the following lemma, applied to $P(\sigma)$ instead of $P$.

\begin{lemma} 
Let  $P=MN$ be a standard parabolic subgroup of $G$. Let $W$ be a finite length smooth $C$-representation of $M$, and assume that for any irreducible subquotient $Y$ of $W$, $\Ind_P^G Y$ is irreducible. The map $Y\mapsto  \Ind_P^G Y$ from the lattice $\mathcal L_W$ of subrepresentations of $W$ to  the lattice  $\mathcal L_{\Ind_P^G W}$ of subrepresentations of $\Ind_P^G W$ is an isomorphism.
\end{lemma}
\begin{proof} We recall from \cite[Theorem 5.3]{Vigneras-adjoint} that the functor $ \Ind_P^G$ has a right adjoint $R^G_P$ and that the natural map $\Id\to R^G_P \Ind_P^G$ is an isomorphism of functors.
Let $\varphi:  \mathcal L_W \to \mathcal L_{\Ind_P^G W}$ be the map $Y\mapsto  \Ind_P^G Y$ and let  $\psi: \mathcal L_{\Ind_P^G W}\to  \mathcal L_W$ be the map $X\mapsto R_P^G X$. The composite  $\psi \circ \varphi$ is a bijection.
 If $\psi$ is injective, then $\psi$ and $\varphi$ are bijective, reciprocal to each other. To show that $\psi$ is injective, we show first that $X\in \mathcal L_{\Ind_P^G W}$ and $ R_P^G X \in \mathcal L_W$ have always the same length.

Step 1. An irreducible subquotient $X$ of $\Ind_P^G W$ has the form $\Ind_P^G Y$ for an irreducible subquotient $Y$ of $W$; in particular, $R^G_P X \simeq Y$ is irreducible. Thus, $W$ and  $\Ind_P^G W$ have the same length.

Step 2. Let $X$ be a subquotient of $\Ind_P^G W$. Denote the length by $\lg(-)$. We prove that $\lg(R^G_P X)\leq \lg(X)$, by induction on $\lg(X)$. If $X\neq 0$, insert $X$ in an exact sequence 
$0\to X'\to X \to X''\to 0$ with $X''$ irreducible; then the sequence $0\to R^G_P X'\to R^G_P X \to R^G_P X''$ is exact and $R^G_P X''$ is irreducible. So $\lg(R^G_P X)\leq \lg(R^G_P X')+1\leq  \lg(X')+1=\lg(X)$.

Step 3. Let $X\in \mathcal L_{\Ind_P^G W}$.  We deduce from the steps 1 and 2  that $\lg(R^G_P X)= \lg(X)$. Indeed, the exact sequence 
$0\to X\to \Ind_P^G W \to (\Ind_P^G W)/X\to 0$ gives  an exact sequence 
$0\to R^G_P X\to W \to R^G_P ((\Ind_P^G W)/X)$. By Step 2, $\lg(R^G_P X)\leq \lg(X)$ and $\lg(R^G_P((\Ind_P^G W)/X))\leq \lg((\Ind_P^G W)/X)$; by Step 1, $\ell (\Ind_P^G W)=\ell (W)$, so we get equalities instead of inequalities.

 We can show now that $\psi$ is injective. Let $X,X'$ in $
\mathcal L_{\Ind_P^G W}$ such that $R_P^G X=R_P^G X'$. Applying $R_P^G$ to the exact sequence $0\to X\cap X'\to X\oplus X'\to X+X'\to 0$ gives an exact sequence $0\to R_P^G(X\cap X')\to R_P^G X\oplus R_P^G X'\to R_P^G(X+X')$ because $R_P^G $ is compatible with direct sums.  As $R_P^G $  respects the length,  the last map is surjective by length count. But then $R_P^G(X+X')=R_P^G(X)+ R_P^G( X')$ inside $R_P^GW$. Hence $R_P^G(X+X')=R_P^GX=R_P^GX'$. So $X=X'=X+X'$ by length preservation.
 \end{proof}

\begin{remark}
Note that $\lg(R^G_P X)= \lg(X)$ for a subquotient $X$ of $\Ind_P^G W$. Indeed, insert $X$ in an exact sequence 
$0\to X'\to X'' \to X\to 0$ where $X''$ is  a subrepresentation of $\Ind_P^G W$. 
The exact sequence $0\to R^G_P X'\to R^G_P X'' \to R^G_P X$ and $\lg(R^G_P X')= \lg(X')$, $\lg(R^G_P X'')= \lg(X'')$ give $\lg(R^G_P X)\geq \lg(X)$; with Step 2, this inequality is an equality.
\end{remark}

We are now finally in a position to prove Theorem \ref{thm:3.2}. It follows from Proposition \ref{prop:3.6} that $X\mapsto \Ind_{P(\sigma)}^G (e(\sigma)\otimes X)$
 gives an isomorphism of the lattice of subrepresentations of $\Ind_{P_1 \cap P(\sigma)}^{P(\sigma)} \St_Q^{M_1 \cap P(\sigma)}$  (a quotient of the $\Ind_P^{P(\sigma)} \charone$) onto the lattice  of subrepresentations of $\Ind_{P(\sigma)}^G (e(\sigma)\otimes \Ind_{P_1 \cap P(\sigma)}^{P(\sigma)} \St_Q^{M_1 \cap P(\sigma)})$
 isomorphic to $\Ind_{P_1}^G\sigma_1$. The desired result then follows from  Proposition \ref{prop:3.5} applied to $G=P(\sigma), P=P_1 \cap P(\sigma)$ describing the first lattice.

\subsection{Twists by unramified characters} 
Recall the definition of unramified characters of $G$. If $X_F^*(\mathbf G)$ is the group of algebraic $F$-characters of $\mathbf G$, we have a group homomorphism $H_G:G\to 
\Hom(X_F^*(\mathbf G), \mathbb Z)$ defined by $H_G(g)(\chi)=\val_F(\chi(g))$ for $g\in G$ and $\chi\in X_F^*(\mathbf G)$, where $\val_F$ is the normalized valuation of $F$, $\val_F (F-\{0\})=\mathbb Z$. The kernel ${}^0G $ of $H_G$ is open and closed in $G$, and the image $H_G(G)$ has finite index in $\Hom(X_F^*(\mathbf G), \mathbb Z)$. It is well known  (see 2.12 in \cite{Henniart-Lemaire}) that  ${}^0G $ is the subgroup of $G$ generated by its compact subgroups. A smooth character $\chi:G\to C^*$ is {\bf unramified}  if it is trivial on  ${}^0G $; the
unramified characters of $G$ form the group of $C$-points of the algebraic variety 
$\Hom_{\mathbb Z}(H_G(G), \mathbf G_m)$.

Let $\sigma_1$ be an irreducible admissible $C$-representation $\sigma_1$ of $M_1$ and we now examine the effect on $\Ind_{P_1}^G \sigma_1$ of twisting $\sigma_1$ by unramified characters of $M_1$.
As announced in \S \ref{S:1.2}, we want to prove that for a general unramified character $\chi:M_1\to C^*$, the representation $\Ind_{P_1}^G \chi \sigma_1$ is irreducible. For that we translate the irreducibility criterion $P(\chi|_M \sigma)\subset P_1$ given in Corollary \ref{cor:3.3} into more concrete terms. Note that $ \chi|_M$ is an unramified character of $M$. By Proposition \ref{prop:2.1}, $P(\chi|_M \sigma)\subset P_1$ means that for each $\alpha \in \Delta\setminus \Delta_{P_1}$, 
$\chi  \sigma$ is non-trivial on $Z\cap M'_\alpha$. Because $\chi|_M \sigma$ is supercuspidal,  when $\alpha\in \Delta$ is not orthogonal to $\Delta_P$, $\chi  \sigma$ is not trivial on $Z\cap M'_\alpha$. Let $\Delta_{nr}(\sigma)$ be the set of roots 
 $\alpha \in \Delta\setminus \Delta_{P_1}$ orthogonal to $\Delta_P$, such that there exists an unramified  character $\chi_\alpha: M\to C^*$ such that $\chi_\alpha  \sigma$ is  trivial on $Z\cap M'_\alpha$; for $\alpha \in \Delta_{nr}(\sigma)$, choose such a $\chi_\alpha$. 
 
 Recall from \cite[III.16 Proposition]{MR3600042} that the quotient of $Z\cap M'_\alpha$ by its maximal compaxt subgroup is infinite cyclic; if we choose $a_\alpha \in Z\cap M'_\alpha$ generating the quotient, then $\chi \sigma$ is trivial on $Z\cap M'_\alpha$ is and only if 
 $\chi(a_\alpha)=\chi_\alpha(a_\alpha)$. We conclude:
 
 \begin{proposition} Let  $\chi:M_1\to C^*$ be  an unramified   $C$-character   of $M_1$. Then $\Ind_{P_1}^G \chi \sigma_1$ is irreducible if and only if  for all  $\alpha \in \Delta_{nr}(\sigma)$ we have $\chi(a_\alpha)\neq \chi_\alpha(a_\alpha)$. 
 \end{proposition}
 
 The following corollary answers a question of J.-F. Dat.
  \begin{corollary} The  set of  unramified   $C$-characters  $\chi$ of $M_1$ such that
  $\Ind_{P_1}^G \chi \sigma_1$ is irreducible is a Zariski-closed proper subset of the space of unramified characters.
  \end{corollary}
  Indeed by the proposition, the reducibility set is the union, possibly empty, of hypersurfaces with equation  $\chi(a_\alpha)=\chi_\alpha(a_\alpha)$ for $\alpha \in \Delta_{nr}(\sigma)$.

\section{Admissibility}\label{S:4}
\subsection{Generalities}\label{S:4.1} 
 Let $H$ be a locally profinite group and let $R$ be a commutative ring.
   When  $R$ is noetherian,   a subrepresentation of an admissible $R$-representation  of $H$ is admissible. If $H$ is locally pro-$p$ and $p$  is invertible in $R$, then taking fixed points under a pro-$p$ open subgroup   of $H$ is an exact functor \cite[I.4.6]{MR1395151}, so for noetherian $R$ a quotient of an admissible $R$-representation of $H$ is again admissible. This is not generally true, however when $p=0$ in $R$, as the following example shows.
  
  \begin{example}\label{ex:4.1}
  Assume that $p=0$ in $R$ so that $R$ is a $\mathbb Z/p \mathbb Z$-algebra. Let $H$ be the additive group $(\mathbb Z/p \mathbb Z)^{\mathbb N}$, with the product of the discrete topologies on the factors; it is a pro-$p$ group. The space $C_c^\infty(H,R)$ (\S \ref{S:2.2}) can be interpreted as the space of functions $H\to R$ which depend only on finitely many terms of a sequence $(u_n)_{n\in \mathbb N}\in H$. The group $H$ acts by translation yielding a smooth $R$-representation of $H$; if $J$ is an open subgroup of $H$, the $J$-invariant functions in $C^\infty(H,R)$ form the finitely generated free $R$-module of functions $J\backslash H\to R$. In particular, $V=C^\infty(H,R)$ is an admissible $R$-representation of $H$. However the quotient of $V$ by its subrepresentation $V_0=V^H$ of constant functions is not admissible. Indeed,   a linear form $f\in \Hom_{\mathbb Z/p \mathbb Z}(H, R)$ contained in $V$ satisfies $w f(v)-f(v)=f(w+v)-f(v)=f(w) $ for $v,w\in H$ so $f$ produces an $H$-invariant vector in $V/V_0$. Such linear forms make an infinite rank free $R$-submodule of $V$ and $V/V_0$ cannot be admissible. That example will be boosted below in  \S \ref{S:4.2}.
  \end{example}

  \begin{lemma}\label{lemma:4.1-2} Assume that $R$ is noetherian. Let $M$ be an $R$-module and $t$ a nilpotent $R$-endomorphism of $M$. Then $M$ is finitely generated  if and only if $\Ker t$ is.
  \end{lemma}
   \begin{proof} If $M$ is finitely generated  so is its $R$-submodule $\Ker t$,  because $R$ is noetherian. Conversely assume that  $\Ker t$  is a finitely generated  $R$-module; we prove that $M$ is finitely generated by induction over the smallest  integer $r\geq 1$ such that $t^r=0$. The case $r=1$ is a tautology so we assume $r\geq 2$. By induction, the $R$-submodule $\Ker t^{r-1}$  is finitely generated. As $t^{r-1}$ induces an injective map $M/\Ker t^{r-1}\to \Ker t$ of finitely generated image because $R$ is noetherian, the $R$-module $M$ is finitely generated.
     \end{proof}    
 \begin{lemma}\label{lemma:4.4}Assume that $R$ is noetherian. Let $H$ be a locally pro-$p$ group and $J$ an open pro-$p$ subgroup of $H$. Let $M$ be a smooth $R$-representation of $H$ such that   the multiplication $p_M$ by $p$ on $M$ is nilpotent. Then the following are equivalent:
\begin{enumerate}
\item $M$ is admissible;
\item $M^J$ is finitely generated over $R$;
\item $M^J \cap \Ker p_M$  is finitely generated over $R/pR$.
\end{enumerate}  \end{lemma}
    \begin{proof} Clearly (i) implies (ii) and the equivalence of (ii) and (iii) comes from Lemma \ref{lemma:4.1-2}  applied to $t=p_M$. Assume now (ii). To prove  (i), it suffices to prove that for any open normal subgroup $J'$ of $J$, the $R$-module $M^{J'}$ is finitely generated. By Lemma \ref{lemma:4.1-2}, it suffices to do it for  $M^{J'}\cap \Ker p_M$, that is, we can assume $p=0$ in $R$. Now  $M^{J'}=\Hom_{J'}(R,M)\simeq \Hom_J(R[J/J'],M)$ as $R$-modules. The     group algebra $\mathbb F_p[J/J']$ has a decreasing filtration by two sided ideals $A_i$ for $0\leq i\leq r$ with $A_0=\mathbb F_p[J/J'], A_r=\{0\}$ and $A_i/A_{i+1}$ of dimension $1$ over $\mathbb F_p$ with trivial action of $J/J'$. By tensoring with $R$ we get an analogous filtration with $B_i=R\otimes A_i$ for $R[J/J']$. By decreasing induction on $i$, we prove that $\Hom_J(B_i,M)$ is finitely generated over $R$. Indeed, the case $i=r$ is a tautology, the exact sequence $$0\to B_{i+1}\to B_i \to B_{i}/ B_{i+1} \to 0$$ gives an exact sequence
    $$0\to \Hom_J(B_{i}/ B_{i+1},M) \to  \Hom_J(B_{i}, M) \to  \Hom_J(B_{i+1},M)$$
and $ \Hom_J(B_{i}/ B_{i+1},M)\simeq M^J$ is a finitely generated $R$-module by assumption.
Since $\Hom_J(B_{i+1},M)$ is finitely generated by induction, so is   $\Hom_J(B_{i}, M)$ because $R$ is noetherian.  The case $i=0$ gives what we want.
      \end{proof} 

  
 \subsection{Examples}\label{S:4.2}
   Let us now take up the case of  a reductive connected group $G=\mathbf G (F)$. Here the characteristic of $F$ plays a role. When $\charf(F)=0$, $G$ is an analytic $p$-adic group, in particular contains a uniform open  pro-$p$ subgroup, so that at least when $R$ is a finite local $\mathbb Z_p$-algebra \cite{MR2667882} or a field of characteristic $p$ \cite[4.1 Theorem 1 and 2]{MR2533002}, a quotient of an admissible representation of $G$ is still admissible. That does not survive when $\charf(F)=p$, as the following example shows.

    \begin{example}\label{ex:4.2} An admissible representation of  $F^*$ with a non-admissible quotient, when  $\charf(F)=p>0$ and $pR=0$.
    
     If $\charf(F)=p>0$ , then  $1+P_F$ is a quotient of $F^*$.
 Choose a uniformizer $t$ of $F$; it is known that  the map $\prod_{(m,p)=1, m\geq 1}\mathbb Z_p\to 1+P_F$ sending $(x_m)$ to $\prod_m(1+ t^m)^{x_m}$ is a topological group isomorphism. The group $H$ of Example \ref{ex:4.1} is a topological quotient of $F^*$. When and $pR=0$ the admissible $R$-representation $ C_c^\infty(H,R)$ of $H$ with the non-admissible quotient $ C_c^\infty(H,R)/  C_c^\infty(H,R)^H$ inflates to  an admissible $R$-representation $V$ of $F^*$  containing the  trivial representation $V_0=V^{1+P_F}$ with  a non-admissible quotient $V/V_0$.
    \end{example}
  
  That contrast also remains when we consider Jacquet functors. Let $P=MN$ be a standard parabolic subgroup of $G$. Assume that $R$ is noetherian. The parabolic induction $\Ind_P^G:Mod_R^\infty(M)\to Mod_R^\infty(G)$ respects admissibility \cite[Corollary 4.7]{Vigneras-adjoint}. Its left adjoint $L_P^G$ respects admissibility  when $R$ is a field of characteristic different from  $p$  \cite[II.3.4]{MR1395151}. More generally,
  \begin{proposition}\label{prop:noeinvert}  Assume that $R$ is noetherian and that $p$ is invertible in $V$. Let $V\in \Mod_R^\infty(G)$ such that for any open compact subgroup $J$ of $G$, the $R$-module $V^J$ has finite length. Then for any open compact subgroup $J_M$ of $M$, the $R$-module $V_N^{J_M}$ has finite length. 
 \end{proposition} 
  \begin{proof} Assume that $p$ is invertible in $V$. We recall first the assertions (i) and (ii) of    the last part of \cite{Vigneras-adjoint}.  Let $(K_r)_{r\geq 0}$ be  a decreasing sequence   of open pro-$p$ subgroups of $G$ with an Iwahori decomposition with respect to $P=MN$, with $K_r$ normal in $K_0$,  $\cap K_r=\{1\}$. We write $\kappa:V\to V_N$ for the natural map and $M_r=M\cap K_r, N_r=N\cap K_r, W_r=V^{K_rN_0}$. Let $z\in Z(M)$ strictly contracting $N_0$ (subsection~\ref{subsec:K,Hecke}).
Then we have
  
    {\sl For any finitely generated submodule $X$ of $V_N^{M_r}$ there exists $a\in \mathbb N$ with} $z^a X \subset \kappa( W_r)$. 
  
We prove now the proposition.  As  $K_rN_0$ is a compact open subgroup of $G$, the $R$-module $W_r$ has finite length, say $\ell$.  The $R$-modules $\kappa(W_r)$ and $z^a X $ have finite length $\leq \ell$, hence  $X$ also. This is valid for all $X$ hence $V_N^{M_r}$ has finite length $\leq \ell$. We have  $z^a V_N^{M_r} \subset \kappa(W_r) \subset V_N^{M_r} $  for some $a\in \mathbb N$. The three $R$-modules have finite length hence $ \kappa(W_r)=V_N^{M_r} $.
As any open compact subgroup $J_M$ of $M$ contains $M_r$ for $r$ large enough,  the proposition is proved. \end{proof}  

\begin{remark}\label{rem:noeinvert}
The proof is essentially due to Casselman \cite{casselman-note}, who gives it for complex coefficients. The proof shows that  $V_N^{M_r} = \kappa(W_r)$ where $ W_r\subset V^{N_0}$ for all $r\geq 0$. This implies $\kappa(V^{N_0})=V_N$ because $V_N$ being  smooth is equal to  $\bigcup_{r\geq 0}V_N^{M_r}$.
\end{remark}
When $R$ is artinian, any finitely generated $R$-module has finite length, so the  proposition implies:
 
\begin{corollary}
$L_P^G$ respects admissibility  when $R$ is artinian (in particular a field) and $p$ is invertible in $R$.
\end{corollary}
 
 \begin{remark}\label{rem:Dat_result}
 This corollary was already noted by Dat \cite{MR2485794}. The corollary is expected to be true for $R$ noetherian when $p$ is invertible in $R$.
 Using the theory of types, Dat proves it when $G$ is a general linear group, a classical group with $p$ odd, or a group of relative rank $1$ over $F$.
 \end{remark}
  
  Emerton has proved that $L_P^G$ respects admissibility  when  $R$ is a finite local $\mathbb Z_p$-algebra  and $\charf(F)=0$ \cite{MR2667882}. But again, his proof does not survive when $\charf(F)=p>0$ and $pR=0$.
    
   \begin{example}\label{ex:4.2-2} An admissible representation   of  $SL(2,F)$  with  a non-admissible space of $U$-coinvariants, when $\charf(F)=p>0$ and $p R=0$.
   
    Assume $\charf(F)=p>0$ and $p R=0$.   Let $B=TU$ the upper triangular subgroup of $G= SL(2,F)$ and identify 
   $T$ with $F^*$ via $\diag(a,a^{-1})\mapsto a$. Example \ref{ex:4.2} provides an admissible $R$-representation $V$ of $T$ containing the trivial representation $V_0 $ (the elements of $V$ fixed by the maximal pro-$p$ subgroup of $T$), such that $V/V_0$ is not admissible. The representation $\Ind_B^G V$ of $G$ contains $\Ind_B^G V_0$, which contains the trivial subrepresentation $V_{00}$.  We claim that  the quotient $W=(\Ind_B^G V)/ V_{00}$ is admissible and that $W_U$ is not admissible (as a representation of $T$).
   
For the second assertion, it suffices to prove  that $W_U= V/V_0$. The  Steinberg representation $\St=\Ind_B^G V_0/V_{00} $ of $G$   is contained in $W$ and $W/\St$ is isomorphic to $ \Ind_B^G(V/V_0)$. We get an exact sequence
 $$\St_U\to W_U \to (\Ind_B^G(V/V_0))_U\to 0.$$
It is known that $\St_U=0$ (see the more general result in Corollary \ref{cor:6.8} below).
Hence the module $(\Ind_B^G(V/V_0))_U$ is canonically isomorphic to $ V/V_0$  \cite[Theorem 5.3]{Vigneras-adjoint}. 

We  now prove  the admissibility of $W$. Let  $\mathcal{U}$ be  the  pro-$p$ Iwahori subgroup of $G$, consisting of integral matrices in $SL(2,O_F)$ congruent modulo $P_F$ to the strictly upper triangular subgroup of $GL(2,k)$. We prove that 
$W^\mathcal{U}=\St^\mathcal{U}$, so $W$ is admissible by  Lemma \ref{lemma:4.4}, because  $\St$ is admissible. 
 Let $f\in \Ind_B^G V$ with a $\mathcal{U}$-invariant image in $W$, hence for  $x\in \mathcal{U}$, there exists $v_x\in V_{0}$ with $f(gx)-f(g)=v_x$ for all $g\in G$. Put  $s=\begin{pmatrix}0&1\cr-1&0 \end{pmatrix} $.  Then $f(sx)-f(x)=f(sx)-v_x-(f(x)-v_x)=f(s)-f(1)$.  Put $v=f(s)-f(1)\in V$. 
 If $x\in  \overline U$, then $sxs^{-1}\in U$ and $f(sg)=f(sxs^{-1}sg)=f(sxg)$.
 If $x\in \mathcal{U} \cap  U$ and $z\in \mathcal{U}$ we have $f(sz)=f(z)+v=f(xz)+v=f(sxz)$. An easy matrix calculation shows that $\mathcal{U}$ is generated by $\mathcal{U} \cap \overline U$ and $ \mathcal{U} \cap  U$, so the  map $z\mapsto f(sz)$ from $\mathcal{U}$ to $V$ is invariant under left multiplication by  $\mathcal{U}$. 
 We have   $V_0= V^{\mathcal{U} \cap T}$ and $\mathcal{U} \cap T$ is stable by conjugation by $s$. For $t\in  \mathcal{U} \cap T$ and $z\in \mathcal{U}$ we have  $f(sz)=f(stz)=sts^{-1} f(sz)$ and $f(z)=f(sz)-v=f(stz) -v=f(tz)=tf(z)$. Therefore, $f(sz)$ and $f(z)$ lie in $V_0$. But  $G$ is the union of $B\mathcal{U}$ and $Bs\mathcal{U}$, so $f(g)\in V_0$ for all $g\in G$, which means $f\in \Ind_B^G V_0$ and its image in $W$ does belong to $\St^\mathcal{U}$.
  \end{example}

\subsection{Admissibility and $R^G_P$}\label{S:4.3} We turn to the main result of this section (theorem \ref{thm:1.4} of the introduction) for a general connected reductive group $G$ and a standard parabolic subgroup $P=MN$ of $G$. 
    \begin{lemma}\label{lemma:4.7}  Let $V$ be a noetherian $R$-module, let $t$ be an endomorphism of $V$, and view $V$ as a $\mathbb Z[T]$-module with $T$ acting through $t$. Then the map $f\mapsto f(1)$ yields an isomorphism $e$ from $\Hom_{\mathbb Z[T]}(\mathbb Z[T, T^{-1}],M)$ onto the submodule $V^\infty=\cap_{n\geq 0}t^n V$ of infinitely $t$-divisible elements.
    \end{lemma}
     \begin{proof}  A $\mathbb Z[T]$-morphism $f:\mathbb Z[T, T^{-1}]\to V$ is determined by the values $m_n=f(T^{-n})$ for $n\in \mathbb N$, which are only subject to the condition $t m_{n+1}=m_n$ for $n\in \mathbb N$. Certainly $f(1)=m_0$
 is in   $V^\infty$. Let us prove that $e$ is surjective. As $V$ is noetherian, there is some $n\geq 0$ such that $\Ker t^{n+k}=\Ker t^{n}$ for $k\geq 0$. Let $m\in V^\infty$ and for $k\geq 0$ choose $m_k$ such that $m=t^km_k$. Then for $k\geq 0$, $m_{n+k}-t m_{m+k+1}$ belongs to $\Ker t^{n+k}$ so that $t^nm_{n+k}=t^{n+1} m_{m+k+1}$ Putting $\mu_k= t^n m_{n+k}$ we have $ \mu_k=t \mu_{k+1}$ and $\mu_0=m$. Therefore $e$ is surjective. 
 By \cite[\S 2, No 2, Proposition~2]{MR3027127}, the action of $t$ on $V^\infty$ being surjective is bijective because the $R$-module $V^\infty$ is noetherian, so $e$ is indeed bijective.  \end{proof} 
 
 \begin{theorem}\label{thm:4.11} Assume that $R$ is  noetherian and $p$ is nilpotent in $R$. Then the functor $R^G_P: \Mod_R^\infty(G)\to  \Mod_R^\infty(M)$ respects admissibility.
 \end{theorem}

\begin{proof}
Let $\pi$ be an admissible $R$-representation of $G$ and we prove $R_P^G(\pi)$ is admissible.
By Lemma~\ref{lemma:4.4}, we may replace $\pi$ with $\Ker(p\colon \pi\to \pi)$, hence we assume that $p = 0$ in $R$.

Recall that we have fixed a special parahoric subgroup $\mathcal{K}$ in \S\ref{subsec:K,Hecke}.
Take a finite extension $\mathbb{F}$ of $\mathbb{F}_p$ such that all absolute irreducible representations of $\mathcal{K}$ in characteristic $p$ are defined over $\mathbb{F}$.
Then for any open pro-$p$ subgroup $J$ of $\mathcal{K}\cap M$, we have
\begin{align*}
R_P^G(\pi)^J\subset R_P^G(\mathbb{F}\otimes_{\mathbb{F}_p}\pi)^J& = \Hom_{\mathbb{F}[J]}(\mathbb{F},R_P^G(\mathbb{F}\otimes_{\mathbb{F}_p}\pi))\\
& = \Hom_{\mathbb{F}[\mathcal{K}\cap M]}(\Ind^{\mathcal{K}\cap M}_J(\mathbb{F}),R_P^G(\mathbb{F}\otimes_{\mathbb{F}_p}\pi)).
\end{align*}
Since we have a filtration on $\Ind^{\mathcal{K}\cap M}_J(\mathbb{F})$ whose successive quotients are absolute irreducible representations, it is sufficient to prove that the $R$-module
\[
\Hom_{\mathbb{F}[\mathcal{K}\cap M]}(V,R_P^G(\mathbb{F}\otimes_{\mathbb{F}_p}\pi)).
\]
is finitely generated for any irreducible $\mathbb{F}$-representation $V$ of $\mathcal{K}\cap M$.

Put $\pi_1 = \mathbb{F}\otimes_{\mathbb{F}_p}\pi$.
This is also admissible.
Let $V_0$ be an irreducible $\mathbb{F}$-representation of $\mathcal{K}$ which is $\overline{P}$-regular \cite[Definition~3.6]{MR3001801} and $(V_0)_{\mathcal{N}\cap \mathcal{K}}\simeq V$.
This $V_0$ exists by the classification of absolute irreducible representations of $\mathcal{K}$ (\cite[Theorem~3.7]{MR3001801}, \cite[III.10 Lemma]{MR3600042}).
Then by \cite[Theorem~1.2]{MR3001801} we have
\[
\Ind_P^G(\ind_{\mathcal{K}\cap M}^M(V))\simeq \mathcal{H}_M(\mathcal{K}\cap M,V)\otimes_{\mathcal{H}_G(\mathcal{K},V)}\ind_{\mathcal{K}}^G(V_0).
\]
Hence
\begin{align*}
\Hom_{\mathbb{F}[\mathcal{K}\cap M]}(V,R_P^G(\pi_1))
& = \Hom_{\mathbb{F}[M]}(\ind_{\mathcal{K}\cap M}^M(V),R_P^G(\pi_1))\\
& = \Hom_{\mathbb{F}[G]}(\Ind_P^G(\ind_{\mathcal{K}\cap M}^M(V)),\pi_1)\\
& = \Hom_{\mathbb{F}[G]}(\mathcal{H}_M(\mathcal{K}\cap M,V)\otimes_{\mathcal{H}_G(\mathcal{K},V_0)}\ind_{\mathcal{K}}^G(V_0),\pi_1)\\
& = \Hom_{\mathcal{H}_G(\mathcal{K},V_0)}(\mathcal{H}_M(\mathcal{K}\cap M,V),\Hom_{\mathbb{F}[\mathcal{K}]}(V_0,\pi_1)).
\end{align*}
As $\mathcal{H}_M(\mathcal{K}\cap M,V)$ is a localization of $\mathcal{H}_G(\mathcal{K},V_0)$ at some $\tau\in \mathcal{Z}_G(\mathcal{K},V_0)$, the $R$-module 
\[
\Hom_{\mathcal{H}_G(V_0)}(\mathcal{H}_M(\mathcal{K}\cap M,V),\Hom_{\mathbb{F}[\mathcal{K}]}(V_0,\pi_1))
\]
identifies with 
\[
\Hom_{\mathbb{F}[T]}(\mathbb{F}[T,T^{-1}],\Hom_{\mathbb{F}[\mathcal{K}]}(V_0,\pi_1))
\]
with $T$ acting on $\Hom_{\mathbb{F}[\mathcal{K}]}(V_0,\pi_1)$ through $\tau$.
Since the $R$-module $\Hom_{\mathbb{F}[\mathcal{K}]}(V_0,\pi_1)$ is finitely generated and $R$ is noetherian, Lemma~\ref{lemma:4.7} show that $\Hom_{\mathbb{F}[T]}(\mathbb{F}[T,T^{-1}],\Hom_{\mathbb{F}[\mathcal{K}]}(V_0,\pi_1))$ is also a finitely generated $R$-module.
\end{proof}

\begin{remark}
Using \cite[Proposition 4.5]{arXiv:1703.04921} instead of \cite[Corollary~1.3]{MR3001801}, the argument works replacing $\mathcal{K}$ by a pro-$p$ Iwahori subgroup.
Note that the only irreducible representation of pro-$p$ Iwahori subgroup in characteristic $p$ is the trivial representation.
So we may take $\mathbb{F} = \mathbb{F}_p$.
\end{remark}
When $R$ is noetherian,  $\Ind_P^G: \Mod_R^\infty(M)\to \Mod_R^\infty(G)$ respects admissibility and induces  a functor  $\Ind_P^{G,a}:\Mod_R^a(M)\to \Mod_R^a(G)$ between the category  of admissible representations.  Emerton's {\bf $\overline P$-ordinary part functor} $\Ord_{\overline P}^G$  is  right adjoint to $\Ind_P^{G,a}$.
For $V\in \Mod_R^\infty(G)$ admissible,  
\begin{equation}\label{eq:OrdP} \Ord_{\overline P}^G V = (\Hom_{R[\overline N]}(C_c^\infty(\overline N, R), V))^{Z(M)-f},
\end{equation}
  is the space of $Z(M)$-finite vectors of $\Hom_{R[\overline N]}(C_c^\infty(\overline N, R), V)$ with the natural action of $M$  (the representation $\Ord_{\overline P}^G V$ of $M$ is smooth)  \cite[\S 8]{Vigneras-adjoint}. 

 If $R_P^G$  respects admissibility, the restriction of $R_P^G$ to the category of admissible representations is necessarily    right adjoint to $\Ind_P^{G,a} $, hence  is isomorphic to  $\Ord_{\overline P}^G$.

   \begin{corollary}\label{cor:4.10}  Assume $R$  noetherian and either $p$ nilpotent in $R$. Then $R_P^G$ is isomorphic  to  the $\overline P$-ordinary part functor $\Ord_{\overline P}^G$ on admissible $R$-representations of $G$.
    \end{corollary}

       \begin{corollary}\label{cor:4.12} Assume that $R$ is a field of characteristic $p$. Let $V$ be an irreducible admissible $R$-representation of $G$ which is a quotient of $\Ind_P^GW$ for some smooth $R$-representation $W$ of $M$. Then $V$ is a quotient of $\Ind_P^GW'$  for some irreducible admissible subquotient $W'$ of $W$.
       \end{corollary}
         The latter corollary was previously known only under the assumption that $W$ admits a central character and $R$ is algebraically closed \cite[Proposition 7.8]{MR3001801}.  Its proof is as follows.
 By assumption, there is a non-zero $M$-equivariant map $f:W\to R^G_PV$. By the theorem $R^G_PV$ is admissible so $f(W)$ contains an irreducible admissible subrepresentation $W'$ because $\charf R=p$ \cite[Lemma 7.9]{MR3001801}. The inclusion of $W'$ into $R^G_PV$ gives a non-zero  $G$-equivariant map $\Ind_P^GW' \to V$, so that $V$ is a quotient of $\Ind_P^GW'$.

\begin{remark}   
When $R$ is a field of characteristic $ \neq p$ and  $R^G_P $ respects admissibility,  then Corollary  \ref{cor:4.12} remains true.
\end{remark}
\begin{proof} It suffices to  modify the proof of Corollary \ref{cor:4.12} as follows.  We  reduce to a finitely generated  $R$-representation $W$ of $G$, by replacing $W$ by the representation of $M$ generated by the values of  an element   of $ \Ind_P^GW$ with non-zero image in $V$. An admissible quotient of $W$ is also finitely generated,  thus is  of finite length \cite[II.5.10]{MR1395151}, and in particular, contains an irreducible admissible subrepresentation $W'$. By the arguments in the proof of Corollary \ref{cor:4.12},  $V$ is a quotient of $\Ind_P^GW'$. 
 \end{proof} 
   
Let $V\in \Mod_R^\infty(G)$. Obviously,  $ \Ord_{\overline P}^G(V) $ given by the formula \eqref{eq:OrdP}depends only on the restriction of $V$ to $\overline P$, and $L_{P}^G V=V_{N}$ depends only on the restriction of $V$ to $P$. We ask:
 
 \begin{question} Does $R_{P}^G V$ depend only on the restriction of $V$ to $\overline P$ ?
 \end{question}

To end  this section we assume that $R$ is noetherian and $p$ is invertible in $R$ and we compare $L_P^G$ and $\Ord_P^G$. In the same situation than  in Proposition  \ref{prop:noeinvert}, we take up the same  notations. For $V\in \Mod_R^a(G)$ we have the $R$-linear map 
\begin{equation}\varphi\mapsto \kappa ( \varphi (1_{N_0})):  \Ord_P^G(V)\xrightarrow{e_V}L_P^G(V)=V_N,
\end{equation}
where $1_{N_0}$ is the characteristic function of $N_0$. Replacing $N_0$ by a compact open subgroup $J_N\subset N$   multiplies $e_V$ by the generalized index $[J_N:N_0]$ which is a power of $p$. Following the action of $m\in M$ which sends $\varphi\in  \Ord_P^G(V)$ to $m\circ \varphi \circ m^{-1}$, 
$$\kappa ((m \varphi)(1_{N_0}))=\kappa (m(  \varphi (1_{m^{-1}N_0 m })))=  [m^{-1}N_0 m:N_0]m ( \kappa (   \varphi (1_{N_0}))),$$
we get that $e_V$ is an $R[M]$-linear map $ \Ord_P^G(V)\to \delta_P^{-1}L_P^G(V)$, and that $V\mapsto e_V$ defines on $\Mod_R^a(G)$ a morphism of functors $e: \Ord_P^G\to \delta_P^{-1}L_P^G$.
Here $\delta_P(m) = [mN_0m^{-1}:N_0]$ for $m\in M$.

 \begin{proposition} Assume  $R$  noetherian and $p$ invertible in $R$. Let $V\in \Mod_R^\infty(G)$ such that for any open compact subgroup $J$ of $G$, the $R$-module $V^J$ has finite length. Then $e_V$ is an isomorphism.
  \end{proposition}
\begin{proof} 1) We  recall the Hecke version of the Emerton's functor \cite[\S 7, \S 8]{Vigneras-adjoint} for $V\in \Mod_R^a(G)$.
We fix an open compact subgroup $N_0$ of $N$ as in \cite[\S 3.1.1]{MR2667882}.
The monoid $M^+\subset M$ of $m\in M$ contracting $N_0$ acts on $V^{N_0}$ by the Hecke action: 
$$(m,v)\mapsto h_m(v)= \sum_{n\in N_0/mN_0m^{-1}}nmv: M^+\times V^{N_0}\to V^{N_0}.$$
We write $I_{M^+}^M :\Mod_R(M^+)\to \Mod_R(M)$ for the induction, right adjoint of the restriction $\Res_{M^+}^M :\Mod_R(M)\to \Mod_R(M^+)$.  
Let $z\in Z()M$ strictly contracting $N_0$ (subsection~\ref{subsec:K,Hecke}).
{\sl The map
\begin{equation}\label{eq:PhiV}\varphi \mapsto f(m)= (m\varphi)(1_{N_0}): \Ord_P^G V\xrightarrow{\Phi_V} (I_{M^+}^MV^{N_0})^{z^{-1}-f}
\end{equation}
is an isomorphism in $\Mod_R^a(M)$} (loc. cit. Proposition 7.5  restricted to the smooth and $Z(M)$-finite part, and Theorem 8.1 which says that the right hand side is admissible, hence is smooth and $Z(M)$-finite).    {\sl For any $r\geq 0$, $W_r$  is stable by $h_z$,  the restriction from $M$ to $z^{\mathbb Z}$ gives a $R[z^{\mathbb Z}]$-isomorphism} 
\begin{equation}\label{eq:droite}((I_{M^+}^MV^{N_0})^{z^{-1}-f})^{M_r}\simeq (I_{z^{\mathbb N}}^{z^{\mathbb Z}}(V^{N_0M_r}))^{z^{-1}-f}\end{equation}
 (loc. cit. Remark 7.7 for $z^{-1}$-finite elements, Proposition 8.2), {\sl   the RHS of \eqref{eq:droite} is contained in  $I_{z^{\mathbb N}}^{z^{\mathbb Z}}(W_r)$,  and we have the   isomorphism}  
 $$f\mapsto (f(z^{-n}))_{n\in \mathbb N}: I_{z^{\mathbb N}}^{z^{\mathbb Z}}(W_r)\to  \{ (x_n)_{n\geq 0}, \ x_n\in h_z^\infty(W_r)=\cap_{n\in \mathbb N} h_z^n(W_r),  \ h_z(x_{n+1})=x_n\}  $$ 
  (loc. cit. Proposition 8.2, for the isomorphism Lemma \ref{lemma:4.7}).

2)  The inclusion  above is an equality 
$ (I_{z^{\mathbb N}}^{z^{\mathbb Z}}(V^{N_0M_r}))^{z^{-1}-f}= I_{z^{\mathbb N}}^{z^{\mathbb Z}}(W_r),$ because the  map
 \begin{equation}\label{eq:1}f\to f(1): I_{z^{\mathbb N}}^{z^{\mathbb Z}}(W_r)\to h_z^\infty(W_r) \end{equation} is an isomorphism:   on the finitely generated $R$-module $h_z^\infty(W_r)$, $h_z$ is   bijective as it is surjective  (Lemma \ref{lemma:4.7}), hence 
 any element  $f\in I_{z^{\mathbb N}}^{z^{\mathbb Z}}(W_r)$ is $z^{-1}$-finite as  $(z^{-n}f)(1)=f(z^{-n})$ for  $n\in \mathbb N$  and  a $R$-submodule of $h_z^\infty(W_r)$ is finitely generated.   
  
 Through the isomorphisms \eqref{eq:PhiV}, \eqref{eq:droite}, \eqref{eq:1} the restriction of $e_V$ to  $(\Ord_P(V))^{M_r}$  translates into the restriction $\kappa_r$ of $\kappa$ to  $h_z^\infty(W_r)$
$$ h_z^\infty(W_r) \xrightarrow{\kappa_r}V_N^{M_r} .$$ 
 
3) The sequence $\Ker (h_z^n|_{W_r})$ is increasing hence stationary. Let $n$ the smallest number such that $\Ker (h_z^n|_{W_r})= \Ker (h_z^{n+1}|_{W_r})$. By \cite[III.5.3 Lemma, beginning of the proof of III.5.4 Lemma]{casselman-note},   
$$\Ker (\kappa|_{W_r})=\Ker (h_z^n|_{W_r}), \quad h_z^n (W_r)\cap \Ker (h_z^n|_{W_r})=0.$$

4) If the $R$-module $W_r$ has finite length,  $h_z^\infty(W_r)=h_z^n(W_r)$ and $W_r=h_z^n(W_r)\oplus \Ker (h_z^n|_{W_r})$. Indeed, the sequence $(h_z^m(W_r))_{m\in \mathbb N}$ is decreasing and $\ell(W_r)=\Ker (h_z^{m}|_{W_r})+\ell(h_z^m(W_r))$. Therefore   $\kappa_r$ is injective of image $\kappa(W_r)$. As $\kappa(W_r)= V_N^{M_r}$ (proof of Proposition \ref{prop:noeinvert}, $\kappa_r$  is an isomorphism.

5)  If  the $R$-module $W_r$ has finite length for any $r\geq 0$, then  $\kappa(V^{N_0})=V_N$  (Remark  \ref{rem:noeinvert}) and $e_V$ is an isomorphism.
\end{proof}
\begin{remark}\label{rem:sfini}
The arguments in part 1) show that for $V\in \Mod_R^a(G)$, we have $\Ord_P^G V =  (\Hom_{R[\overline N]}(C_c^\infty(\overline N, R), V))^{z^{-1}-f}$ for any $z\in Z(M) $ strictly contracting $\overline{N}$ (subsection~\ref{subsec:K,Hecke}).
\end{remark}
When $R$ is artinian, any finitely generated $R$-module has finite length, so the  proposition implies:
\begin{corollary}\label{rem:second-adjoint}
Assume $R$ artinian (in particular a field) and $p$ is invertible in $R$.  
On $\Mod_R^a(G)$, the functors $ \Ord_P^G$ and $ \delta_P^{-1}L_P^G$ are isomorphic via $e$.
\end{corollary}
\begin{remark}
We expect the corollary to be true for noetherian $R$ with $p$ invertible in $R$. We even expect that the functors $ R_{\overline P}^G$ and $ \delta_P^{-1}L_P^G$ are isomorphic on $\Mod_R^\infty(G)$ (second adjunction). That is proved by Dat  for the same groups as in Remark \ref{rem:Dat_result}, and for those groups $ R_{\overline P}^G$ preserves admissibility. 
\end{remark}

 \subsection{Admissibility of $I_G(P,\sigma,Q)$}
 \begin{theorem}Assume $R$ noetherian. Let $(P,\sigma,Q)$ be an $R[G]$-triple with $\sigma$ admissible. If $p$ is invertible or nilpotent in $R$, then $I_G(P,\sigma, Q)$ is admissible.
 \end{theorem}
  It is already known that $\St_Q^G$ is admissible when $R$ is noetherian (when $G$ is split \cite[Corollary B]{MR3263032}, in general \cite[Remark 5.10]{MR3402357}).
  
\begin{proof}  Since parabolic induction preserves admissibility, we may assume $P(\sigma)=G$. If $p$ is invertible in $R$, the result is easy because $I_G(P,\sigma, Q)$ is a quotient of $\Ind_P^G \sigma$: if $\sigma$ is admissible so are $\Ind_P^G \sigma$ and all its subquotients.  Therefore, it is enough  to prove the theorem when $p$ is nilpotent in $R$ and $P(\sigma)=G$.
Then $I_G(P,\sigma, Q)= e(\sigma)\otimes _R \St_Q^G$.
Let $\mathcal{U}$ be a pro-$p$-Iwahori subgroup which has the Iwahori decomposition $\mathcal{U} = (\mathcal{U}\cap \overline{N})(\mathcal{U}\cap M)(\mathcal{U}\cap N)$.
Using Lemma \ref{lemma:4.4} that is a consequence of \cite[Theorem~4.7]{AHenV2} which shows that  the natural linear map 
$e(\sigma)^\mathcal{U}\otimes _R (\St_Q^G)^\mathcal{U}\to (e(\sigma)\otimes _R \St_Q^G)^\mathcal{U}$  is an isomorphism, hence  $(e(\sigma)\otimes _R \St_Q^G)^\mathcal{U}$ is a finitely generated $R$-module.
\end{proof}

\subsection{$\Ind_P^G$ does not respect finitely generated representations}
We add a few remarks on finiteness: when $R$ is the complex number field, the parabolic induction preserves the finitely generated representations~\cite[Variante~3.11]{MR771671}.
However when $R = C$ (recall that $C$ is an algebraically closed field of characteristic $p$), this does not hold as we see in the following.
\begin{proposition}\label{prop:finiteness_of_parabolic}
Let $P = MN$ be a proper parabolic subgroup, $V_0$ an irreducible $C$-representation of $M\cap \mathcal{K}$.
Set $\sigma = \ind_{M\cap \mathcal{K}}^MV_0$.
Then $\Ind_P^G\sigma$ is not noetherian.
In particular it is not finitely generated.
\end{proposition}
\begin{proof}
Let $V$ be an irreducible $C$-representation of $\mathcal{K}$ such that $V_{N\cap \mathcal{K}}\simeq V_0$ and $V$ is $\overline{P}$-regular (\cite[Theorem~3.7]{MR3001801}, \cite[III.10 Lemma]{MR3600042}).
Let $I_{V}\colon \ind_K^GV\to \Ind_P^G\sigma$ be the injective homomorphism defined in \cite[Definition~2.1]{MR3001801}.
Then by \cite[Theorem~1.2]{MR3001801}, $I_{V}$ induces an isomorphism
\[
\Ind_P^G\sigma\simeq \mathcal{H}_M(M\cap \mathcal{K},V_0)\otimes_{\mathcal{H}_G(\mathcal{K},V)}\ind_{\mathcal{K}}^G V.
\]
Set $X = \Ima I_{V}$.
As $\mathcal{H}_M(M\cap \mathcal{K},V_0)$ is the localization of $\mathcal{H}_G(\mathcal{K},V)$ at $\tau\in \mathcal{Z}_G(\mathcal{K},V)$ (subsection~\ref{subsec:K,Hecke}), we have $\Ind_P^G\sigma = \bigcup_{n\in\mathbb{Z}_{\ge 0}}\tau^{-n}X$.
By the following lemma, $X\ne \Ind_P^G\sigma$ and since $\tau$ is invertible on $\Ind_P^G\sigma$, we have $\tau^{-n}X\ne \Ind_P^G\sigma$.
Hence $\Ind_P^G\sigma$ is not noetherian.
\end{proof}
\begin{lemma}
Assume $R = C$.
If $P\ne G$, then $I_V$ is not surjective for any irreducible representation $V$ of $\mathcal{K}$.
\end{lemma}
\begin{proof}
Take $\tau\in \mathcal{Z}_G(\mathcal{K},V)$ such that $\mathcal{H}_M(M\cap \mathcal{K},V_{N\cap \mathcal{K}}) = \mathcal{H}_G(\mathcal{K},V)[\tau^{-1}]$.
Since the ring homomorphism $\mathcal{S}^G_P\colon \mathcal{H}_G(\mathcal{K},V)\to \mathcal{H}_M(M\cap \mathcal{K},V_{N\cap \mathcal{K}})$ is not surjective (this follows from the description of the image of $\mathcal{S}^G_B\colon \mathcal{H}_G(\mathcal{K},V)\to \mathcal{H}_Z(Z\cap \mathcal{K},V_{U\cap \mathcal{K}})$~\cite{MR3331726}), $\tau$ is not invertible.
Assume that $I_V$ is surjective.
Since $\tau$ is invertible on $\Ind_P^G(\ind_{M\cap \mathcal{K}}^MV_{N\cap \mathcal{K}})$ and $I_V$ is $\mathcal{H}_G(\mathcal{K},V)$-equivariant, $\tau$ is invertible on $\ind_K^GV$.
Hence $\tau$ is a unit in $\End_G(\ind_K^GV) = \mathcal{H}_G(\mathcal{K},V)$.
This is a contradiction.
\end{proof}
We also have the following.
\begin{proposition}
If $P\ne G$ and $R = C$, then the functor $R_P^G$ does not preserve infinite direct sums.
\end{proposition}
\begin{proof}
For an infinite family of representations $\{\pi_n\}$ and a finitely generated representation $\sigma$ of $M$, we have $\Hom_M(\sigma,\bigoplus_n R_P^G(\pi_n)) = \bigoplus_n\Hom(\sigma,R_P^G(\pi_n))\simeq \bigoplus_n\Hom(\Ind_P^G\sigma,\pi_n)$.
Hence it is sufficient to prove
\[
\bigoplus_n\Hom_G(\Ind_P^G\sigma,\pi_n)\ne \Hom_G(\Ind_P^G\sigma,\bigoplus_n\pi_n)
\]
for some $\{\pi_n\}$ and $\sigma$.

We take $\sigma$ as in Proposition~\ref{prop:finiteness_of_parabolic} and use the same notation as in the proof of Proposition~\ref{prop:finiteness_of_parabolic}.
Set $\pi = \Ind_P^G\sigma$ and $X_n = \tau^{-n}X$.
Then we have $\pi \ne X_n$ for all $n\in\mathbb{Z}_{\ge 0}$ and $\bigcup_n X_n = \pi$.
The homomorphism $\Ind_P^G\sigma = \pi \to \bigoplus_n \pi/X_n$ induced by the projections $\pi\to \pi/X_n$ is not in $\bigoplus_n\Hom_G(\Ind_P^G\sigma,\pi/X_n)$.
\end{proof}
\begin{remark}
The functor $R_P^G$ preserves infinite direct sums when $R_P^G = \delta_P L_{\overline{P}}^G$ (the second adjoint theorem) holds true.
It is known when $R$ is the complex number field \cite{Bernstein-second-adjoint}, when $R$ is an algebraically closed field of characteristic different from $p$~\cite[II.3.8 (2)]{MR1395151} and in many cases when $p$ is invertible in $R$~\cite[Th\'eor\`eme~1.5]{MR2485794}.
\end{remark}

\section{Composing $\Ind_P^G  $ with adjoints of  $\Ind_{P_1}^G  $ when $p$ is nilpotent}\label{S:5}
 Let  us keep a general reductive connected group $G$ and a commutative ring $R$. Let $P=MN, P_1=M_1N_1$ be two standard parabolic subgroups of   $G$.  
    
      \subsection{Results}\label{S:5.1}   We start our investigations on the compositions of the functor   $\Ind_P^G  $  with $L_{P_1}^G$ and $R_{P_1}^G$ by some considerations on coinvariants.

\begin{lemma}\label{lemma:6.8} Let $H$ be a group and let $V,W$ be $R[H]$-modules, and assume that $H$ acts trivially on $W $. Then  the $R$-modules  $(V\otimes_R W)_{H}$ and $V_{H}\otimes_R W$ are isomorphic.
\end{lemma}
\begin{proof} We write as usual $V(H)$ for the $R$-submodule of $V$ generated by the elements $hv-v$ for $h\in H,v\in V$. The exact sequence 
$0\to V(H)\to V\to V_{H}\to 0$ of $R[H]$-modules gives by tensor product over $R$ with $W$ an exact  sequence $$  V(H)\otimes_RW \to V\otimes_RW \to V_{H}\otimes_R W \to 0$$ of $R[H]$-modules. Because $H$ acts trivially on $W$, $(V\otimes_RW)(H)$ is the image of $V(H)\otimes_RW$ in $ V\otimes_RW $, hence the result.
\end{proof}
      
As a consequence of Lemma \ref{lemma:6.8}, if  $V$ is a  $\mathbb Z[H]$-module and $W=R$ with the trivial action of $H$, the $R$-modules       $(V\otimes_{\mathbb Z} R)_{H}$ and $V_{H}\otimes_{\mathbb Z} R$ are isomorphic.

 \bigskip    Let us study now  $C^\infty_c(H,R)_H= C^\infty_c(H,\mathbb Z)_H \otimes_{\mathbb Z} R$. A  right {\bf  Haar measure on $H$ with values in $R$} is a non-zero element of  $\Hom_R(C^\infty_c(H,R)_H,R)$.

 \begin{proposition}\label{prop:5.4} Let  $H$ be a locally pro-$p$ group having an  infinite open pro-$p$ subgroup $J$ and $W$ an $R$-module on which $H$ acts trivially. The 
  $R$-module of $H$-coinvariants $C^\infty_c(H,W)_H$ is  isomorphic to $R[1/p]\otimes_RW$. 
   \end{proposition}
 
\begin{proof}    Lemma \ref{lemma:6.8} reduces us to the case $R=W=\mathbb Z$.  We consider the right Haar measure on $H$ with values in $\mathbb Z[1/p]$ sending the characteristic function $\charone_{J}$
of $J$  to $1$. It induces a linear map 
$C^\infty_c(H,\mathbb Z)\to \mathbb Z[1/p]$. This map is surjective because $J$ is infinite hence has open subgroups of  index $p^n$ for $n$ going to infinity. Let $f$ be in its kernel. We write $f$ as a finite sum $\sum_i a_i h_i \charone_{J'}$ where $J'$  is a suitable open subgroup  of $J$, $a_i\in \mathbb Z, h_i\in H$. Then 
$\sum_i a_i [J:J']^{-1}=0$ in $\mathbb Z[1/p]$ hence $\sum_i a_i =0$ and $f=\sum_i a_i ( h_i \charone_{J'}-\charone_{J'})$ belongs to the kernel of the natural map $C^\infty_c(H,\mathbb Z)\to (C^\infty_c(H,\mathbb Z))_H$.
We thus get an isomorphism $C^\infty_c(H,\mathbb Z)_H\simeq \mathbb Z[1/p]$.
Therefore   $C^\infty_c(H,W)_H\simeq R[1/p]\otimes_RW$.  
 \end{proof}

 \begin{corollary}\label{cor:5.4}
    $C^\infty_c(H,R)_H=\{0\}$  if and only if $p$ is nilpotent in $R$, and in general, $C^\infty_c(H,W)_H=\{0\}$
  if and only if $W$ is $p$-torsion.
  
  $\Hom_R(C^\infty_c(H,R)_H, R)=\{0\}$  if and only if $\Hom (\mathbb Z[1/p], R)=\{0\}$ if and only if there is no   Haar measure on $H$ with values in $R$. 
  \end{corollary}
\begin{proof} 
  $R[1/p]=\{0\}$ if and only if $p$ is nilpotent  in  $R$ by \cite[II.2 Corollary 2]{MR782296} and $R[1/p]\otimes_RW=\{0\}$ if and only if any element of $W$ is killed by  a power of $p$ ($W$ is called {\bf $p$-torsion}). \end{proof}
   The $p$-ordinary part of an $R$-module  $V$ is 
   \[
V_{p-ord}=\bigcap_{k\geq 0} p^kV.
\]
When 
 $R$ is a field, the  three conditions: $p$ nilpotent,  $R_{p-ord}= \{0\}$,  $\Hom (\mathbb Z[1/p], R)=\{0\}$, are equivalent to  $ \charf (R)= p $. 
  The equivalence of these three conditions is not true for a  general commutative ring, contrary to what is claimed in 
\cite[I (2.3.1)]{MR1395151}, \cite[\S 5]{Vigneras-adjoint}.

\begin{lemma} 1)  $p$ is nilpotent in $R$ if and only if $V_{p-ord}=\{0\}$ for all $R$-modules $V$.

2) $R_{p-ord}= \{0\}$ implies $\Hom (\mathbb Z[1/p], R)=\{0\}$. The converse is true  if $R$ is noetherian.

\end{lemma}
\begin{proof} 1) Let $n\in \mathbb N$ be the characteristic of $R$ ($n\mathbb Z$ is the kernel of the canonical map 
  $\mathbb Z \to R$). Then $p$ is nilpotent in $R$ if and only if $n=p^k$ for some $k\geq 1$.  Clearly $p^k=0$  in $R$ implies $p^kV=0$ for all $R$-modules $V$. Conversely, if $p$ is not nilpotent there exists a prime ideal $J$ of $R$ not containing $p$. The fraction field  of $R/J$    is a field $V$ of characteristic  $\charf (V)\neq  p$.

2)  For the last assertion see Lemma \ref{lemma:4.7}. 
\end{proof}

   For $W\in \Mod_R^\infty(M)$, Frobenius reciprocity gives a natural map $L_P^G \Ind_P^GW\to W$ sending the image of $f\in \Ind_P^GW$ to $f(1)$; that yields a natural transformation $L_P^G \Ind_P^G\to \Id_{ \Mod_R^\infty(M)}$.  When $ p$ is nilpotent in $R$, that natural transformation is an isomorphism of functors \cite[Theorem 5.3]{Vigneras-adjoint} (this uses Proposition \ref{prop:5.4}); by  general nonsense it follows that the natural morphism $\Id_{ \Mod_R^\infty(M)}\to R_P^G \Ind_P^G$ coming from the adjunction property is also an isomorphism of functors. We generalize these statements.
    \begin{theorem}\label{thm:5.4} When $p$ is nilpotent in $R$, the two functors $L_{P_1}^G \Ind_P^G$ and 
    $\Ind_{P\cap M_1}^{M_1}L_{P_1\cap M}^M $ from $\Mod_R^\infty(M)$ to $ \Mod_R^\infty(M_1)$ are isomorphic.
    \end{theorem}
 Before proving the theorem, we deduce a corollary:
   \begin{corollary} \label{cor:5.5}In the same situation, the two functors $R_{P_1}^G \Ind_P^G$ and 
    $\Ind_{P\cap M_1}^{M_1}R_{P_1\cap M}^M $ from $\Mod_R^\infty(M)$ to $ \Mod_R^\infty(M_1)$ are isomorphic.
   \end{corollary}  
   \begin{proof} By Theorem \ref{thm:5.4} the functors $L_{P_1}^G \Ind_P^G$ and 
    $\Ind_{P\cap M_1}^{M_1}L_{P_1\cap M}^M $ are isomorphic, so are their right adjoints  $R_P^G \Ind_{P_1}^G$ and $\Ind_{P\cap M_1}^{M_1}R_{P_1\cap M}^M $.   \end{proof} 

 In fact, our results are more precise than Theorem \ref{thm:5.4} and Corollary \ref{cor:5.5}. See Corollaries \ref{cor:5.8} and \ref{cor:5.9}.
 Our proof of Theorem \ref{thm:5.4} is inspired by the proof of the ``geometric lemma'' in \cite{MR0579172}. But \cite{MR0579172} uses complex coefficients,   also Haar measures on unipotent groups  and normalized parabolic inductions  which are not available $p$ is nilpotent in $R$. In fact,  our result  is simpler than for complex coefficients.  As will be apparent in the proof, the isomorphism comes from the natural maps 
    $L_{P_1}^G \Ind_P^G W\to
    \Ind_{P\cap M_1}^{M_1}L_{P_1\cap M}^M W$ for $W\in \Mod_R^\infty(M)$ sending the class of $f\in \Ind_P^G W$ to the function $m_1\mapsto $ image of $f(m_1)$ in $W_{N\cap M_1}$. 
    To control $L_{P_1}^G \Ind_P^G W $ we  look at $ \Ind_P^G W$ as a representation of $P_1$. The coset space $P\backslash G/P_1$ is finite and we choose a sequence $X_1, \dotsc, X_r$ of $(P,P_1)$-double cosets in $G$ such that $G=X_1\sqcup  \dotsb \sqcup X_r, X_r=PP_1$ and $X_1\sqcup  \dotsb \sqcup X_i$ is open in $G$ for $i=1,\ldots, r$. We let $I_i$ be the space of functions in $\Ind_P^G W$ with support included in $X_1\sqcup  \dots \sqcup X_i$, and put $I_0=\{0\}$.   For $i=1,\ldots, r$, restricting to $X_i$ functions in $I_i$ gives an isomorphism from $I_i/I_{i-1}$ onto the space $J_i=\ind_P^{X_i} W$ of functions $f:X_i\to W$ satisfying $f(mng)=mf(g)$ for $m\in M, n\in N,g\in X_i$,  which are locally constant and of support compact in $P\backslash X_i$. That isomorphism  is obviously compatible with the action of $P_1$ by right translations.   For $i=1,\ldots, r$, we have the exact sequence
$$0\to I_{i-1}\to I_i\to J_i \to 0$$
and by taking $N_1$-coinvariants, an exact sequence 
$$ (I_{i-1})_{N_1}\to (I_i)_{N_1}\to (J_i)_{N_1} \to 0.$$

\begin{proposition}\label{prop:5.3} Let $W\in \Mod_R^\infty(M)$.
\begin{enumerate}
\item The $R$-linear map $\ind_P^{PP_1} W\to \Ind_{P\cap M_1}^{M_1}W_{M\cap N_1}$ sending $f\in \ind_P^{PP_1} W$ to the function $m_1\mapsto$  image of $f(m_1)$ in $W_{M\cap N_1}$, gives an isomorphism of $(\ind_P^{PP_1} W)_{N_1}$ onto $\Ind_{P\cap M_1}^{M_1}W_{M\cap N_1}$ as representations of $M_1$.
\item Assume $W$ is a $p$-torsion $R$-module. The space of $N_1$-coinvariants of $\ind_P^{X_i} W$ is $0$ for $i=1,\ldots, r-1$.
\item Let $V\in \Mod_R^\infty(M_1)$ with $V_{p-ord}=0$. Then the space $\Hom_{M_1}((\ind_P^{X_i} W)_{N_1} , V)$ is $0$ for $i=1,\ldots, r-1$.
\end{enumerate}\end{proposition}

The proof of Proposition \ref{prop:5.3} is given in \S \ref{S:5.2}. Composing  the surjective map in Proposition \ref{prop:5.3} (i) with the restriction from  $\Ind_P^G W$ to $\ind_P^{PP_1} W$ we get a surjective functorial $M_1$-equivariant  homomorphism 
\begin{equation}\label{LI}
L_{P_1}^G \Ind_P^G W\to \Ind_{P\cap M_1}^{M_1}L_{P_1\cap M}^M W .
\end{equation}

\begin{corollary} \label{cor:5.8} For any  $W\in \Mod_R^\infty(M)$ which is $p$-torsion, \eqref{LI} is an isomorphism:   $$L_{P_1}^G \Ind_P^G W \simeq \Ind_{P\cap M_1}^{M_1}L_{P_1\cap M}^M W .$$
\end{corollary}
\begin{proof}   Proposition \ref{prop:5.3} (i)  shows by induction on $i$ that $(I_i)_{N_1}=0$ when $i\leq r-1$; when $i=r$ we have   $J_r=\ind_P^{PP_1} W$ and with  Proposition \ref{prop:5.3} (ii), we get   the isomorphism. 
\end{proof}
If $p$ is nilpotent in $R$, every  $W\in \Mod_R^\infty(M)$ is $p$-torsion (and conversely), and  Theorem  \ref{thm:5.4} follows from the corollary.

\bigskip Let $V\in \Mod_R^\infty(M_1)$, and any $W\in \Mod_R^\infty(M)$, the surjective homomorphism \eqref{LI} gives an injection
\begin{equation}\label{LIH}\Hom_{M_1}(\Ind_{P\cap M_1}^{M_1}L_{P_1\cap M}^M W ,V)\to \Hom_{M_1}(L_{P_1}^G \Ind_P^G W, V).
\end{equation}
Taking the right adjoints of the functors we get an injection
\begin{equation}\label{LIH?}\Hom_{M_1}(W , \Ind_{P_1\cap M}^M R_{P\cap M_1}^{M_1} V)\to \Hom_{M_1}(W,R_P^G \Ind_{P_1}^G V)
\end{equation}
which is functorial in $W$. Consequently, we have an $M$-equivariant injective homomorphism
\begin{equation}\label{LIHH}\ \Ind_{P_1\cap M}^M R_{P\cap M_1}^{M_1} V \to  R_P^G \Ind_{P_1}^G V 
\end{equation}

\begin{corollary} \label{cor:5.9} For any  $V\in \Mod_R^\infty(M_1)$ with $V_{p-ord}=0$, \eqref{LIHH} is an isomorphism:   $$\Ind_{P_1\cap M}^M R_{P\cap M_1}^{M_1} V \simeq  R_P^G \Ind_{P_1}^G V .$$
\end{corollary}
\begin{proof}   Proposition \ref{prop:5.3} (ii) and (iii)  shows that (4) is a bijection for any $W\in \Mod_R^\infty(M)$. This means that  \eqref{LIHH} is an isomorphism.
\end{proof}

Now assume that $R$ is noetherian and $V$ is admissible. If for any admissible $W\in \Mod_R^\infty(M)$, $L_{P_1\cap M}^M W$ is admissible, from  \eqref{LIH}  we get by right adjunction an injection
\begin{equation}\label{LIord?}\Hom_{M_1}(W , \Ind_{P_1\cap M}^M \Ord_{\overline P\cap M_1}^{M_1} V)\to \Hom_{M_1}(W,\Ord_{\overline P}^G \Ind_{P_1}^G V)
\end{equation}
which is functorial in admissible $W$. So, we have an $M$-equivariant injective homomorphism
\begin{equation}\label{LIHOrd}\ \Ind_{P_1\cap M}^M \Ord_{\overline P\cap M_1}^{M_1} V \to  \Ord_{\overline P}^G \Ind_{P_1}^G V. 
\end{equation}
As for Corollary \ref{cor:5.9}, we deduce:
\begin{corollary} \label{cor:5.9ord} Assume that $R$ is noetherian. Let  $V\in \Mod_R^\infty(M_1)$  be admissible with $V_{p-ord}=0$. If for any admissible $W\in \Mod_R^\infty(M)$, $L_{P_1\cap M}^M W$ is admissible, then \eqref{LIHOrd} is an isomorphism:   $$\Ind_{P_1\cap M}^M \Ord_{\overline P\cap M_1}^{M_1} V \simeq   \Ord_{\overline P}^G \Ind_{P_1}^G V 
 .$$
\end{corollary}
\begin{remark}  
1)If $P_1\supset P$, $L_{P_1\cap M}^M W=W$ so the hypothesis on $W$ is always satisfied.

2) If $p$ is nilpotent in $R$ then $R_P^G$ respects admissibility and is isomorphic to  $\Ord_{\overline P}^G$. Hence \eqref{LIHH} gives an isomorphism  $$\Ind_{P_1\cap M}^M \Ord_{\overline P\cap M_1}^{M_1} V \simeq   \Ord_{\overline P}^G \Ind_{P_1}^G V 
 .$$
\end{remark}

\subsection{Proofs}\label{S:5.2} To prove Proposition \ref{prop:5.3} (i), we control the action of $N_1$ on $\ind_P^{X_i} W$ for $i=1,\ldots, r-1$. Since $B$ contains $N_1$ we may filter $X_i$ by $(P,B)$ double cosets, exactly as we did in \S \ref{S:5.1}. Reasoning exactly as  in \S \ref{S:5.1}, it is enough to prove the following lemma.

\begin{lemma}\label{lemma:5.4}  Let  $W\in \Mod_R^\infty(M)$ and $V\in \Mod_R^\infty(M_1)$. Let   $X$ be a $(P,B)$  double coset not contained in $PP_1$.
\begin{enumerate}
\item the space of $N_1$-coinvariants of $\ind_P^{X} W$ is $0$ if  $W$ is $p$-torsion.
\item $\Hom_{M_1}((\ind_P^{X} W)_{N_1}, V)=0$ if $V_{p-ord}=0$.
\end{enumerate}\end{lemma}
 
\begin{proof}  By the Bruhat decomposition $G= B \mathcal N B$, we may assume that $X=PnB$ for some $n\in \mathcal N$, and the assumption that $X$  is not contained in $ PP_1$ means the image $w$ of $n$ in $\mathbb W= \mathcal N/Z$ does not belong to $\mathbb W_{0,M}\mathbb W_{0,M_1}$.  The map $u\mapsto Pnu:U\to P\backslash G$ is continuous and induces a bijection from $(n^{-1} P n \cap U)\backslash U$ onto $P\backslash PnB$. By Arens's theorem that bijection is an homeomorphism. The group $n^{-1} P n \cap U$ is $Z$-invariant and is equal to the product (in any order) of subgroups $U_\alpha$ for some reduced roots $\alpha$. More precisely,  
$$n^{-1} P n \cap U= \prod_{\alpha \in \Phi_{red}^+, w(\alpha)\in \Phi_M \cup \Phi_N}U_\alpha,$$
where $\Phi_N=\Phi^+-\Phi^+_M$  and  $\Phi$ is the disjoint union $ \Phi_M \sqcup \Phi_N \sqcup (-\Phi_N)$ (\S\ref{S:2.1}). 
We choose a reduced root $\beta$  such that $w(\beta)$ belongs to $-\Phi_N$ (we check the existence of $\beta$ in Lemma \ref {lemma:5.5}), and an ordering $\alpha_1, \ldots, \alpha_r$ with $\alpha_r=\beta$ of the reduced roots $\alpha \in \Phi_{red}^+$ such that $w(\alpha)\in -\Phi_N$. Let  $U'$ denote the subset $U_{\alpha_1}\times \dots\times U_{\alpha_{r-1}}$ of $U$. Then the product map 
$ (n^{-1} P n \cap U)\times U'\times U_\beta\to U$ is a bijection, indeed a  homeomorphism, so we get a homeomorphism $U'\times U_\beta \to (n^{-1} P n \cap U)\backslash U$, which moreover is $U_\beta$-equivariant for the right translation.
 All taken together we  have an  $U_\beta$-equivariant isomorphism of $R$-modules:
   $$
 f \mapsto (u',u_\beta)\mapsto f(nu'u_\beta): \ind_P^{X}W\to C_{c}^{\infty}( U'\times U_\beta,  W).$$
 Now $C_{c}^{\infty}( U'\times U_\beta,  W)$  is  
 $C_{c}^{\infty}( U',  R) \otimes_RC_{c}^{\infty}( U_\beta, R) \otimes_RW$
 where  $U_\beta$ acts only on the middle factor. By Proposition \ref{prop:5.4},  $C_{c}^{\infty}( U_\beta, R)_{ U_\beta} $ is isomorphic to $R[1/p]$. If $W$ is $p$-torsion, 
 $C_{c}^{\infty}( U_\beta, R)_{ U_\beta}  \otimes_RW =0 $  hence $(\ind_P^{PnB}(W))_{U_\beta}=0$ and a fortiori $(\ind_P^{PnB}(W))_{N_1}=0$ by transitivity of the coinvariants, since $N_1$ contains $U_\beta$.  We get (i).
Similarly, if  $V_{p-ord}=0$, $\Hom_{M_1}(C_{c}^{\infty}( U_\beta, R)_{ U_\beta} , V)=0$ hence 
 we get  (ii).
\end{proof}

\begin{lemma}\label{lemma:5.5} Let $w\in \mathbb W\setminus \mathbb W_{0,M}\mathbb W_{0,M_1}$. Then there exists $\beta \in \Phi_{N_1}$ such that $w(\beta)$ belongs to $-\Phi_N$.
\end{lemma}
We can take $\beta$ reduced. If $\beta$ is not reduced, replace it by $\beta/2$.
 
\begin{proof}  The property  in Lemma \ref {lemma:5.5} depends only on the double coset $\mathbb W_{0,M}w\mathbb W_{0,M_1}$ because $\Phi_N$ is stable by $\mathbb W_{0,M}$ and  $\Phi_{N_1}$ is stable by $\mathbb W_{0,M_1}$. We suppose that $w$ is the element of minimal length in $\mathbb W_{0,M}w\mathbb W_{0,M_1}$. This condition translates as:
 \begin{enumerate}
 \item $w^{-1}(\Phi^-)\cap \Phi^+ \subset \Phi_{N_1}$,
 \item $\Phi^-\cap w(\Phi^+) \subset -\Phi_N$.
 \end{enumerate}
Proceeding by contradiction we suppose $w(\Phi_{N_1})\subset \Phi_M \cup \Phi_N$. This implies  
 $w(\Phi_{N_1})\cap \Phi^- \subset \Phi_M^-$   then  (ii) implies $w(\Phi_{N_1})\cap \Phi^- =\emptyset$ so $w(\Phi_{N_1})\subset \Phi^+$. With  (i)  we get $\Phi^- \cap w(\Phi^+)\subset  w(\Phi_{N_1}) \subset  \Phi^+$. Then comparing with (ii), $ w(\Phi^+)\subset \Phi^+$ which implies $w=1$. This is absurd hence Lemma \ref {lemma:5.5} is proved.
  \end{proof}

 This ends the proof of  Proposition \ref{prop:5.3} (ii) and (iii). 
 To prove Proposition \ref{prop:5.3} (i),  we control  $\ind_P^{PP_1} W$ as a representation of $P_1$. As the inclusion of $P_1$ in $PP_1$ induces an homeomorphism $(P\cap P_1)\backslash P_1\to P\backslash PP_1$, we think of  $\ind_P^{PP_1} W$ as the representation $\ind_{P\cap P_1}^{P_1} W$ of $P_1$. To identify  $(\ind_{P\cap P_1}^{P_1} W)_{N_1}$  and $\Ind_{P\cap M_1}^{M_1}W_{M\cap N_1}$ we proceed exactly as in \cite[5.16 case $IV_1$]{MR0579172}; indeed mutatis mutandis we are in that case: their $G=Q$ is our $P_1$, their $M=P$ is our $P\cap P_1$, their $N$ is our $M_1$ and their $V$ our $N_1$. Their reasoning applies to get the desired result: it is enough to realize that the equivalence relation between $\ell$-sheaves on $(P\cap P_1)\backslash P_1$ and smooth representations of $P\cap P_1$ is valid for $R$ as coefficients  \cite[5.10 to 5.14]{MR0579172} and also that although $N_1$ is locally pro-$p$, forming $N_1$-coinvariants is still compatible with inductive limits  \cite[1.9 (9)]{MR0579172}. This latter property   is valid for any functor $\Mod_R^\infty (G)\to \Mod_R^\infty (M_1)$ having a right adjoint, because  $\Mod_R^\infty (G)$  is a Grothendieck  category \cite[Proposition 2.9, lemma 3.2]{Vigneras-adjoint}.

   \section{Applying adjoints of  $\Ind_{P_1}^G  $ to $I_G(P,\sigma,Q)$}\label{S:6}
    Let  us keep a general reductive connected group $G$ and a commutative ring $R$. Let $ P_1=M_1N_1$ be a standard parabolic subgroup  of   $G$  and   $(P=MN,\sigma,Q)$ an  $R[G]$-triple (\ref{S:2.2}).
     
 \subsection{Results and applications } \label{S:6.1} 
 We would like to compute 
   $L_{P_1}^G I_G(P,\sigma, Q)$ when $\sigma$ is $p$-torsion and  $R_{P_1}^G I_G(P,\sigma, Q)$ 
when  $\sigma_{p-ord}=0$. Applying Corollaries \ref{cor:5.8} and  \ref{cor:5.9}  we may reduce to the case where $P(\sigma)=G$, so 
$I_G(P,\sigma, Q) = e(\sigma)\otimes \St_Q^G$.
  But we have no direct construction of $R_{P_1}^G$. When $R$ is noetherian and $p$ is nilpotent in $R$, then for admissible $V\in \Mod_R^\infty(G)$, $R_{P_1}^GV\simeq  \Ord_{\overline P_1}^G V$ (Corollary \ref{cor:4.10}). Consequently, in the following     Theorem  \ref{thm:6.1}, Part (ii)  we may replace $ \Ord_{\overline P_1}^G$ by $R_{P_1}^G$ and $  \Ord_{M\cap \overline P_1}^M$ by $  R_{M\cap \overline P_1}^M$ when $p$ is nilpotent in $R$. 

\begin{theorem} \label{thm:6.1} Assume $P(\sigma)=G$. We have:
\begin{enumerate}
\item Assume that $\sigma$ is $p$-torsion. Then $L_{P_1}^G(e(\sigma)\otimes \St_Q^G)$ is isomorphic to $e_{M_1}(L_{M\cap P_1}^M(\sigma))\otimes \St_{M_1\cap Q}^{M_1}$ if  $ \langle Q, P _1\rangle =G$, and is $0$ otherwise.
\item Assume   $R$  noetherian,  $\sigma$  admissible, and $\sigma_{p-ord}=0$. Then 
$\Ord_{\overline P_1}^G(e(\sigma)\otimes \St_Q^G)$ is isomorphic to  $e_{M_1}(\Ord_{M\cap \overline P_1}^M(\sigma))\otimes \St_{M_1\cap Q}^{M_1}$ if $ \langle  P,P_1\rangle \supset Q$, and is $0$ otherwise.
\end{enumerate}\end{theorem}

 In part (i), the statement includes that $L_{M\cap P_1}^M(\sigma)$ extends to $M_1$ and similarly  in part (ii) for $\Ord_{M\cap \overline P_1}^M(\sigma)$. 
Before the proof  of the theorem  (\S \ref{S:6.2}, \S \ref{S:7}) we derive consequences.

Without any assumption on $P(\sigma)$, we get:

\begin{corollary} \label{cor:6.2} 
\begin{enumerate}
\item Assume that $\sigma$ is $p$-torsion. Then  $L_{P_1}^G I_G(P,\sigma, Q)$ is isomorphic to
\begin{equation}\label{tard} \Ind_{P (\sigma)\cap M_1}^{M_1}(  e_{M_1\cap M(\sigma)}(L_ {M\cap P(\sigma)}^M (\sigma))\otimes \St_{Q\cap M_1}^{M_1\cap M(\sigma) })
\end{equation} when $\langle P_1\cap P(\sigma), Q\rangle = P(\sigma)$, and  is $0$ otherwise.  
\item Assume   $R$  noetherian,  $\sigma$  admissible, and $p$ nilpotent in $R$. Then 
$\Ord_{\overline P_1}^G I_G(P,\sigma, Q)$ is isomorphic to 
\begin{equation}\label{tarde}  \Ind_{P (\sigma)\cap M_1}^{M_1}(  e_{M_1\cap M(\sigma)}(\Ord_ {M\cap \overline P(\sigma)}^M (\sigma))\otimes \St_{Q\cap M_1}^{M_1\cap M(\sigma) })
\end{equation} if $ \langle  P,P_1\cap P(\sigma)\rangle \supset Q$, and is $0$ otherwise.
\end{enumerate}\end{corollary}
In the corollary,  $L_ {M\cap P_1}^M (\sigma)$ might extend to a parabolic subgroup of $M_1$ bigger than $M_1\cap P(\sigma)$. So we cannot write \eqref{tard} as $I_{M_1}(P\cap M_1,L_ {M\cap P_1}^M (\sigma)    , Q\cap M_1)$. A similar remark applies to \eqref{tarde}.

\begin{proof} (i) $L_{P_1}^G I_G(P,\sigma, Q)=
L_{P_1}^G \Ind_{P(\sigma)}^G (e_{M(\sigma)}(\sigma)\otimes \St^{M(\sigma)}_{Q\cap M(\sigma)})$ is isomorphic  to (Corollary \ref{cor:5.8})
$ 
\Ind_{P (\sigma)\cap M_1}^{M_1} L_{P_1\cap M(\sigma)}^{M(\sigma)}e_{M(\sigma)}(\sigma)\otimes \St^{M(\sigma)}_{Q\cap M(\sigma)}$. Applying Theorem \ref{thm:6.1}, we get (i).

(ii) Similarly,  $\Ord_{\overline P_1}^G I_G(P,\sigma, Q)\simeq  \Ind_{P (\sigma)\cap M_1}^{M_1} \Ord_{M\cap \overline P_1}^{M(\sigma)}(e_{M(\sigma)}(\sigma)\otimes \St^{M(\sigma)}_{Q\cap M(\sigma)})$ by Corollary \ref{cor:5.9}. Applying Theorem \ref{thm:6.1}, we get (ii). \end{proof}

\begin{definition} A smooth $R$-representation $V$ of $G$ is called left cuspidal if $L^G_PV=0$ for all proper parabolic subgroups $P$ of $G$, and right cuspidal if $R^G_PV=0$ for all proper parabolic subgroups $P$ of $G$.
\end{definition}
We may restrict to proper standard parabolic subgroups in this definition,  since any parabolic subgroup of $G$ is conjugate to a standard one.

\begin{proposition} \label{prop:6.3} Assume that $R$ is a field of characteristic $p$. Then a supercuspidal representation is right-cuspidal.
\end{proposition}
\begin{proof} An irreducible admissible $R$-representation $V$ of $G$ such that $R_P^GV\neq 0$ is a quotient of $\Ind_P^G R^G_P V$ and by Corollary \ref{cor:4.12} is a quotient of $\Ind_P^G W$ for some irreducible admissible $R$-representation $W$ of $M$ because the characteristic of $R$ is $p$ (Corollary \ref{cor:4.12}). If $V$ is supercuspidal, then $P=G$, so $V$ is right cuspidal.
\end{proof} 
\begin{corollary}\label{cor:6.4}  Assume that $R$ is a field of characteristic $p$  and  $(P,\sigma,Q)$ is an  $R[G]$-triple with $\sigma$ supercuspidal. Then $R_{P_1}^G I_G(P,\sigma, Q)$ is isomorphic to $I_{M_1}(P\cap M_1, \sigma, Q\cap M_1)$ if $ P_1\supset Q$, and is $0$ otherwise.
\end{corollary}
This corollary implies Theorem \ref{thm:1.1} (ii).
\begin{proof} (i) Assume first $P(\sigma)=G$. As a supercuspidal representation   is $e$-minimal,
 we may apply Theorem \ref{thm:6.1} Part  (ii). Thus $R_{P_1}^G I_G(P,\sigma, Q)=0$ unless  $\langle P, P_1\rangle\supset Q$ in which case it is isomorphic to  $e_{M_1}(R_{M\cap  P_1}^M(\sigma))\otimes \St_{M_1\cap Q}^{M_1}$.

 If  $P_1$ does not contain $P$,   then $P_1\cap M$ is a proper parabolic subgroup of $M$ and by Proposition \ref{prop:6.3}, $R_{P_1\cap M}^M \sigma=0$.

If $P_1 \supset P$,  then $M\cap P_1=M$ and $R_{P_1\cap M}^M \sigma=\sigma$. Moreover,   $\langle P, P_1\rangle\supset Q$ if and only if $P_1\supset Q$. This gives the result when $P(\sigma)=G$.

(ii) Without hypothesis on $P(\sigma)$,  we proceed as in the proof of Corollary \ref{cor:6.2}.
\end{proof}
 We now turn to consequences where $R=C$.    
 
 \bigskip We have the  supersingular $C$-representations of $G$ -  we recall their definition.
Recall the homomorphism $\mathcal{S}_P^G$ in \S\ref{subsec:K,Hecke}.
A homomorphism $\chi: \mathcal Z_G(\mathcal K,V)\to C$ is  supersingular if it does not factor through $\mathcal{S}_P^G$ when $P\neq G$. 
\begin{definition}\label{def:supersingularG}A $C$-representation $\pi$ of $G$ is  called supersingular  if it is irreducible admissible and for all irreducible  smooth $C$-representations $V$ of $\mathcal K$, the eigenvalues of $\mathcal Z_G(\mathcal K,V)$ in 
$\Hom_G(\ind_K^G V,\pi)$ are supersingular. 
\end{definition}A $C$-representation $\pi$ of $G$ is  supersingular if and only if it is supercuspidal \cite[I.5 Theorem 5]{MR3600042}.
 
\begin{proposition} \label{prop:6.4}  A supersingular $C$-representation  of $G$ is left-cuspidal.
\end{proposition}
\begin{proof} 
Let $\pi$ be an  admissible $C$-representation  of $G$  and $P=MN$ be a standard parabolic subgroup of $G$  such that $L_P^G \sigma \neq 0$. Putting $W=L_P^G \pi$, adjunction gives a $G$-equivariant map $\pi\to \Ind_P^GW$.  Choose an irreducible smooth $C$-representation of the special parahoric subgroup $\mathcal K$ of $G$ such that the space $\Hom_G(\ind_{\mathcal K}^G V, \pi)$ (isomorphic to $\Hom_{\mathcal K}( V, \pi)$ and finite dimensional) is not zero.  The commutative algebra $\mathcal Z({\mathcal K}, V)$ posseses an eigenvalue on this space; that eigenvalue is also an eigenvalue 
of $\mathcal Z({\mathcal K}, V)$ on $\Hom_G(\ind_{\mathcal K}^G V, \Ind_P^GW)$  which necessarily factorizes through $\mathcal{S}_P^G$ (\S \ref{S:6.1}). 
If  $\pi$ is supersingular (in particular irreducible), $P=G$ hence $\pi$ is left cuspidal.
\end{proof}
The classification theorem \ref{thm:3.1}, Propositions \ref{prop:6.3} and  \ref{prop:6.4} imply:

\begin{corollary}\label{cor:6.5}  Assume that  $(P,\sigma,Q)$ is a  $C[G]$-triple with $\sigma$ supercuspidal. In that situation $L_{P_1}^G I_G(P,\sigma, Q)$ is isomorphic to $I_{M_1}(P\cap M_1, \sigma, Q\cap M_1)$ if $P_1\supset P$ and $\langle  P_1,Q\rangle \supset P(\sigma)$, and is $0$ otherwise.
\end{corollary}
This corollary is Theorem \ref{thm:1.1} (ii).
 \begin{proof} We proceed as for the proof of Corollary \ref{cor:6.4}. With the  same reasoning   we get $L^M_{P_1\cap M} \sigma=0$ if $P_1$ does not contain $P$ and $L^M_{P_1\cap M} \sigma=\sigma$ if $P_1\supset P$. Therefore, 
 Theorem \ref{thm:6.1} Part  (i) implies the result when $P(\sigma)=G$.
Otherwise, we use Theorem \ref{thm:5.4} to reduce to the case $P(\sigma)=G$.
   \end{proof}
 From Corollary \ref{cor:6.4} and \ref{cor:6.5} we deduce immediately:
  \begin{corollary}  An irreducible admissible $C$-representation of $G$ is  left and right cuspidal if and only if it is supercuspidal.
\end{corollary}
 
 Now  it is easy to describe the left or right cuspidal  irreducible admissible $C$-representations of $G$.
 \begin{corollary}\label{cor:6.8} Let $(P,\sigma, Q)$ be a  $C[G]$-triple with $\sigma$ supercuspidal. Then  $I_G(P,\sigma, Q)$ is  
 \begin{enumerate}
 \item left cuspidal if and only if $Q=P$ and $ P(\sigma)=G$, so $I_G(P,\sigma, Q)=e(\sigma)\otimes \St_P^G$;
 \item right cuspidal if and only if  $Q= P(\sigma)=G$, so $I_G(P,\sigma, Q)=e(\sigma)$.
  \end{enumerate}
 \end{corollary}  
  \begin{proof} (i)   By Theorem \ref{thm:1.1} Part (i), $I_G(P,\sigma, Q)$ is   left cuspidal if and only if 
  $$  \text{$\Delta_{P_1}\supset \Delta_P$  and $ \Delta_{P_1} \cup \Delta_Q \supset  \Delta_{P(\sigma)}$ implies $\Delta_{P_1}=\Delta$.}
$$
    This displayed  property    is equivalent to  $\Delta_\sigma\setminus (\Delta_Q\cap \Delta_\sigma) = \Delta \setminus \Delta_P$,  and this  is equivalent to $Q=P$ and $ P(\sigma)=G$.

(ii)   By Theorem \ref{thm:1.1} Part (ii),  $I_G(P,\sigma, Q)$ is   right  cuspidal if and only if $P_1 \supset Q$ implies $P_1=G$.
    This  latter property    is equivalent to $Q=G$. But $Q\subset P(\sigma)$ hence  $I_G(P,\sigma, Q)$ is   right  cuspidal if and only if $Q=P(\sigma)=G$.
  \end{proof}

\begin{remark}
We compare with the case where $R$ is   a field of characteristic $ \neq p$.  Then,    $L_P^G$  is exact, a subquotient of  a left cuspidal smooth $R$-representation of $G$ is also left cuspidal. 
For a representation $\pi$ of  $G$ satisfying the second adjointness property $R^G_P \pi= \delta_P L_{\overline P}^G \pi$ for all parabolic subgroups $P$ of $G$ (see \S \ref{S:4.3}), then left cuspidal is equivalent to right cuspidal. For an irreducible smooth $R$-representation (hence admissible), supercuspidal  implies obviously left and right cuspidal. The converse is true   when  $R$ is an algebraically closed field of characteristic $0$ or  banal \cite[II.3.9]{MR1395151}. When $G=GL(2,\mathbb Q_p)$ and  the characteristic $\ell$ of $C$  divides $p+1$, the smooth $C$-representation $\Ind_B^G \charone$ of $G$ admits a left and right cuspidal irreducible subquotient \cite{MR1026328}, which is not supercuspidal.
\end{remark}

 \subsection{The case of $N_1$-coinvariants}\label{S:6.2}
 We proceed to the proof of Theorem  \ref{thm:6.1}, Part (i). First we assume that 
   $\Delta_M$ is orthogonal to $ \Delta\setminus \Delta_M$. Put $M_2=M_{\Delta\setminus\Delta_M}$. Then $e(\sigma)$ is obtained by extending $\sigma$ from $M$ to $G=MM'_2$ trivially on $M'_2$. 

 \bigskip (6.2.1) Assume   $P_1 \supset P$, so that $N_1$ acts trivially on $e(\sigma)$ because $N_1 \subset M'_2$.
  We start from the exact sequence defining $\St_Q^G$ and we tensor it by $e(\sigma)$
 \begin{equation}\label{eq:st}\bigoplus_{Q' \in \mathcal Q}e(\sigma)\otimes  \Ind_{Q'}^G  \charone\to e(\sigma) \otimes  \Ind_{Q}^G\charone\to e(\sigma) \otimes  \St_Q^G \to 0,
 \end{equation}
 where $\mathcal Q$ is the set of parabolic subgroups of $G$ containing strictly $Q$. Applying the right exact functor $L_{P_1}^G$ gives an exact sequence. As $\sigma$ is $p$-torsion, Corollary \ref{cor:5.8} gives a natural isomorphism $L_{P_1}^G (e(\sigma) \otimes\Ind_{Q}^G\charone ) \simeq e_{M_1}(\sigma) \otimes \Ind_{M_1\cap Q}^{M_1}\charone$ and similarly for $Q'\in \mathcal Q$, so we get the exact sequence 
  \[\bigoplus_{Q' \in \mathcal Q}e_{M_1}(\sigma) \otimes  \Ind_{M_1\cap Q'}^{M_1}\charone \to e_{M_1}(\sigma) \otimes \Ind_{M_1\cap Q}^{M_1}\charone \to L_{P_1}^G (e(\sigma) \otimes\St_Q^G) \to 0.\]
The map on the left is given by the natural inclusion for each summand. If for some $Q' \in \mathcal Q$ we have $M_1\cap Q'=M\cap Q'$ then that map is surjective and $ L_{P_1}^G (e(\sigma) \otimes\St_Q^G)=0$. Otherwise  $ \langle Q, P _1\rangle =G$ (see the lemma below) and from the exact sequence we have an isomorphism  $L_{P_1}^G (e(\sigma) \otimes\St_Q^G)\simeq e_{M_1}(\sigma) \otimes \St_{M_1\cap Q}^{M_1}$. 

\begin{lemma} $ \langle Q, P _1\rangle =G$ if and only if $M_1\cap Q'\neq M\cap Q'$ for all $Q' \in \mathcal Q$. In this case, the map $Q'\mapsto M_1\cap Q'$ is a bijection from  $\mathcal Q$ to the set of parabolic subgroups of $M_1$ containing strictly $Q\cap M_1$.
\end{lemma}
\begin{proof} The proof is immediate after translation in terms of subsets of $\Delta$.
\end{proof}

 (6.2.2) Assume  $\langle P, P_1\rangle =G$. Then $P_1\supset P_2$,    $N_1$ is contained in $M'$ and acts trivially on $\St_Q^G$ because $\Delta_M$ and $\Delta\setminus \Delta_M$ are orthogonal. By Lemma \ref{lemma:6.8} we find that $L_{P_1}^G(e(\sigma)\otimes \St_Q^G)\simeq L_{P_1}^G e(\sigma)\otimes \St_Q^G|_{M_1}  $.  Decomposing $P_1=(P_1\cap M)M'_2=(M_1\cap M)N_1 M'_2 $ and $M_1=(M_1\cap M) M'_2 $ we see that the $R[P_1]$-module $L_{P_1}^G e(\sigma)$ is $L^M_{M\cap P_1}\sigma=\sigma_{N_1}$ trivially extended to $M'_2 $. That is $L_{P_1}^G e(\sigma)=e_{M_1}(L^M_{M\cap P_1}\sigma)$. On the other hand, because $Q\supset  M$ and $M_1\supset M_2$ we have $G=MM_2=Q M_1$ and the inclusion of $M_1$ in $G$ induces an homeomorphism $(Q\cap M_1)\backslash M_1 \simeq Q\backslash G$. So, $(\Ind_{Q}^G\charone)|_{M_1}$ identifies with $\Ind_{M_1\cap Q}^{M_1}\charone$, this also applies to the $Q'\in \mathcal Q$ containing $Q$, thus $\St_{Q}^G |_{M_1} \simeq \St_{M_1\cap Q}^{M_1}$. We get  $L_{P_1}^G(e(\sigma)\otimes \St_Q^G)\simeq e_{M_1}(L^M_{M\cap P_1}\sigma) \otimes \St_{M_1\cap Q}^{M_1}$ proving what we want when $P_1\supset M_2$, since 
$\Delta_Q \cup \Delta_{M_1}=\Delta$. Note that    the assumption that $\sigma$ is $p$-torsion was not used.

\bigskip (6.2.3) The case where $P_1$ is arbitrary can finally be obtained in two stages, using the transitivity property of the coinvariant functors: first apply $L^G_{ P_3}$ where $P_3=MP_1$ contains $P$ then apply $L^{M_3}_{M_3\cap P_1}$ where $M_3\cap P_1$ contains $M_3 \cap M_2$.  Applying  (6.2.2), $L^G_{ P_3}(e(\sigma)\otimes \St_Q^G)=0$ unless $\Delta_{P_3}\cup \Delta_Q=\Delta$ in wich case  $L_{P_3}^G\St_Q^G \simeq e_{M_3}(\sigma)\otimes \St_{M_3\cap Q}^{M_1}$. Applying (6.2.3), $L^{M_3}_{M_3\cap P_1}(e_{M_3}(\sigma)\otimes \St_{M_3\cap Q}^{M_3}) \simeq (e_{M_1}( L^M_{M\cap P_1}\sigma)\otimes \St_{M_1\cap Q}^{M_1})$.

\bigskip  This ends the proof of Theorem  \ref{thm:6.1}  (i) when  $\Delta_M$ is orthogonal to $ \Delta\setminus \Delta_M$. 

\bigskip In general, we introduce $P_{\min}=M_{\min}N_{\min}$ and an $e$-minimal representation $\sigma_{\min}$ of $M_{\min}$ as in Lemma \ref{lemma:min}, such that $ \sigma =e_P(\sigma_{\min})$.  Then $\Delta_{M_{\min}}= \Delta_{\min}$ is orthogonal to $ \Delta\setminus \Delta_{\min}$ (Lemma \ref{lemma:2.2}), and $\sigma$ is $p$-torsion so is $\sigma_{\min}$ so we can apply Theorem  \ref{thm:6.1}  (i) to $\sigma_{\min}$. As $e(\sigma)=e(\sigma_{\min})$ we get:

 $L_{P_1}^G(e(\sigma)\otimes \St_Q^G)$ is isomorphic to $e_{M_1}(L_{M_{\min}\cap P_1}^{M_{\min}}(\sigma_{\min}))\otimes \St_{M_1\cap Q}^{M_1}$ if  $ \langle Q, P _1\rangle =G$, and is $0$ otherwise.

We prove now $e_{M_1}(L_{M_{\min}\cap P_1}^{M_{\min}}(\sigma_{\min}))=  e_{M_1}(L_{M\cap P_1}^{M}(\sigma))$.   Write $J=\Delta_M \setminus \Delta_{\min}$ and $\Delta_{M_1}=\Delta_{ 1}$. 
The orthogonal decomposition $\Delta_M \cap \Delta_{ 1}= (\Delta_{\min } \cap \Delta_{ 1})\perp (J \cap \Delta_{ 1})$ implies  $M\cap M_1 = (M_{\min}\cap M_1)(M_J\cap M_1)'$. But  $(M_J\cap M_1)' \subset M'_J$   acts trivially on $\sigma$ (\S \ref{S:2.2}), so we deduce that
 $\sigma _{M\cap N_1}$ extends  $(\sigma_{\min})_{M_{\min}\cap N_1}$ and  $e_{M_1}(L_{M_{\min}\cap P_1}^{M_{\min}}(\sigma_{\min}))=e_{M_1}(L_{M\cap P_1}^{M}(\sigma))$. This ends the proof of  Theorem  \ref{thm:6.1}  (i).

  \section{Ordinary functor $\Ord_{\overline P_1}^G$}\label{S:7}
   Let  us keep a general reductive connected group $G$ and a commutative ring $R$. Let $ P_1=M_1N_1$ be a standard parabolic subgroup  of   $G$  and   $(P=MN,\sigma,Q)$    an   $R[G]$-triple  with $P(\sigma)=G$.

In this section  \S\ref{S:7},  we prove  Theorem  \ref{thm:6.1}, Part (ii) after establishing some general  results in \S\ref{S:7.1} and \S \ref{S:7.2}, with varying assumptions on $R$. As in \S \ref{S:6} for the coinvariant functor $L_P^G$, first we assume that $\sigma$ is $e$-minimal, so that
   $\Delta_M$ is orthogonal to $ \Delta\setminus \Delta_M$;  it suffices to consider two special  cases  $P_1\supset P$ (\S\ref{S:7.3}) and  $ \langle P_1, P\rangle =G$ (\S\ref{S:7.4}) and the general case is   obtained in two stages, introducing the parabolic subgroup  $ \langle P_1, P\rangle = MP_1$. When $\sigma$ is no longer assumed to be $e$-minimal, we proceed as above, using $\sigma_{\min}$.

\subsection{Haar measure and $t$-finite elements} \label{S:7.1} 

Let $H$ be a locally profinite group acting on a locally profinite topological space $X$ and on itself by left translation. 
For $x\in X$, we denote by $H_x$ the $H$-stabilizer  of $x$. The group $H$ acts on $C^\infty_c(X,R)$ by $(hf)(x)=f(h^{-1}x)$ for $h\in H, f\in C^\infty_c(X,R),x\in X$.

\begin{proposition}\label{prop:6.12} Assume that $R$ is a field and that there is a non-zero $R[H]$-linear map
$C^\infty_c(H,R)\to C^\infty_c(X,R)$. Then for some $x\in X$ there is an $R$-valued left Haar measure on  $H_x$.
\end{proposition}
\begin{proof} We show that the proposition follows from Bernstein's localization principle \cite[1.4]{MR748505}  which, we remark, is valid for an arbitrary field $R$.  

Let $C_c^\infty(H,R)\xrightarrow{\varphi} C_c^\infty(X,R)$ be a non-zero linear map. We show that there exists $x\in X$ such that $\Hom_R(C_c^\infty(H\times \{x \},R), R)\neq 0$. We view $\varphi$ as providing an integration along the fibres of the projection map $H\times X \to X $, that is,   a  non-zero linear map $C_c^\infty(H\times X,R)  \xrightarrow{\Phi}C_c^\infty(X,R)$ defined by
$$\Phi(f)(x) = \varphi(f_x)(x)$$ for $x\in X, f\in C_c^\infty(H\times X,R) $, where $f_x\in C_c^\infty(H,R)$ sends $h\in H$ to $f(h,x)$.  The dual of $\Phi$  is a non-zero linear map 
$$\Hom_R(C_c^\infty(X,R),R)  \xrightarrow{{}^t\Phi}\Hom_R(C_c^\infty(H\times X,R),R)$$ of image the space of linear functionals on $C_c^\infty(H\times X,R) $ vanishing on the kernel of $ \Phi$.

But $C_c^\infty(X,R)$ is also  an $R$-algebra for the multiplication $\psi_1\psi_2 (x)=\psi_1 (x)\psi_2 (x)$ if $\psi_1,\psi_2 \in C_c^\infty(X,R)$ and $x\in X$.
Then, $C_c^\infty(H\times X,R)  $  is naturally a $C_c^\infty(X,R)$-module: for  $\psi\in C_c^\infty(X,R)$ and  $f\in C_c^\infty(H\times X,R)  $, then  $\psi f\in C_c^\infty(H\times X,R)$ is the function $(h,x)\mapsto (\psi f)(h,x)=\psi(x)f(h,x)$. The map $\Phi$ is $C_c^\infty(X,R)$-linear:  $(\psi f)_x =\psi (x) f_x$ and $\Phi(\psi f)(x)=\varphi((\psi f)_x)(x)=\psi (x)\varphi( f_x)(x)= \psi(x)\Phi(f)(x)$.     The image    of  ${}^t\Phi$  is a  $C_c^\infty(X,R)$-submodule: for    $\psi\in C_c^\infty(X,R)$ and $L\in \Hom_R(C_c^\infty(H\times X,R),R)$ vanishing on $\Ker \Phi$, $(\psi L)(f)=L(\psi f)$.
 
 By   Bernstein's localization principle, $\Ima ({}^t\Phi)$ is the closure of the span of those functionals in  $\Ima ({}^t\Phi)$ which are supported on $H\times \{x\}$ for some $x\in X$. Consequently, as $\Ima ({}^t\Phi)\neq 0$, there exists $x\in X$ and a non-zero $L\in \Hom_R(C_c^\infty(H\times X,R),R)$ vanishing on $\Ker \Phi$ which factors through the restriction map  $C_c^\infty(H\times X,R) \xrightarrow{\res} C_c^\infty(H\times \{x\},R)$. There is a non-zero element $\mu\in \Hom_R(C_c^\infty(H\times \{x\},R), R)$ such that 
 $L=\mu \circ \res$.
 
Now assume that $\varphi$ is $H$-equivariant. We show that  $\mu $ is $H_x$-invariant.  Indeed, 
 denote by $\chi$  the characteristic function of a small open neighborhood $V$ of $x_0$. Let $f\in C_c^\infty(H,R)$. Take $ f \otimes \chi$ in $C_c^\infty(H\times X,R)$. Then $\Phi (f\otimes \chi) =\varphi(f)\chi$ whereas  $\Phi (hf \otimes \chi) =\varphi(hf)\chi= (h \varphi(f))\chi$ for  $h\in H_{x}$. We can certainly take $V$ small enough for $\varphi(f)$ and $h\varphi(f)$ to be constant  on $V$; as $hx=x$, they are equal at $x$ hence on all $V$. In particular $L (f\otimes \chi)=L (hf\otimes \chi)$ which implies that  $\mu $ is $H_{x}$-invariant.  
 
  Now, for $x\in X$,  applying Bernstein's localization principle to the natural map $H\to H_{x}\backslash H$,  the existence of  a non-zero $H_{x}$-invariant element of $\Hom_R(C_c^\infty(H\times \{x \},R), R)$ implies the existence of a $R$-valued left Haar measure on $H_{x}$.

 \end{proof}

There is a variant of Proposition \ref{prop:6.12} where   $R$ is replaced by an $R$-module $V$ with zero $p$-ordinary part.
\begin{corollary}\label{cor:6.13}   Assume that  $V$ is an $R$-module with  $\bigcap_{k\geq 0}\, p^kV=\{0\}$ and that there is a non-zero $R[H]$-linear map
$\varphi:C^\infty_c(H,R)\to C^\infty_c(X,V)$. Then for some $x\in X$ there is a $\mathbb F_p$-valued left Haar measure on  $H_x$.
\end{corollary}
\begin{proof} As $\cap_{k\geq 0}\, p^kV=\{0\}$, there exists a largest integer $k$ such that the image of $\varphi$ is contained in $p^k V$ but not in $p^{k+1}V$. The map $\varphi$ induces a non-zero $(R/pR)[H]$-linear map $ C^\infty_c(H,R/pR)\to C^\infty_c(X,p^kV/p^{k+1}V)$. By $R/pR$-linearity, it restricts to a non-zero 
$\mathbb F_p[H]$-linear map  $\varphi_p:C^\infty_c(H,\mathbb F_p)\to C^\infty_c(X,p^kV/p^{k+1}V)$. The values of the functions  in the image of $\varphi_p$ is a non-zero $\mathbb F_p$-subspace $V_p$ of $p^kV/p^{k+1}V$ and composing with a $\mathbb F_p$-linear form on $V_p$, we get a non-zero $\mathbb F_p[H]$-linear map $C^\infty_c(H,\mathbb F_p)\to C^\infty_c(X, \mathbb F_p)$. Applying Proposition \ref{prop:6.12} to $R=\mathbb F_p$, we get the desired result.
\end{proof}

In the special case $X=H$ acting on itself by left translation, all stabilizers $H_x$ are trivial, and there are non-zero $R[H]$-endomorphisms of $C^\infty_c(H,R)$, for example those given by right translations by elements of $H$. 

Consider the special situation, which appears later in the proof of the theorem, where there is an automorphism $t$ of $H$ and an open compact subgroup $H^0$ of $H$ such that  $t^k (H^0)\subset t^{k+1} (H^0)$ for $k\in \mathbb Z$, $H=\bigcup_{k\in \mathbb Z}t^k (H^0)$ and $\{0\}= \bigcap_{k\in \mathbb Z}t^k (H^0)$. Let moreover $W$ be an $R$-module with a {\bf trivial} action of $H$ and an action of $t$ via an automorphism.
Then we have a natural action of $t$ on $C^\infty_c(H,W)$ - that we identify with $C^\infty_c(H,R)\otimes W$ -   and on $\Hom_{R[H]}(C^\infty_c(H,R),C^\infty_c(H,W))$ by 
$$tf(h)=t(f(t^{-1}h)), \quad (t\varphi)(f)=t(\varphi(t^{-1}f)) ,
$$
for $h\in H, f\in C^\infty_c(H,W), \varphi\in \Hom_{R[H]}(C^\infty_c(H,R),C^\infty_c(H,W))$.

We recall that, for a monoid $A$   and  an $R[A]$-module $V$, an element $v\in V$ is {\bf $A$-finite} if the $R$-module generated by the $A$-translates  of $v$  is finitely generated. 

We say that $V$ is {\bf $A$-locally finite} if every element of $V$ is $A$-finite, If $A$ is generated by an element $t$, we say {\bf $t$-finite} instead of $A$-finite. When $R$ is noetherian,  the set $V^{A-f}$  of $A$-finite vectors in $V$  is a submodule of $V$.

If $w\in W$ is $t$-finite, then  $f\mapsto f\otimes w$ in $\Hom_{R[H]}(C^\infty_c(H,R),C^\infty_c(H,W))$ is obviously  $t$-finite. Conversely:
\begin{proposition}\label{prop:7.4} When $R$ is noetherian, any  $t$-finite element of  
$$\Hom_{R[H]}(C^\infty_c(H,R),C^\infty_c(H,W))$$ has the form $f\mapsto f\otimes w$ for some $t$-finite vector  $w\in W$.
\end{proposition}
\begin{proof} For $r\in \mathbb Z$ let $f_r\in C^\infty_c(H,R)$ be the characteristic function of $t^r(H^0)$ so that $t^k f_r=f_{k+r}$ for  $ k\in \mathbb Z$,  $hf_r$ is the characteristic function of $ht^r(H^0)$ for $h\in H$, and for $r'\geq r$, $f_{r'}=\sum _{h\in t^{r'}(H^0)/t^r(H^0)} h f_r$. Any $f\in C_c^\infty(H,R)$ is a linear combination of $H$-translates of $f_r$, $r \in \mathbb Z$.

Let $\varphi \in \Hom_{R[H]}(C^\infty_c(H,R),C^\infty_c(H,W))$. The support of $\varphi (f_0)\in C^\infty_c(H,W)$ is contained in $t^r (H^0)$ for some integer $r\geq 0$. For $r'\geq 0$, the $H$-equivariance of $\varphi$ implies that $\varphi (f_{r'})= \sum _{h\in t^{r'}(H^0)/ H^0}h \varphi (f_0)$; in particular,  $\varphi (f_r)$ has support contained in $t^r (H^0)$ and since $\varphi (f_r)$ is $t^r (H^0)$-invariant, it has the form $f_r\otimes w$ for some $w\in W$. 
For $r'\geq r$, we have similarly $\varphi (f_{r'})= \sum _{h\in t^{r'}(H^0)/ t^rH^0}h \varphi (f_r)=f_{r'}\otimes w$.
For $k\geq 0$, we compute 
\begin{equation}\label{eq:kr}(t^k\varphi )(f_{r'+k})=t^k(\varphi (t^{-k}f_{r'+k}))=t^k(\varphi  ( f_{\ell' }))=t^k(f_{r'}\otimes w)=f_{r'+k}\otimes t^kw.
\end{equation}
Assume now that $\varphi$ is $t$-finite. Then there is an integer $n\geq 1$ such that the $t^k \varphi$, $0\leq k \leq n-1$, generate the $R$-submodule $V_{\varphi}$ generated by the $t^k \varphi$, $h\in \mathbb N$, and there is a relation 
\begin{equation}\label{eq:na}t^n\varphi=a_1t^{n-1}\varphi+ \dots  +a_{n-1} t \varphi+ a_n \varphi,
\end{equation}
with $a_1, \ldots,  a_n \in R$. Applying \eqref{eq:na} to $f_{n+r}$ and using $(t^k\varphi )(f_{n+r})=f_{n+r}\otimes t^kw $  for $ 0 \leq k \leq n$
 by \eqref{eq:kr}, we get 
$$
f_{n+r}\otimes t^nw=   f_{n+r}\otimes (a_1t^{n-1}w + \dots  +a_{n-1}  t w+ a_n  w).$$
So that  $t^nw=   a_1t^{n-1}w + \dots  +a_{n-1}  t w+ a_n  w$ and $w$ is $t$-finite. 

We have already seen that $\varphi(f_{r'})=f_{r'}\otimes w$ for $r'\geq r$. Let $k\geq 1$ and assume that $\varphi(f_{r'})=f_{r'}\otimes w$ for $r'\geq k$.  Noting that $(t^i\varphi )(f_{n+k-1})=f_{n+k-1}\otimes t^iw $  for $ 0 \leq i \leq n-1$ because $n+k-1-i \geq k$, we apply \eqref{eq:na} to $f_{n+k-1}$ and we deduce
$$(t^n\varphi)(f_{n+k-1})=f_{n+k-1} \otimes (a_1t^{n-1}w+ \dots  +a_{n-1} t w+ a_n w)=f_{n+k-1} \otimes t^nw,$$
so that $t^n(\varphi(f_{k-1}))= t^n(f_{k-1} \otimes w) $ and finally $\varphi(f_{k-1})=  f_{k-1} \otimes w$.
This proves the proposition by descending induction on $k$.
\end{proof}

 We suppose now that $W$ is a {\bf free} $R$-module  with a {\bf trivial} action of $H$ and of $t$. Let $V$ be an $R[H]$-module with a compatible action of $t$. As above, we have a natural action of $t$ on $\Hom_{R[H]}(C^\infty_c(H,R),V)$ and on  $\Hom_{R[H]}(C^\infty_c(H,R),V\otimes W)$.
 
\begin{proposition}\label{prop:7.5}  When $R$ is noetherian, the natural map $\Hom_{R[H]}(C^\infty_c(H,R),V)\otimes W\to \Hom_{R[H]}(C^\infty_c(H,R),V\otimes W)$ induces an isomorphism between the submodules of $t$-finite elements.

 \end{proposition}
\begin{proof} The natural map sends $\varphi \otimes w$ to $f\mapsto \varphi(f)\otimes w$. It is an embedding because $W$ is $R$-free. It sends a $t$-finite element to a $t$-finite element because $t$ acts trivially on $W$.
Let $\varphi\in \Hom_{R[H]}(C^\infty_c(H,R),V\otimes W)$ and let $(w_i)_{i\in I}$ be an $R$-basis of $W$. For 
$f\in C^\infty_c(H,R)$ we write uniquely $\varphi(f)=\sum _{i\in I}v_i(f)\otimes w_i$ for $v_i(f)\in V$ vanishing outside some finite subset $I(f)$ of $I$. For each $i\in I$, the map $f\mapsto v_i(f)$ is $R[H]$-linear 
but it is not clear if the map vanishes outside  a finite subset of $I$. Now assume that $\varphi$ is $t$-finite. As in \eqref{eq:na}, there exists $n\geq 1$ and $a_1, \ldots, a_n \in R$ such that for each $i\in I$, 
\begin{equation}\label{eq:nai}
t^nv_i(t^{-n}f)=a_1t^{n-1}v_i(t^{-n+1}f)+ \dots  +a_{n-1} t v_i(t^{-1}f)+ a_n v_i( f).
\end{equation}
Let $I_0= I(f_0) $ be a finite subset of $I$ such that $v_i(f_0)=0$  for $i\in I\setminus I_0$. For $r\geq 0$, $v_i(f_r)=0$  for $i\in I\setminus I_0$ because $f_r$ is a sum of $H$-translates  of $f_0$. Let $k\in \mathbb Z$ and assume that for $r\geq k$, $v_i(f_r)=0$  for $i\in I\setminus I_0$.
Apply \eqref{eq:nai} to $f=f_{n+k-1}$ for $i\in I\setminus  I_0$. This gives $t^n v_i (f_{k-1})=0$ hence  $v_i (f_{k-1})=0$. As any $f\in C_c^\infty(H,R)$ is a linear combination of $H$-translates of $f_k$, $k \in \mathbb Z$, we have $v_i(f)=0$  for $i\in I\setminus I_0$ and $\varphi(f)=\sum _{i\in I_0}v_i(f)\otimes w_i$ does belong to $\Hom_{R[H]}(C^\infty_c(H,R),V)\otimes W$; each of the $v_i\in \Hom_{R[H]}(C^\infty_c(H,R),V)$  for $i\in I_0$ is $t$-finite (because $\varphi$ is $t$-finite), and that proves the proposition.
\end{proof}

   \subsection{Filtrations} \label{S:7.2}We analyze the   sequence  \eqref{eq:st} defining $\St_Q^G$,   by   filtering  $\Ind_Q^G\charone$ by subspaces of functions with support in a union of $(Q, \overline B)$ double cosets.
An important fact is that the  $(Q, \overline B)$-cosets outside  $Q\overline P_1$ do not contribute.

For convenience of references to \cite{MR3600042}, we first  consider $(Q,  B)$ double cosets - we shall switch to  $(Q, \overline B)$-cosets later.
  A $(Q,  B)$-double coset has the form $QnB$ for some $n\in \mathfrak N$; if $w$ is the image of $n$ in the finite Weyl group $\mathbb W=\mathfrak N/Z$ we write, as is customary, $QwB$ instead of $QnB$. The coset $\mathbb W_{Q}w$  is uniquely determined by $QwB$ and contains a single element of minimal length. We write ${}^Q\mathbb W$ for the set of $w\in \mathbb W$ with minimal length in $ \mathbb W_{Q} w$; they are characterized by the condition $w^{-1}(\alpha)>0$ for  $\alpha \in \Delta_Q$ \cite[2.3.3]{MR794307}. We have the disjoint union
  $$G=\bigsqcup_{w\in {}^Q\mathbb W} Qw B.$$
  By standard knowledge, for $w,w'\in {}^Q\mathbb W$, the closure of $QwB$ contains $Qw'B$ is and only if $w\geq w'$ in the Bruhat order of $W$.
As in \cite[V.7]{MR3600042}, we let  $A\subset {}^Q\mathbb W$ be a non-empty upper subset (if $a\leq w, a\in A, w\in  {}^Q\mathbb W$, then $w\in A$)    so that $QAB$ is open in $G$, and we  choose $w_A\in A$ minimal for the Bruhat order; letting $A'=A-\{w_A\}$, $QA'B$ is open in $G$ too.  Let $  \ind_Q^{QAB} \charone \subset \Ind_Q^G\charone$ be the subspace of functions with support in $QAB$,
 $$ \ind_Q^{QAB} \charone \simeq C_c^\infty (Q\backslash {QAB}, R).$$
 For  a parabolic subgroup  $Q_1$  of $G$  containing $Q $, we have $\Ind_{Q_1}^G1 \subset \Ind_{Q}^G1 $ and we let  
 $$I_{Q_1}^{QAB}=\Ind_{Q_1}^G\charone   \cap \ind_Q^{QAB} \charone .
 $$ It is the subspace of   functions with support in the union of the cosets $Q_1 x$ contained in $QAB$. We have    $I_{Q_1}^{QA'B}\subset I_{Q_1}^{QAB}$.
We also use an abbreviation $I_{Q_1,A} = I_{Q_1}^{QAB}$.
 
\begin{lemma}\label{1} For $Q_1 \supset Q$, the injective natural map $I_{Q_1}^{QAB}/I_{Q_1}^{QA'B}\to  \ind_Q^{QAB} \charone / \ind_Q^{QA'B} \charone $  is an isomorphism if $w_A \in  {}^{Q_1}\mathbb W$, and $I_{Q_1}^{QAB}=I_{Q_1}^{QA'B}$  otherwise.
\end{lemma}
\begin{proof} We write $w=w_A$. Assume first that $w\not\in  {}^{Q_1}\mathbb W$. Write $w= v w'$ with $v\in W_{Q_1,0}-\{1\} , w'\in  {}^{Q_1}\mathbb W$. We have $w'<w$  and $w$ is minimal in $A$ hence $w'\not\in A$. Let   $\varphi \in I_{Q_1,A}$.  If the support of $\varphi$ meets $ Qw B$, it meets  $  w' B$ and this is impossible because $w'\not\in A$. Thus  $\varphi \in I_{Q_1,A'}$ and $I_{Q_1,A}=I_{Q_1,A'}$ as desired.

Assume now that $w \in  {}^{Q_1}\mathbb W$ and let  $\varphi \in I_{Q,A}$. As $w\in  {}^{Q_1}\mathbb W$, the natural map $U\mapsto  Q_1\backslash Q_1w B$ induces a homeomorphism
$(w^{-1} \overline U w\cap U)\backslash U \xrightarrow{\simeq} Q_1\backslash Q_1w B$; as $w\in  {}^{Q}\mathbb W$,   the natural map $U\mapsto  Q \backslash Q w B$ induces  also a homeomorphism
$(w^{-1} \overline U w\cap U)\backslash U \xrightarrow{\simeq} Q \backslash Q w B$ \cite[V.7]{MR3600042}.  Consequently,   there is a function $\psi$ on $Q_1 wB$    left invariant under $Q_1$ and locally constant with compact support modulo $Q_1$ which has the same restriction as $\varphi$ to $Qw B$.
Set $A_{1,\geq w}\subset  {}^{Q_1}\mathbb W$ to be the upper subset of $u$ with $u\geq w $.  The  set $Q_1A_{1,\geq w} B $ is open in $G$ and  $ Q_1wB $ is closed in $Q_1A_{1,\geq w} B $.  There exists a function $\tilde \psi$ on $Q_1A_{1,\geq w} B $    left invariant under $Q_1$ and locally constant with compact support modulo $Q_1$ which is equal to  $\psi$ on $Q_1w B$. For $u\in A_{1,\geq w}$ the double coset $Q_1uB$ is the union of double cosets $Q tu B$  for $t\in W_{Q_1,0}$ with $tu\in {}^{Q}\mathbb W$; as $tu\geq u  \geq w$ we have $ tu\in A$ hence  $Q_1u B \subset QAB$ and naturally  $Q_1A_{1,\geq w} B \subset QAB $.
 Now,  we have  $\tilde \psi \in I_{Q_1, A}$,   $\tilde \psi$ and $\varphi$ have the same restriction to $QwB$, hence the same image in $I_{Q,A}/I_{Q,A'}$, and the map of the lemma is surjective.
\end{proof}

\begin{lemma}  \label{2}If  $\mathcal P$ is a set of parabolic subgroups of $G$ containing $Q$, then 
\[
\left(\sum _{Q_1\in \mathcal P} \Ind_{Q_1}^G\charone \right)\cap \ind_Q^{QAB}\charone =\sum _{Q_1\in \mathcal P} \Ind_{Q_1}^{QAB}\charone.
\]
\end{lemma}
\begin{proof} The left hand side obviously contains the right hand side.  
The  reverse inclusion is  proved as in \cite[V.16 Lemma 9]{MR3600042} by descending induction on the order of $A$. The case where $A={}^{Q}\mathbb W$ being a tautology, we assume the result for $A$ and we prove it for $A'=A-\{w_A\}$.  As  $(\sum _{Q_1\in \mathcal P} \Ind_{Q_1}^G1)\cap I_{Q,A'} $ is nothing else than $(\sum _{Q_1\in \mathcal P} I_{Q_1,A} )  \cap I_{Q,A'}$, we pick $f_{Q_1}\in  I_{Q_1,A}$ for $Q_1\in \mathcal P$ and assume that 
$\sum_{Q_1\in \mathcal P} f_{Q_1}\in I_{Q,A'}$; we want to prove that $\sum_{Q_1\in \mathcal P} f_{Q_1}\in \sum _{Q_1\in \mathcal P} I_{Q_1,A'}$.

  If  $w_A  \not\in  {}^{Q_1}\mathbb W$, $ f_{Q_1}\in  I_{Q_1,A'}$ by Lemma \ref{1}.  We are done if $w_A  \not\in  {}^{Q_1}\mathbb W$ for all $Q_1\in \mathcal P$.
 
Otherwise, $Q_1\in \mathcal P$ such that   $w_A  \in  {}^{Q_1}\mathbb W$   is contained in the  parabolic subgroup  $Q_2$ associated to $\Delta_2=\{\alpha \in \Delta,  w^{-1}(\alpha)>0\}$ and  $w_A\in {}^{Q_2}\mathbb W$; we  choose $f_{Q_2}\in I_{Q_2,A}$ such that $ f_{Q_1}- f_{Q_2}\in I_{Q_1,A'}$, that is possible by Lemma \ref{1}.   We write $\sum_{Q_1\in \mathcal P} f_{Q_1}$ as 
$$\sum_{Q_1\in \mathcal P} f_{Q_1}= \sum_{Q_1\in \mathcal P, w_A \not\in  {}^{Q_1}W} f_{Q_1}+ \sum_{Q_1\in \mathcal P,w_A \in  {}^{Q_1}W} (f_{Q_1}- f_{Q_2})+ \sum_{Q_1\in \mathcal P,w_A \in  {}^{Q_1}W} f_{Q_2}.
$$
The last term on the right  belongs also to $I_{Q,A'}$  because the other terms do, and even to $I_{Q_2,A'}$.   We have   $I_{Q_2,A'}\subset I_{Q_1,A'}$, and the last term belongs to $I_{Q_1,A'}$ for any $Q_1\in \mathcal P$  such that $w \in  {}^{Q_1}\mathbb W$. This ends the proof of the lemma.
 \end{proof}
 To express Lemmas  \ref{1}, \ref{2} in terms of $(Q, \overline B)$-double cosets we apply the remark that $QwBw_0=Qww_0 \overline B$ if $w_0$ is the longest element in $\mathbb W$, so translating by  $w_0^{-1}$ a function with support in $QAB$ gives a function with support in $QAw_0\overline B$. For a parabolic subgroup   $Q_1\subset Q$,
 $$I_{Q_1}^{QAw_0\overline B}= \Ind_{Q_1}^G\charone   \cap \ind_Q^{QAw_0\overline B} \charone$$ is the set of functions obtained in this way from $I_{Q_1 }^{QAB}$. We have $w\leq w'$ if and only if $w'w_0\geq ww_0$ for $w,w'\in \mathbb W$ \cite[Proposition~2.5.4]{MR2133266},   ${}^Q\mathbb Ww_0$ is the set of $w\in \mathbb W$ with maximal length in $ \mathbb W_{Q} w$, $Aw_0$ is a non-empty lower subset of ${}^Q\mathbb Ww_0$ and $w_A w_0$ is a maximal element of  $Aw_0$ for the Bruhat order. We  get:
 
 \begin{lemma} \label{lemma:overline1}For   $Q_1\supset Q $, the  natural map $$I_{Q_1}^{QAw_0\overline B}/I_{Q_1}^{QA'w_0\overline B}\to \ind_Q^{QAw_0\overline B} \charone/\ind_Q^{QA'w_0\overline B} \charone$$  is an isomorphism  if $w_A\in  {}^{Q_1}\mathbb W$, and  $I_{Q_1}^{QAw_0\overline B}=I_{Q_1}^{QA'w_0\overline B}$
otherwise.
\end{lemma}
\begin{lemma}  \label{overline2}If  $\mathcal P$ is a set of parabolic subgroups of $G$ containing $Q$, then  
\[
\left(\sum _{Q_1\in \mathcal P} \Ind_{Q_1}^G\charone\right)\cap  \ind_{Q }^{Q Aw_0 \overline B}\charone=\sum _{Q_1\in \mathcal P} \Ind_{Q_1} ^{Q Aw_0 \overline B} \charone.
\]
\end{lemma}
Note that $$ \ind_Q^{QAw_0\overline B} \charone/\ind_Q^{QA'w_0\overline B} \charone \simeq  \ind_Q^{Q w_A w_0 \overline B} \charone $$  as  representations of $\overline B$. 
 The image of $  \Ind_Q^{QAw_0\overline B}$ in $\St_Q^G$ is denoted by   $ \St_Q^{QAw_0\overline B}.$ 
 \begin{lemma} \label{lemma:free} The $R$-modules $ \ind_Q^{QAw_0\overline B} \charone$ and  $ \St_Q^{QAw_0\overline B}$ are free.
 \end{lemma} 
    \begin{proof} We denote $\St_Q^G=\St_Q^G(R)$ or $\St_Q^A=\St_Q^A(R)$ to indicate the coefficient ring $R$. The module $C_c^\infty (Q\backslash {QAw_0\overline B}, \mathbb Z) $ and $\St_Q^G (\mathbb Z)$ are free \cite{MR3402357} and a submodule of the  free $\mathbb Z$-module  $\St_Q^G (\mathbb Z)$ is free, hence   $\St_Q^A( \mathbb Z)$ is also free.   
The exact sequence of free modules defining $\St_Q^G(\mathbb Z)$ or $\St_Q^A(\mathbb Z)$  remains exact when we tensor by $R$. As $C_c^\infty (Q\backslash {QAw_0\overline B}, R)=   C_c^\infty (Q\backslash {QAw_0\overline B},\mathbb Z) \otimes_{\mathbb Z}  R$,  we have also $\St_Q^G \otimes_{\mathbb Z}R= \St_Q^G(R)$ and $\St_Q^A \otimes_{\mathbb Z}R= \St_Q^A(R)$. Thus, the lemma.
\end{proof}

 \begin{lemma} \label{lemma:stein} $ \St_Q^{QAw_0\overline B} = \St_Q^{QA'w_0\overline B} $ if $w_A\in {}^{Q_1}\mathbb W$ for some $Q_1\in \mathcal Q$ (notation of (6.2.1)). Otherwise the map  $ \ind_Q^{QAw_0\overline B}\charone \to \St_Q^{QAw_0\overline B}$ induces an isomorphism $$\ind_Q^{QAw_0\overline B}\charone/\ind_Q^{QA'w_0\overline B}\charone\simeq \St_Q^{QAw_0\overline B}/\St_Q^{QA'w_0\overline B}.$$
 \end{lemma} 
    \begin{proof}Set $\overline{I}_{Q_1,A} = I_{Q_1}^{QAw_0\overline{B}}$.
    If $w_A\in {}^{Q_1}w_0$ for some $Q_1\in \mathcal Q$, then by Lemma \ref{lemma:overline1}, 
    $\overline I_{Q,A}= \overline I_{Q_1,A}+ \overline I_{Q,A'}$ and taking images in $\St_Q^G$ we get $\St_Q^{A'}=\St_Q^{A}$. Otherwise, $\overline I_{Q_1,A}=\overline I_{Q_1,A'}$ for all  $Q_1\in \mathcal Q$  by Lemma \ref{lemma:overline1}. The kernel of the map $\overline I_{Q,A}\to \St_Q^A$ is $\sum_{Q_1\in \mathcal Q} \overline I_{Q_1,A}$ by Lemma \ref{overline2} and similarly for $A'$.  Hence the kernels of the maps $\overline I_{Q,A}\to \St_Q^A$ and $\overline I_{Q,A'}\to \St_Q^{A'}$ are the same, and we get the last assertion.
\end{proof}

  \begin{proposition} Assume  that $P_1$ and $Q_1$ contain $Q$   but that $P_1$ does not contain $Q_1$.  Then $\Ind_{Q_1}^G\charone  \cap \ind_Q^{Q\overline P_1}\charone =0$.
 \end{proposition} 
 \begin{proof} 
We prove that  the assumptions of the proposition imply that $Q\overline P_1$ does not contain any coset $Q_1 x$. We note that  $P_1\supset Q$ implies 
\begin{equation}\label{eq:QP1}Q\overline P_1= P_1\overline P_1= N_1 M_1 \overline N_1.
\end{equation} The inclusion $P_1\overline P_1\supset Q\overline P_1$ is obvious, and the inverse inclusion   (and the second equality) follows from $N_1\subset N_Q$ and $P_1\overline P_1=N_1\overline P_1, Q\overline P_1=N_Q\overline P_1$.
If  $Q\overline P_1$ contains a coset $Q_1 x$, we can suppose that $x= \overline p_1$ with $ \overline p_1\in \overline P_1$. We have $N_1\subset N_Q \subset  Q_1$ and  $Q_1 \overline p_1 \subset P_1\overline P_1$ implies $Q_1\subset P_1\overline P_1$, in particular $M_{Q_1}\subset P_1\overline P_1$. By that  latter inclusion, for $y\in M_{Q_1}$  there exist unique  $n_1\in N_1, m_1\in M_1, \overline n_1\in \overline N_1$ with $y=n_1m_1  \overline n_1$. For any central element $z $ of $M_{Q_1}$, we have $zyz^{-1}=y$ and by uniqueness $zn_1z^{-1}=n_1,\ zm_1z^{-1}=m_1,  \ z\overline{n}_1z^{-1}=\overline{n}_1$. But then,  $n_1\in N_{Q_1}, m_1\in M_{Q_1}, \overline{n}_1\in \overline N_{Q_1}$ and we deduce  $M_{Q_1}=(M_{Q_1}\cap N_1)(M_{Q_1}\cap M_1)(M_{Q_1}\cap \overline N_1)$; this contradicts the fact that $M_{Q_1}\cap P_1$ is a proper parabolic subgroup of $M_{Q_1}$ when $P_1 $ does not contain $Q_1$.
  \end{proof}
  \begin{corollary} \label{cor:7.15} For $P_1\supset Q$, the exact sequence \eqref{eq:st} induces   an exact sequence of $\overline P_1$-modules
  $$0\to \sum_{Q \subsetneq Q_1\subset P_1}(\ind_{Q_1}^G\charone \cap \ind_Q^{Q\overline P_1}\charone) \to \ind_Q^{Q \overline P_1}\charone  \to  \St_Q^{Q \overline P_1}\to 0.
  $$  
    \end{corollary} 

 \subsection{Case $P_1 \supset P$}\label{S:7.3}
 Assume  that $\sigma$ is $e$-minimal, hence
   $\Delta_M$ is orthogonal to $ \Delta\setminus \Delta_M$, and that  $P_1 \supset P$ in this whole section \S \ref{S:7.3}. We start the proof of the theorem \ref{thm:6.1} (ii).
 
  \begin{proposition} \label{prop:outside} Assume  $\sigma_{p-ord}=\{0\}$. When
  $w\in  \mathbb W\setminus \mathbb W_{Q} \mathbb W_{M_1}$, 
 $$\Hom_{\overline N_1}(C_c^\infty(\overline N _1,R), e(\sigma) \otimes \ind_Q^{Qw \overline B}\charone)=0$$   \end{proposition}
 Note that $w\in  \mathbb W\setminus \mathbb W_{Q} \mathbb W_{M_1}$ is equivalent to $Qw \overline B\not\subset Q\overline P_1$ and that   $\overline N_1$ acts trivially on $e(\sigma)$ because $P_1 \supset P$ as in (6.2.1).
 \begin{proof}   As $\sigma_{p-ord}=0$, Corollary \ref{cor:6.13} applied to  $H=\overline N_1$, $X=Q\backslash Qw \overline B, V $ the space of $\sigma$,  implies 
$$\Hom_{\overline N_1}(C_c^\infty(\overline N_1,R), e(\sigma) \otimes \ind_Q^{Qw \overline B} \charone )=\Hom_{\overline N_1}(C_c^\infty(\overline N_1,R), e(\sigma) \otimes C_c^\infty(Q\backslash {Qw \overline B},R )=0, $$
if the $\overline N_1$-fixator of any coset  $Qx$ contained in $Qw \overline B$ is infinite (the  infinite closed subgroups of   $\overline N_1$ being  locally pro-$p$-groups  do not admit an $\mathbb F_p$-valued   Haar measure). 
This latter property is equivalent to $Q\cap w \overline N_1 w^{-1}$ infinite, because $\overline N_1$ is normalized by  $\overline{P}_1\supset \overline{U}$. 
Indeed, $Qw\overline B= Qw\overline U$ and $Qx=Qw\overline u$ with $\overline u\in \overline U$.  For $\overline n_1 \in \overline N_1$, $Qw\overline u\overline n_1=Qw\overline u$ if and only if $\overline u\overline n_1 \overline u^{-1}$ fixes $Qw$ if and only if $\overline u\overline n_1 \overline u^{-1} \in w^{-1} Q w\cap  \overline N_1$.

When $w\in  \mathbb W\setminus \mathbb W_{Q} \mathbb W_{M_1}$, there exists  $\beta\in -\Phi_{N_1}=\Phi_{\overline N_1}$ with $w(\beta)\in \Phi_{N_Q}$  by    Lemma \ref{lemma:5.5}. The group $Q\cap w \overline N_1 w^{-1}$ is infinite because it contains  $U_{w(\beta)}$.  We get the proposition.
 \end{proof}

    \begin{corollary} \label{cor:Ord} When  $\sigma_{p-ord}=\{0\}$, we have
 \begin{align*}\Hom_{\overline N_1}(C_c^\infty(\overline N_1,R), e(\sigma) \otimes  \Ind_Q^G\charone )  = \Hom_{\overline N_1}(C_c^\infty(\overline N_1,R), e(\sigma) \otimes   \ind_Q^{Q\overline P_1}\charone ), \\\Hom_{\overline N_1}(C_c^\infty(\overline N_1,R), e(\sigma) \otimes \St_Q^G) =\Hom_{\overline N_1}(C_c^\infty(\overline N_1,R), e(\sigma) \otimes  \St_Q^{Q\overline P_1}) .
 \end{align*}
  \end{corollary}
 
\begin{proof}$Q\overline P_1$ is open in $G$ (a union of $Q$-translates of $N_1\overline P_1$) and there is a sequence of double cosets $Qw_i\overline B, \ w_i\in \mathbb W, \  i=1,\ldots, r$, disjoint form each other and not contained in $Q\overline P_1$ such that
\[X_i=Q\overline P_1 \sqcup \left(\bigsqcup_{j\leq i} Qw_j\overline B\right)\]
is open in $G$  and  $G= X_r$.
  We reason by descending induction on $i\leq r$. Consider the exact sequence of free $R$-modules (Lemma \ref{lemma:free})
 $$0\to\ind_{Q}^{X_{i-1}}\charone \to  \ind_{Q}^{X_{i}}\charone \to \ind_{Q}^{Qw_i \overline B} \charone \to 0.$$
  Tensoring by $e(\sigma)$ keeps an exact sequence, 
  and  applying $\Hom_{\overline N_1}(C_c^\infty(\overline N_1,R), - )$ we  obtain an  isomorphism    (Proposition \ref{prop:outside} and the latter functor is left exact)  
  $$\Hom_{\overline N_1}(C_c^\infty(\overline N_1,R), e(\sigma) \otimes \ind_{Q}^{X_{i-1}}\charone )\xrightarrow{\simeq}\Hom_{\overline N_1}(C_c^\infty(\overline N_1,R),e(\sigma) \otimes \ind_{Q}^{X_{i}}\charone).$$
  Composing these isomorphisms we get the first equality of the corollary.
    For the second equality, we suppose that each  $w_i$ has maximal length in the coset $\mathbb W_Qw_i$ and is maximal in $\{w_1,\dots,w_i\}$ for the Bruhat order. This is possible because
  $Q\overline P_1 = \bigcup_{w\in \mathbb W_{Q}\mathbb W_{M_1}} Qw \overline P_1 $ and $\mathbb W_{Q}\mathbb W_{M_1}$ is a lower set for the Bruhat order hence there are no $w,w' \in \mathbb W$ of maximal length in their cosets $\mathbb W_Q w, \mathbb W_Qw'$ with 
    $w\geq w'$ and $Qw\subset Q\overline P_1$  but $Qw'\not\subset Q\overline P_1$. Now, we have the exact sequence of free $R$-modules (Lemma \ref{lemma:free}), 
 $$0 \to     \St_{Q}^{X_{i-1}}\to  \St_{Q}^{X_{i}}\to Y_i\to 0$$ where $Y_i$ is either $0$ or $\ind_{Q}^{Qw_i \overline B} \charone $ by lemma \ref{lemma:stein}. Tnen proceeding  as above   for the first equality,  we get the second equality of the corollary.
 \end{proof}

     \begin{proposition}\label{prop:StOrd} Assume $R$  noetherian, $\sigma$ admissible, $\sigma_{p-ord}=0$ and $P_1\supset Q$. 
     Then
 $\Ord_{\overline P_1}^G (e(\sigma)\otimes \Ind_Q^G\charone )$ and  $\Ord_{\overline P_1}^G  (e(\sigma)\otimes \St_Q^G )$ are naturally isomorphic to   $e_{M_1}(\sigma) \otimes\Ind_{Q\cap M_1}^{M_1}\charone $ and  $e_{M_1}(\sigma) \otimes \St_{Q\cap M_1}^{M_1} $.
\end{proposition}

\begin{proof} Noting that  $Q \overline P_1=P_1\overline N_1$ because $P_1\supset Q$ and  $N_1 \subset N_Q$, the  $\overline P_1$-module $\Ind_Q^{Q\overline P_1}\charone $ identifies with
$$\ind_{Q\cap M_1}^{M_1 }\charone  \otimes C_c^\infty(\overline N_1,R)
$$
where $\overline N_1$ acts by right translation on $C_c^\infty(\overline N_1,R)$ and trivially on $\ind_{Q\cap M_1}^{M_1 }1$, whereas $M_1$ acts by conjugation on $\overline N_1$ on the second factor and right translation on the first.   If $\sigma_{p-ord}=0$, it suffices  to  recall Corollary \ref{cor:Ord} to   identify 
$\Ord_{\overline P_1}^G (e(\sigma)\otimes \Ind_Q^{G}\charone)= \Ord_{\overline P_1}^G (e(\sigma)\otimes \ind_Q^{Q \overline P_1}\charone)$ with the subspace of $Z(M_1)$-finite vectors in 
\begin{equation}\label{marre}\Hom_{R[\overline N_1]}(C_c^\infty(\overline N_1,R), e(\sigma)\otimes  \Ind_ {Q\cap M_1}^ {M_1 }\charone \otimes C_c^\infty(\overline N_1,R)).
\end{equation} 
By Remark \ref{rem:sfini} we may even take only $t$-finite vectors where $t=z^{-1}$ and $z\in Z(M)$ contracts strictly $N$ (subsection~\ref{subsec:K,Hecke}).
Put $W=e_{M_1}(\sigma)\otimes  \Ind_{M_1\cap Q}^{M_1}\charone$ and then $W\otimes \Id$ for the subspace of  \eqref{marre} made of the maps  $\varphi \mapsto f\otimes\varphi$ for $f\in W$. If $R$ is noetherian, 
$W\otimes \Id$ is $Z(M_1)$-locally finite because  $W$ is an admissible $R$-representation of $M_1$  (a vector $w\in W$ is fixed by an open compact subgroup $J$ of $M_1$ and $W^J$ is a finitely generated $R$-module, invariant by $Z(M_1)$). Hence $\Ord_{\overline P_1}^G (e(\sigma)\otimes \ind_Q^{G}\charone)$ contains $W\otimes \Id$. 
  Applying Proposition \ref{prop:7.4}  with $H=\overline N_1 $ and some suitable $t\in Z(M_1)$ we find that $W\otimes \Id$ is the space of $t$-finite vectors in \eqref{marre}.  
This provides  an isomorphism  
$$\Ord_{\overline P_1}^G (e(\sigma)\otimes \Ind_Q^G\charone  ) \simeq e_{M_1}(\sigma)\otimes \Ind_ {Q\cap M_1}^ {M_1 }\charone .$$   Similarly, for $Q \subset Q_1\subset P_1$,  
$  \Ind_{Q_1}^{Q_1\overline P_1}\charone \simeq \Ind_{Q_1\cap M_1}^{M_1 }\charone  \otimes C_c^\infty(\overline N_1,R),$ 
as $R[\overline P_1]$-modules. 

The exact sequence in Corollary \ref{cor:7.15} is made of free $R$-modules (Lemma \ref{lemma:free}) hence remains exact under tensorisation by $e(\sigma)$, we get a   $R[\overline P_1]$-isomorphism
$$e_{M_1 }(\sigma)\otimes \St_Q^{Q\overline P_1} \simeq e_{M_1 }(\sigma)\otimes  \St_{Q\cap M_1}^{M_1} \otimes C_c^\infty(\overline N_1,R)$$
As $R$ is noetherian and  $\sigma_{p-ord}=0$, $\Ord^G_{\overline P_1}(\St_Q^G)=\Ord^G_{\overline P_1}(\St_Q^{Q \overline P_1})$ identifies (Corollary \ref{cor:Ord}) with the subspace of $Z(M_1)$-finite vectors in 
$$\Hom_{R[\overline N_1]}(C_c^\infty(\overline N_1,R), e_{M_1}(\sigma)\otimes  \St_ {Q\cap M_1}^ {M_1 }  \otimes C_c^\infty(\overline N_1,R)), $$ which is made out of the maps $\varphi \mapsto f\otimes\varphi$ for $f\in \St^{M_1}_{Q\cap M_1}$  by the same reasoning as above, thus  providing an isomorphism 
$$\Ord^G_{\overline P_1}(e(\sigma)\otimes \St_Q^G)\simeq e_{M_1}(\sigma)\otimes \St^{M_1}_{Q\cap M_1}.$$
This ends the proof of the proposition. \end{proof}

  \begin{proposition} \label{prop:P_1Q} When $P_1\not\supset Q$  and  $\sigma_{p-ord}=\{0\}$,  then   \begin{align*}\Hom_{\overline N_1}(C_c^\infty(\overline N_1,R), e(\sigma) \otimes  \Ind_Q^G\charone )  =  \Hom_{\overline N_1}(C_c^\infty(\overline N_1,R), e(\sigma) \otimes \St_Q^G) =0 .
 \end{align*}
\end{proposition}
\begin{proof}As allowed by Corollary \ref{cor:Ord}, we work with 
  \begin{align*}\Hom_{\overline N_1}(C_c^\infty(\overline N_1,R), e(\sigma) \otimes   \ind_Q^{Q\overline P_1}\charone ),  \quad \Hom_{\overline N_1}(C_c^\infty(\overline N_1,R), e(\sigma) \otimes  \St_Q^{Q\overline P_1}) .
 \end{align*}
 We filter $Q\overline P_1$ by double cosets $Qw\overline B$, $w\in \mathbb W_{M_1}$, as above. We simply need the following lemma.
\end{proof}

\begin{lemma}   When $P_1\not\supset Q$  and $w\in \mathbb W_{M_1}$, then $$\Hom_{R[\overline N_1 ]}(C_c^\infty(\overline N_1,R),e(\sigma)\otimes  \ind_Q^{Qw\overline B}\charone )=0.$$
\end{lemma}
\begin{proof}  As in Proposition \ref{prop:outside}, assuming  $\sigma_{p-ord}=0$ that follows from Corollary \ref{cor:6.13} applied to  $H=\overline N_1$  and $X=Q\backslash Qw \overline B, V=e(\sigma)$ if $Q\cap w \overline N_1 w^{-1}$ is not trivial. When $w\in \mathbb W_{M_1}$, we have $\overline N_1=w \overline N_1 w^{-1}$  and  the hypothesis that $P_1$ does not contains $Q$ implies that there is $\alpha \in \Delta_Q$  not contained in $\Delta_{P_1}$. The group  $Q\cap w \overline N_1 w^{-1}= Q\cap \overline N_1 $ is not trivial because it contains $U_{-\alpha}$. We get the lemma.\end{proof}

  \begin{corollary} \label{cor:P_1Q} Assume $R$ noetherian, $\sigma$ admissible,  $\sigma_{p-ord}=\{0\}$, and $P_1\not\supset Q$. Then   
 $ \Ord_{\overline P_1}^G(e(\sigma) \otimes  \Ind_Q^G\charone )=\Ord_{\overline P_1}^G(e(\sigma) \otimes  \St_Q^G )=0$.
 \end{corollary}

\subsection{Case  $\langle P, P_1\rangle=G$} \label{S:7.4}
Assume  that $\sigma$ is $e$-minimal and that  $\langle P, P_1\rangle=G$.
\begin{proposition} \label{prop:7.19} Assume   $R$  noetherian, $\sigma$ admissible. For $X_Q^G$ equal to $ \Ind_Q^G \charone$ or  $\St_Q^G$, we have 
\begin{align*} 
\Ord_{\overline P_1}^G(e(\sigma)\otimes X_Q^G) &\simeq e_{M_1}(\Ord_{M\cap \overline P_1}^M(\sigma))\otimes X_{M_1\cap Q}^{M_1}.
\end{align*}
\end{proposition}
\begin{proof} We have $P_1\supset P_\sigma$, or equivalently  $M_1\supset M_\sigma$ and $N_1\subset N_\sigma$. As $N_1\subset M'$, $N_1$ acts trivially on $\Ind_Q^G \charone$ (hence on its quotient $\St_Q^G$) because $G=M'M_\sigma$ acts on $\Ind_Q^G \charone$ trivially on $M'$ ($\Delta_M$ and $\Delta_{\sigma}$ are orthogonal of union $\Delta$).  As $M_1\supset M_\sigma$,  $Z(M_1) $ commutes with $M_\sigma$ and acts trivially on $\St_Q^G$. We can apply Proposition \ref{prop:7.5} to $H=\overline N_1, V=e(\sigma), W=X_Q^G$ and   $t\in Z(M_1)$ strictly contracting $N_1$ (subsection~\ref{subsec:K,Hecke}), to get  isomorphisms
$$ \Ord_{\overline P_1}^G (e(\sigma)\otimes X_Q^G ) \simeq \Ord_{\overline P_1}^G (e(\sigma))\otimes X_ {Q }^G, $$
as representations of $M_1$. As  $M_1\supset M_\sigma$, the restriction to $M_1$ of $X_Q^G$  is  $X_ {Q\cap M_1}^ {M_1 } $. To prove the desired result, we need to identify $\Ord_{\overline P_1}^G (e(\sigma))$ and $e_{M_1}(\Ord_{M\cap \overline P_1}^M(\sigma))$. Put  $Y=\Hom_{R[\overline N_1 ]}(C_c^\infty(\overline N_1,R),V)$. Then   $\Ord_{\overline P_1}^G (e(\sigma))=Y^{Z(M_1)-f}$ and  $\Ord_{M\cap \overline P_1}^M(\sigma)= Y^{Z(M_1\cap M)-f}$. As $Z(M_1\cap M)\supset Z(M_1)$, a  $Z(M_1\cap M)$-finite vector is also $Z(M_1)$-finite. On the other hand,  $Z(M_1\cap M)\cap M'_\sigma$ acts trivially on $\overline N_1$ and $V$ hence on $Y$. The maximal compact subgroup $Z(M_1\cap M)^0$ of $Z(M_1\cap M)$ acts smoothly on $Y$, hence all vectors in $Y$ are $Z(M_1\cap M)^0$-finite.

\begin{lemma} \label{lemma:7.20} $Z(M_1)Z(M_1\cap M)^0(Z(M_1\cap M)\cap M'_\sigma)$ has finite index in $Z(M_1\cap M)$.
\end{lemma}
Granted that lemma, the inclusion $X^{Z(M_1)-f}\subset X^{Z(M_1\cap M)-f}$ which is obviously $M_1\cap M$-equivariant is an isomorphism. As $X^{Z(M_1)-f}$ is a representation of $M_1$ it is $e_{M_1}(X^{Z(M_1\cap M)-f})$, which is what we want to prove.

  We have $Z(M_1\cap M)^0=Z(M_1\cap M)\cap T^0$.  It suffices to prove that the image of $Z(M_1)(Z(M_1\cap M)\cap M'_\sigma)$ in $X_*(\mathbf T)$  via the map $v:Z\to X_*(T)\otimes_{\mathbb Z}  \mathbb Q$ defined in \S \ref{S:2.1},  has finite index in the image of $Z(M_1\cap M) $.  The orthogonal of  $Z(M_1\cap M)$  in $X^*(\mathbf T)\otimes_{\mathbb Z} \mathbb Q$   is contained in the orthogonal of $Z(M_1)(Z(M_1\cap M)\cap M'_\sigma)$. It suffices to show the inverse inclusion. The orthogonal of $Z(M_1)$ in $X^*(\mathbf T)\otimes_{\mathbb Z} \mathbb Q$  is  generated  by 
 $\Delta_{M_1} $. The image by $v$ of $Z(M_1\cap M)\cap M'_\sigma$  in $X_*(\mathbf T)$ containing the coroots of $\Delta_\sigma$, its orthogonal is contained in  $\Delta_{M} $. We see that the orthogonal for  $Z(M_1)(Z(M_1\cap M)\cap M'_\sigma)$   in $X^*(\mathbf T)\otimes_{\mathbb Z} \mathbb Q$ is contained in $\Delta_{M_1}\cap \Delta_{M}$. As $\Delta_{M_1\cap M}=\Delta_{M_1}\cap \Delta_{M}$  is the orthogonal of   $Z(M_1\cap M)$  in $X^*(\mathbf T)\otimes_{\mathbb Z} \mathbb Q$, the lemma is proved.
\end{proof}
This ends the proof of Proposition \ref{prop:7.19}.

\subsection{General case}

1) First we assume  that $\sigma$ is $e$-minimal. We prove  Theorem \ref{thm:6.1} (ii)   in stages, introducing the standard parabolic subgroup $P_2= \langle P_1, P\rangle$ and taking successively $\Ord^G_{\overline P_2}$ and $\Ord^{M_2}_{M_2\cap \overline P_1}$ using the transitivity of $\Ord_{\overline P_1}^G$. For $X_Q^G$ equal to $ \Ind_Q^G \charone$ or  $\St_Q^G$, we have 
\begin{align*}\Ord_{\overline P_1}^G(e(\sigma)\otimes X_Q^G )=\Ord^{M_2}_{M_2\cap \overline P_1} (e_{M_2}(\Ord ^M_{M\cap \overline P_2}(\sigma))\otimes X_Q^G )\\
= \begin{cases} e_{M_1}(\Ord^{M }_{M \cap \overline P_1}\sigma) \otimes X_{Q\cap M_1}^{M_1} &\text{if} \langle P_1, P\rangle \supset Q, \\ 
0 &\text{if} \langle P_1, P\rangle  \not\supset Q.
\end{cases}
\end{align*}

The first equality follows from Proposition \ref{prop:7.19}, and the second one from Proposition \ref{prop:StOrd}  for the first case noting that $M\subset \overline P_2$, and Corollary \ref{cor:P_1Q} for the second case. This ends the proof of 
Theorem  \ref{thm:6.1}, Part (ii) when
   $\Delta_M$ is orthogonal to $ \Delta\setminus \Delta_M$.

2) General case.   As at the end of \S \ref{S:6.2}, we introduce $P_{\min}=M_{\min}N_{\min}$ and an $e$-minimal representation $\sigma_{\min}$ of $M_{\min}$. The   case 1)  gives
\begin{align}\label{ordp}\Ord_{\overline P_1}^G(e(\sigma_{\min})\otimes X_Q^G ) 
= \begin{cases} e_{M_1}(\Ord^{M_{\min} }_{M_{\min}  \cap \overline P_1}\sigma_{\min} ) \otimes X_{Q\cap M_1}^{M_1} &\text{if} \langle P_1, P_{\min} \rangle \supset Q, \\ 
0 &\text{if} \langle P_1, P_{\min} \rangle  \not\supset Q.
\end{cases}
\end{align}
We have  $e(\sigma)=e(\sigma_{\min})$. So  we can suppress $min$ on the left hand side. We show that we can also  suppress $min$ on the right hand side.

If  $\langle P_1, P  \rangle  \not\supset Q$ then $\langle P_1, P_{\min}   \rangle   \not \supset Q$ as  $P_{\min}\subset P$, hence $\Ord_{\overline P_1}^G(e(\sigma)\otimes X_Q^G )=0$. 

If   $\langle P_1, P  \rangle  \supset Q$ but  $\langle P_1, P_{\min}   \rangle  \not\supset Q$, then  $\Ord_{\overline P_1}^G(e(\sigma)\otimes X_Q^G )=0$ and  we now prove   $\Ord^{M  }_{M   \cap \overline P_1}\sigma  =0$. Our hypothesis implies that there exists a root $\alpha \in \Delta_P$ which does not belong to 
$\Delta_{ 1}\cup \Delta_{ min }$. The root subgroup $U_{-\alpha}$ is contained in $M\cap \overline N_1$ and acts trivially on $\sigma$. Reasoning as in the proof of Proposition \ref{prop:outside}, 
$\Hom_{M\cap \overline N_1}(C_c^\infty(M\cap \overline N_1,R), \sigma)=0$ hence $\Ord^{M  }_{M   \cap \overline P_1}\sigma  =0$.

If  $\langle P_1, P_{\min}   \rangle   \supset Q$  then    $J \subset \Delta_1=\Delta_{P_1}$ where $J=\Delta_M\setminus \Delta_{\min}$. The extensions to $M_1$ of
$$\Ord^{M }_{M  \cap \overline P_1}\sigma  = (\Hom_{R[M\cap \overline N_1]}(C_c^\infty(M\cap \overline N_1, R), \sigma ))^{Z(M \cap M_1)-f}$$
(see \eqref{eq:OrdP}) and of  $\Ord^{M_{\min} }_{M_{\min}  \cap \overline P_1}\sigma_{\min}$ are equal as we show now:   

The group $M \cap \overline N_1$ is generated by the root subgroups $U_\alpha$ for $\alpha$ in $\Phi_M^+$  not in $\Phi_{ 1}$. 
Noting that  $\Phi_M \setminus \Phi_{\min}=\Phi_J$ is disjoint from $\Phi_{\min}$ and contained in  $\Phi_1=\Phi_{M_1}$,  a root $\alpha$ in $\Phi_M^+$  not in $\Phi_{ 1}$ belongs to $\Phi_{\min}^+$;   hence $ M \cap \overline N_1=  M_{\min} \cap \overline N_1$. 

The  group $ Z(M \cap M_1) $ is contained in $ Z(M_{\min} \cap M_1) $. Moreover $T\cap M'_J$ acts trivially on $\sigma$ and on $M\cap \overline N_1$ and, reasoning
as in \ref{lemma:7.20},  $ Z(M\cap M_1)( Z(M_{\min} \cap M_1)\cap M'_J )$ has finite index in $ Z(M_{\min} \cap M_1) $. Consequently taking $ Z(M_{\min} \cap M_1) $-finite
vectors or $ Z(M \cap M_1) $-finite vectors in $\Hom_{R[M\cap \overline N_1]}(C_c^\infty(M\cap \overline N_1, R), \sigma )$ gives the same answer. This finishes the proof 
of Theorem \ref{thm:6.1} (ii) .

\end{document}